\documentclass[11pt]{article}

\usepackage{amsmath,amsthm,amssymb}
\usepackage{caption,graphicx,subfig,xcolor}
\usepackage{setspace}
\usepackage[backref=none,hypertexnames=false,colorlinks=true,urlcolor=blue,linkcolor=blue,citecolor=blue]{hyperref}
\usepackage{ucs}
\usepackage[utf8x]{inputenc}
\usepackage{graphicx}
\usepackage{amsfonts}
\usepackage{dsfont}
\usepackage{amssymb}
\usepackage{amsmath}
\usepackage{amsthm}
\usepackage{enumerate}
\usepackage{stmaryrd}
\usepackage{fullpage}
\usepackage{ifthen}
\usepackage{epic}
\usepackage{authblk}
\usepackage{textcomp}
\usepackage{soul}
\usepackage[title]{appendix}

\usepackage[hypertexnames=false,colorlinks=true,linkcolor=blue,citecolor=blue]{hyperref}
\usepackage[numbers,comma,square,sort&compress]{natbib}

\makeatletter\@addtoreset{equation}{section}\makeatother

\newtheorem{thm}{Theorem}[section] 
\newtheorem{lem}[thm]{Lemma}  
\newtheorem{cor}[thm]{Corollary} 
  
\newtheorem{hyp}{Hypothesis}
\newtheorem{defn}[thm]{Definition}

\newtheoremstyle{named}{}{}{\itshape}{}{\bfseries}{.}{.5em}{\thmnote{#3's }#1}
\theoremstyle{named}

\theoremstyle{definition}
\newtheorem{rmk}{Remark}

\newcommand{\unstar}{\textbf{u}_n^*}

\newcommand{\spec}{\sigma}

\newcommand{\R}{\mathbb{R}}
\newcommand{\N}{\mathbb{N}}

\newcommand{\C}{\mathbb{C}}
\newcommand{\comment}[1]{}

\newcommand{\me}{\mathrm{e}}
\newcommand{\mbi}{\mathbf{i}}
\newcommand{\drm}{\mathrm{d}}

\title{Persistence of steady-states for dynamical systems on large networks}

\author[1]{Jason J. Bramburger}
\author[2,3]{Matt Holzer}
\author[2]{Jackson Williams} 
\affil[1]{\small Department of Mathematics and Statistics, Concordia University, Montr\'eal, QC, Canada}
\affil[2]{\small Department of Mathematical Sciences, George Mason University, Fairfax, VA, USA }
\affil[3]{\small Center for Mathematics and Artificial Intelligence (CMAI), George Mason University, Fairfax, VA, USA}

\date{}

\begin{document}

\maketitle

\begin{abstract}
The goal of this work is to identify steady-state solutions to dynamical systems defined on large, random families of networks. We do so by passing to a continuum limit where the adjacency matrix is replaced by a non-local operator with kernel called a graphon. This graphon equation is often more amenable to analysis and provides a single equation to study instead of the infinitely many variations of networks that lead to the limit.  Our work establishes a rigorous connection between steady-states of the continuum and network systems.   Precisely, we show that if the graphon equation has a steady-state solution whose linearization is invertible, there exists related steady-state solutions to the finite-dimensional networked dynamical system over all sufficiently large graphs converging to the graphon.  
The proof involves setting up a Newton--Kantorovich type iteration scheme which is shown to be a contraction on a suitable metric space. Interestingly, we show that the first iterate of our defined operator in general fails to be a contraction mapping, but the second iterate is proven to contract on the space. We extend our results to show that linear stability properties further carry over from the graphon system to the graph dynamical system. Our results are applied to twisted states in a Kuramoto model of coupled oscillators, steady-states in a model of neuronal network activity, and a Lotka--Volterra model of ecological interaction.        
\end{abstract}

\section{Introduction}\label{sec:intro}

In this paper we are concerned with steady-state solutions to nonlinear differential equations defined on networks. In particular, we study discrete reaction-diffusion-type systems of the form
\begin{equation} \label{eq:main}
    \frac{\drm u_i}{\drm t}= f(u_i)+\frac{1}{n}\sum_{j = 1}^n A_{ij}D(u_i,u_j), \qquad i = 1,\dots, n. 
 \end{equation}
Here $u_i = u_i(t)$ denotes the dynamics of the $i$th component (or agent, node, species, depending upon the application) of the system, $f(u_i)$ describes the component-specific internal kinetics, and $D(u_i,u_j)$ details the manner in which interaction between individual components affects the dynamics. Applications giving rise to systems of the form \eqref{eq:main} abound. A famous and motivating example is the Kuramoto model for synchronization of coupled oscillators \cite{kuramoto84,SyncBasin}, while other examples arise in areas such as power networks \cite{Braess1}, neuroscience \cite{amari77,ermentrout98}, biological pattern formation \cite{othmer71}, and ecology \cite{volterra39}, to name only a few.  
 
Essential to many of these examples is that the individual components may not be coupled identically, but have interaction patterns that can be described by a network. In \eqref{eq:main} such interaction networks are described by the matrix $A = [A_{ij}]_{1 \leq i,j \leq n}$. If $A_{ij}\neq 0$ then the $i$th and $j$th components interact with a strength given by the value of $A_{ij}$. Research related to \eqref{eq:main} typically requires some regularity, simplicity, or symmetry in the matrix $A$ for analytical results to be obtainable. When this structure is lacking -- as is the case for networks described by random graphs -- analysis  of differential equations like \eqref{eq:main} is challenging even at the level of computing the existence and stability of steady-state equilibrium solutions.

To better understand dynamical systems of the form \eqref{eq:main} on very large networks, i.e. $n\gg 1$, one may formally let $n \to \infty$ and attempt to analyze the resulting limiting system to gain insight into the macroscopic behavior of \eqref{eq:main}. The result of taking $n \to \infty$ in \eqref{eq:main} is non-local models taking the form
\begin{equation} \label{eq:main2}
    \frac{\partial u}{\partial t}(t,x)= f(u(t,x)) + \int_0^1 W(x,y)D(u(t,x),u(t,y))\, \drm y,  
\end{equation}
which have been widely studied to gain insight into the role that the network topology plays in the system dynamics when the network is large. Here $u(t,x)$ is a function of the (normalized) latent space $[0,1]$ and $W(x,y)$ is a symmetric, almost-everywhere continuous function that represents the probability of a connection between a node located at $x$ and one at $y$. The formulation in \eqref{eq:main2} is often employed to approximate the dynamics of \eqref{eq:main} when $A$ describes a network.  For example, $W(x,y)=p$ represents an Erd\H{o}s-R\'eyni random graph where the probability of connections between nodes is a fixed constant $p \in [0,1]$, while if $W = W(|x-y|)$ then the probability of connection between nodes is determined by the distance between the nodes in the latent space and $W$ is an abstraction of a ring network. It is often the case that the non-local equation \eqref{eq:main2} is more amenable to analysis than the discrete version \eqref{eq:main} as it provides a single deterministic model against infinitely-many random discrete systems of different sizes. Our goal in the current study is to transfer existence and stability results for steady-state solutions of the limiting non-local model \eqref{eq:main2} back down to the network models \eqref{eq:main} that are large enough to be considered close to the limit. The result of our analysis of \eqref{eq:main2} provides information regarding important dynamical features of infinitely many systems of the form \eqref{eq:main} with large $n$ and varying network topologies.  

The function $W(x,y)$ in \eqref{eq:main2} is known as a {\em graphon} and has been derived as a natural graph limit for sequences of graphs where convergence is measured with respect a metric known as the cut norm; see \cite{borgs08,borgs11,lovasz06,lovasz12} and our brief review presented in Section~\ref{sec:Graphons}. The language and tools from graphon theory are germane to our work herein since probabilistic statements regarding the convergence of sequences of random graphs to their graphon limits are now well-understood \cite{lovasz06,tropp12,vizuete2021laplacian}. Most importantly, while convergence in the cut norm can be difficult to interpret mathematically, it captures the intuitive essence of growing graph structures that is observed using pixel plots of the associated adjacency matrices, as illustrated in Figure~\ref{fig:ER}. This is in contrast to the traditional vector norms that measure element-wise differences in matrices which cannot compare adjacency matrices of different sizes nor remain small when only a single edge is added or subtracted from a very large graph.

The emergence of the non-local equation \eqref{eq:main2} featuring the graphon $W(x,y)$ as a continuum limit of \eqref{eq:main} has been a focus of research for the past decade. Notably, closeness of the solutions to \eqref{eq:main} and \eqref{eq:main2} as initial value problems was proven in \cite{medvedev14}. The goal of the present work is complementary in the sense that we will establish that stable steady-states for the non-local model \eqref{eq:main2} persist as stable steady-states in the discrete model \eqref{eq:main} if $A$ and $W(x,y)$ are close as operators in an appropriate sense to be made precise later. The use of graphons in the analysis of network dynamical systems has also grown in the past decade.  We note contributions to the study of coupled oscillators \cite{chiba19,Nonex,wrightex},  mean field games \cite{caines21,carmona22,parise23}, pattern formation \cite{bramburger2021pattern}, epidemics \cite{delmas22}, control theory \cite{gao19}, power networks \cite{kuehn19}, opinion dynamics \cite{ayi21,bonnet22} and Kuramoto models with higher order interactions \cite{bick22}.  We remark that in many of these works the focus is on the analysis of the system \eqref{eq:main2}, while connections to the discrete system \eqref{eq:main} are provided through numerical investigations. Within the context of mean field games, rigorous results connecting the existence of Nash equilibria in the graphon limit to approximate Nash equilibria in the finite-dimensional graph case have been obtained \cite{carmona22,parise23}. 

\begin{figure}[t] 
    \centering
    \includegraphics[width = \textwidth]{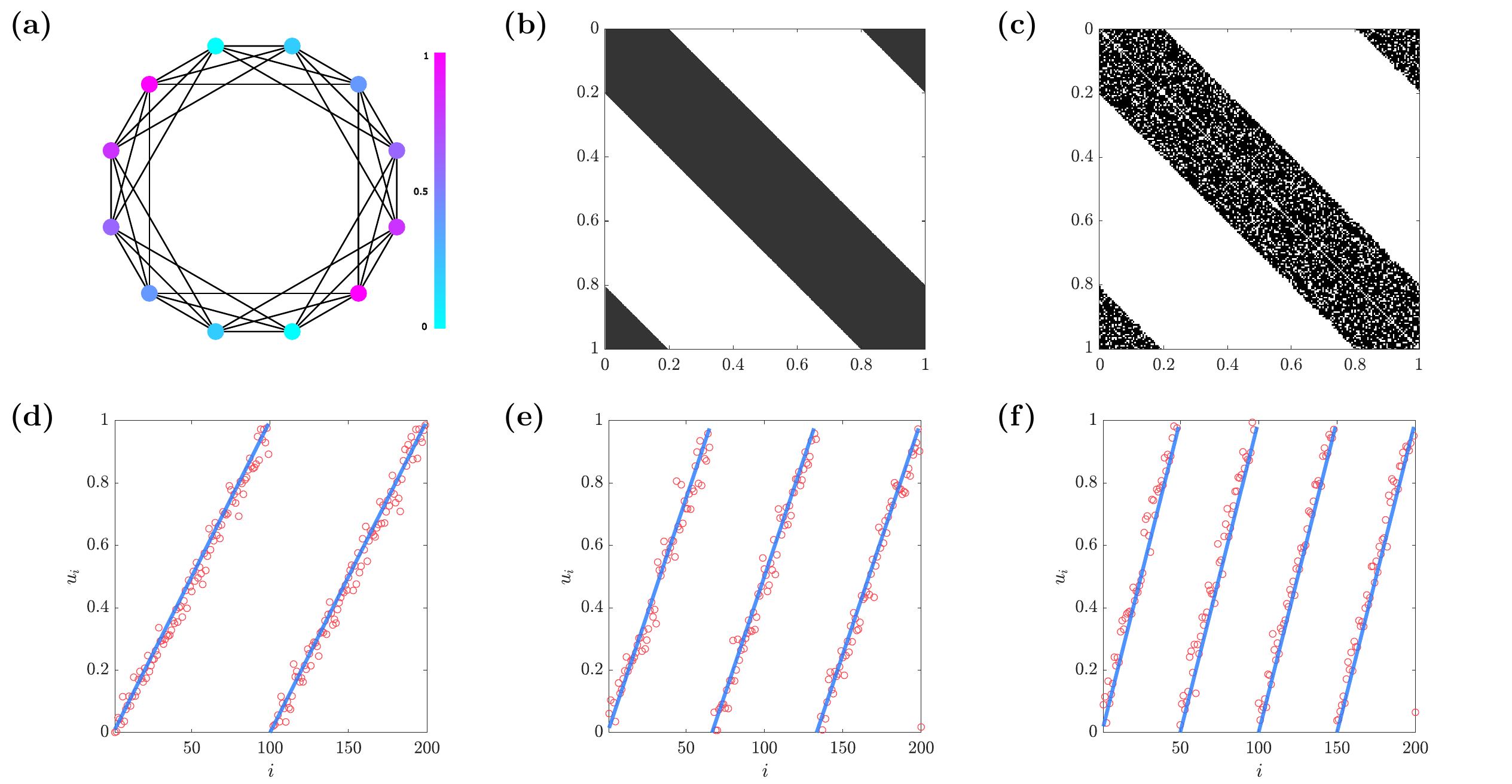}
\caption{(a) A cartoon of a 2-twisted state over a ring network with next-nearest-neighbor connections. (b) Contour plot of the small-world graphon \eqref{eq:IntroGraphon} with $\alpha = 0.2$. (c) the pixel plot of the adjacency matrix of a random network built from the graphon according to \eqref{eq:IntroProb}. Twisted states for the Kuramoto model of coupled oscillators in both the graphon (blue solid line) and random graph model (red dots) are provided with (d) $m = 2$, (e) $m = 3$, and (f) $m = 4$ twists. 
}
\label{twisted_state_graphs}
\end{figure}  

To illustrate our main results, let us consider the Kuramoto model for coupled oscillators \cite{kuramoto84}. In this case the reaction and interaction functions in \eqref{eq:main2} assume the forms  
\begin{equation}
     f(u)=0, \quad D(u(x),u(y))=\sin(2\pi(u(y)-u(x))).
\end{equation}
A special type of solution to the Kuramoto system are {\em $m$-twisted states}, taking the form $u(x,t) = m x$ (typically understood modulo 1 as $u$ represents a phase in this model) for an $m \in \mathbb{Z}$. For ring graphons, the existence and stability of twisted states in \eqref{eq:main2} was established in \cite{wrightex,SyncBasin}. Thus, a question one can ask is whether these twisted state solutions persist as solutions to the discrete system \eqref{eq:main} with a large number of vertices $n \gg 1$ and the $A_{ij}=A_{ji}\in \{ 0,1\}$ being independent random variables generated by the limiting graphon according to
\begin{equation}\label{eq:IntroProb}
    \mathbb{P}\left(A_{ij}=1\right)=W\left(\frac{i-1}{n},\frac{j-1}{n}\right).   
\end{equation}
Our work herein answers this question in the affirmative with high probability. Figure~\ref{twisted_state_graphs} provides numerical illustrations of our results using the same small-world graphon employed in the study \cite{SyncBasin}: 
\begin{equation}\label{eq:IntroGraphon}
    W(x,y) = \begin{cases}
        \frac{1}{2\pi \alpha}, & \min\{|x - y|,1 - |x - y|\} \leq \alpha, \\
        0, & \mathrm{otherwise}.
    \end{cases}
\end{equation}
Moreover, we show that stability properties are also inherited by the persisting solutions in the discrete system. One further arrives at the conclusion that these steady-states are robust with respect to large scale re-wiring of the network structure, so long as the graphs are sufficiently close to the limiting graphon.

Our proof of the persistence of solutions down from the continuum limit follows from an application of the contraction mapping theorem to a Newton-type operator. Central to this task is the identification of an appropriate Banach space on which this operator will act. The Lebesgue spaces, $L^p(\mathbb{R})$, are natural candidates but turn out to be insufficient for our purposes. For any $p \in [1,\infty)$ these spaces are not closed under pointwise multiplication, and so we cannot prove that the right-hand-side of \eqref{eq:main2} is well-defined as an operator on $L^p$ for $p < \infty$. For this reason we pivot to spaces of piece-wise continuous functions equipped with the supremum, or $L^\infty$, norm (see the definition of $X_n$ in Section~\ref{sec:proof1}).  However, our candidate operator is not, generally speaking, a contraction in these spaces.  This can be traced to the fact that the cut-norm cannot be used to control the  $L^\infty\to L^\infty$ operator norm of graphon adjacency operators $ v\mapsto \int_0^1 W(x,y)v(y)\mathrm{d}y$, which are critical to our analysis. Remarkably, it turns out that the second iterate of this operator is a contraction. This is a somewhat unusual property, but we demonstrate that it arises naturally from the definition of the cut-norm and the form of the integral operators arising in the linearization of \eqref{eq:main2} near the steady-state solution. In this way we provide unique analytical methods to arrive at our results in this manuscript.
 
This paper is organized as follows. In Section~\ref{sec:Graphons} we review graphons and describe how they can be used to generate families of discrete graphs which converge to the graphon in the cut norm as the number of vertices increases without bound. Our main results are presented in Section~\ref{sec:mainresults}. In Section~\ref{sec:proof1} we prove our first main result which states that a solution to the continuum problem defined with a graphon persists as a solution to the discrete problem on sufficiently large graphs. Then, in Section~\ref{sec:proof2}, we prove our second result which states that these nearby solutions maintain the stability of the original solution to the continuum model. Specifically, that their eigenvalues lie arbitrarily close to those of the linearization for the continuous problem when the underlying network is sufficiently large. In Section~\ref{sec:examples} we consider several examples where our theory applies. These examples include the aforementioned Kuramoto coupled oscillator model, as well as the Wilson-Cowan model of neuron activity, and Lotka-Volterra models describing ecological competition/cooperation. Finally, we conclude with Section~\ref{sec:discussion}, a discussion on future research related to our results and examples.

\section{Graphs and Graphons}\label{sec:Graphons}

In this paper we demonstrate a useful application of graphons as a tool to analyze dynamical systems on large graphs. Throughout this section we aim to build an understanding of graphons as limiting objects for sequences of graphs on $n$ vertices as $n\to \infty$. We show how such sequences can be generated both deterministically and randomly by starting with a graphon and building a sequence for which it is the limit. What follows is only a limited review of graphons, while the expository works \cite{janson2010graphons,lovasz12} are recommended for readers who would like to understand them further.

\subsection{Graphons and their norms}\label{sec:GraphonIntro}

A \emph{graphon} $W$ is a symmetric, Lebesgue-measurable function mapping $[0,1]^2$ to $[0,1]$. Boundedness of $W$ guarantees that $W \in L^p = L^p([0,1])$ for every $p \in [1,\infty]$, however the space of graphons is typically endowed with a more appropriate metric called the cut norm \cite{borgs08,borgs11,frieze99,lovasz06}, defined by  
\begin{equation}\label{cutnorm}
    \|W\|_\square = \sup_{S,T} \bigg| \int_{S\times T} W(x,y) \, \drm x \drm y \bigg|
\end{equation}
where $S$ and $T$ are measurable subsets of $[0,1]$. The cut norm is a weaker norm than the $L^p$ norms in the sense that $\|W\|_\square \leq \|W\|_p$ for all $p \in [1,\infty]$. As we will see below, the cut norm plays a critical role in graph limit theory as it can be used to interpret the geometric properties of graphs that are not captured well by the $L^p$ norms.

On a technical level, the cut norm has been shown to be equivalent to various operator norms.  For our purposes here, we will mainly use the following version
\begin{equation*}
    \|W\|_{\square,2} =\sup_{\|f\|_\infty \leq 1,\|g\|_\infty \leq 1} \bigg|\int_0^1 \int_0^1 W(x,y)f(x)g(y)\mathrm{d}y\mathrm{d} x \bigg|,
\end{equation*}
which is equivalent to the cut-norm in (\ref{cutnorm}) in the sense that $\|W\|_\square \leq \|W\|_{\square,2} \leq 4 \|W\|_\square$; see Appendix E of \cite{janson2010graphons} for reference.  We also note that the cut norm is often expressed in terms of the operator norm of $T_W:v\to \int_0^1 W(x,y)v(y)\mathrm{d}y$, which is the graphon analog of how the adjacency matrix acts on a vector. Precisely, it holds that (see \cite[Lemma E.6]{janson2010graphons}),  
\begin{equation} \label{eq:TWcutnormbounds}
    \|W\|_{\square,2} \leq \|T_W\|_{L^p\to L^q} \leq \sqrt{2} \|W\|_{\square,2}^{\mathrm{min}(1-1/p,1/q)} 
\end{equation}
for all $p,q \in [1,\infty]$. To reiterate the discussion from the introduction we see that when $p=q=\infty$, the $L^\infty\to L^\infty$ operator norm of $T_W$ cannot be controlled by the cut-norm.

\subsection{Constructing finite graphs from graphons}\label{sec:GraphonExamples}

The goal of this subsection is to review how  graphons can be used as a tool to generate both deterministic and random finite graphs. To start, fix any $n\geq 1$ and partition the interval $[0,1]$ into sub-intervals using the points $x_i=\frac{i-1}{n}$ with $i = 1,\dots,n$. Then, a graphon $W$ leads to a deterministic graph with $n$ vertices represented by the weighted adjacency matrix $A = [A_{i,j}]_{1 \leq i,j\leq n}$ with (undirected) edge weights
\begin{equation}\label{DetAdj}
    A_{i,j} = W(x_i,x_j), \qquad i,j = 1,\dots, n.
\end{equation}
Alternatively, we may construct a random graph from $W$ by assigning edges at random with the probability of connection between vertex $i$ and $j$ given by $W(x_i,x_j)$. Precisely, the associated adjacency matrix $A = [\xi_{i,j}]_{1 \leq i,j\leq n}$ has elements that are independent Bernoulli random variables $\xi_{i,j} = \xi_{j,i}$ with distribution 
\begin{equation}\label{RandomGraph}
    \mathbb{P}(\xi_{i,j} = 1) = 1 - \mathbb{P}(\xi_{i,j} = 0) = W(x_i,x_j), \qquad \forall i > j
\end{equation}
and the diagonal elements fixed as $\xi_{i,i} = 0$. Notice that the resulting random graph is not weighted as the adjacency matrix only encodes whether a connection is present or not.

While our results will apply equally to both deterministic and random graphs, it will be the latter that will be the primary focus of our applications. Several classes of popular random graph models can be described in terms of graphons. The following list enumerates some of these models.

\begin{figure}[t]
    \centering
    \includegraphics[width = \textwidth]{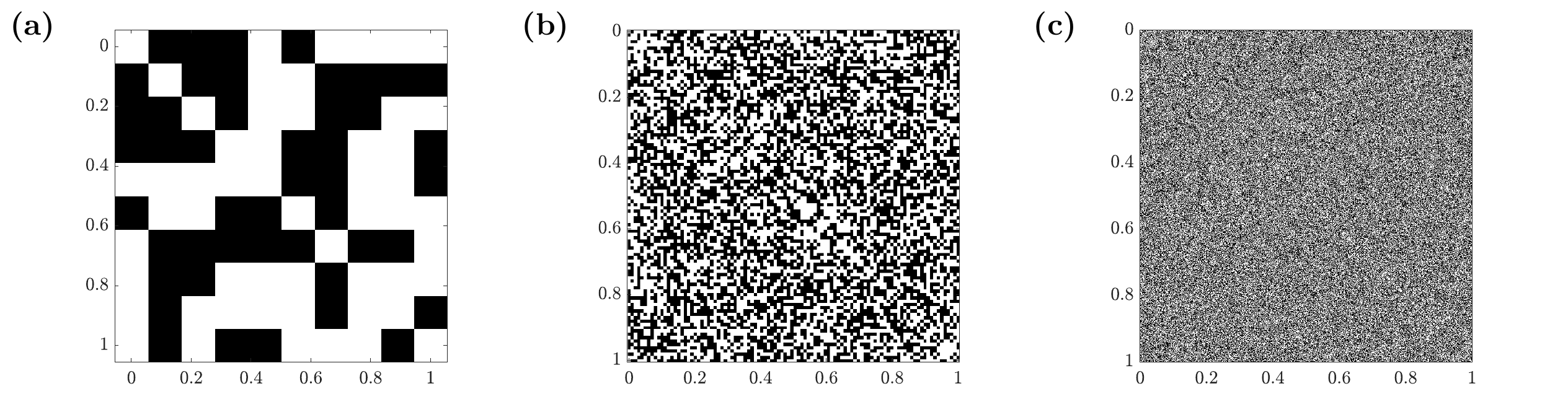}
    \caption{Pixel plots of Erd\H{o}s-R\'enyi random graphs on $n = 10, 100,$ and $1000$ vertices generated by the graphon $W(x,y) = 1/2$ with black representing an edge and white representing no edge. Visually they appear to approach a solid gray state of value 1/2, which is the correct intuition for the cut norm in this case, but not the standard $L^p$ norms.}
    \label{fig:ER}
\end{figure}
  
\begin{enumerate}
    \item {\bf Erd\H{o}s-R\'eyni} networks are perhaps the simplest random graph model, where edges are assigned independently with some fixed probability $p \in [0,1]$. Such models can be generated by a constant graphon $W(x,y)=p$. Figure~\ref{fig:ER} provides pixel plots of random graphs with $n = 10,100$, and $1000$ vertices generated using $W(x,y) = 1/2$. They visually approach a near solid gray state, providing the intuition for convergence in the cut norm, as we will discuss in the next subsection.   
    \item {\bf Ring} networks correspond to periodic arrays of nodes where edges are assigned with a probability that depends only on the distance between nodes. A ring graphon is defined by a piecewise continuous, $1$-periodic function $R:[0,1] \to [0,1]$ such that $W(x,y)=R(|x-y|)$ for all $x,y \in [0,1]$. Ring graphons have a Fourier series representation of the form 
    \begin{equation}
        W(x,y)=\sum_{k\in\mathbb{Z}} c_k \me^{2\pi \mbi k(x-y)}, \qquad c_k = c_{-k} \in \R.
    \end{equation}
    An important sub-class of ring networks are {\bf Watts-Strogatz} or {\bf small-world} networks which satisfy 
    \begin{equation}\label{SmallWorldGraphon}
    W(x,y) = \begin{cases}
        p & \text{for } |x - y|\leq \alpha \text{ or } 1-|x - y| \leq \alpha\\
        q & \text{otherwise}
    \end{cases}
    \end{equation}
    with parameters $\alpha,p,q \in [0,1]$. The corresponding Fourier series of $W$ has coefficients 
    \begin{equation}
        c_k = \begin{cases}
            2\alpha p + (1 - 2\alpha)q & k = 0,\\
            \bigg(\frac{p - q}{\pi k}\bigg)\sin(2\pi k \alpha) & k \neq 0,
        \end{cases} 
    \end{equation}
    which will be of use in our examples in Section~\ref{sec:examples}.
    \item A {\bf bipartite} network seeks to divide vertices into two groups, with probability of connections between vertices in different groups given by $p \in [0,1]$ and no intragroup connections. The corresponding graphon divides $[0,1]$ into $[0,\alpha]$ and $(\alpha,1]$, for some $\alpha \in (0,1)$, and takes the form
    \begin{equation}
        W(x,y) = \begin{cases}
            p & \mathrm{if}\ \min\{x,y\}\leq \alpha,\ \max\{x,y\} > \alpha, \\
            0 & \mathrm{otherwise}.
        \end{cases} \label{eq:bipartite}
    \end{equation}
    That is, vertices $i$ and $j$ are connected with probability $p$ if $x_i \in [0,\alpha]$ and $x_j \in (\alpha,1]$, while no edge is present when $x_i$ and $x_j$ both belong to $[0,\alpha]$ or $(\alpha,1]$. One may extrapolate and define {\bf multipartite} networks by partitioning the interval $[0,1]$ into multiple subintervals and providing probabilities of connections between distinct subintervals.  
\end{enumerate}

\subsection{Step graphons}\label{sec:StepGraphons}

\begin{figure}[t]
    \centering
    \begin{minipage}{.2\textwidth}
        \centering        \includegraphics[width=\textwidth]{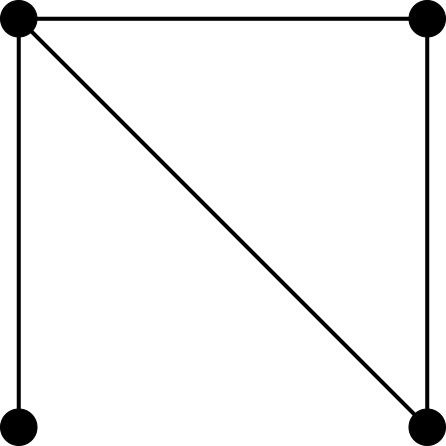}
    \end{minipage}    
    \begin{minipage}{0.32\textwidth}
        \centering
        \[\begin{pmatrix}
            0 & 1 & 1 & 1\\
            1 & 0 & 1 & 0\\
            1 & 1 & 0 & 0\\
            1 & 0 & 0 & 0
        \end{pmatrix}\]
    \end{minipage}
    \begin{minipage}{0.28\textwidth}
        \centering
        \includegraphics[width = \textwidth]{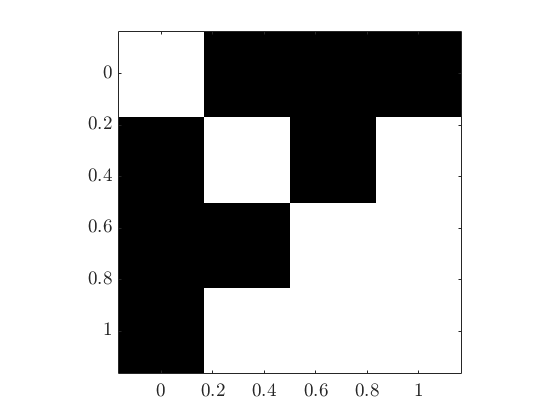}
    \end{minipage}
    \caption{A graph on 4 vertices (left) can be encoded as an adjacency matrix (center) which is used to define a step graphon (right). This step graphon representation is really a pixel plot of the adjaceny matrix, showing values of 0 in white and 1 in black.}
    \label{fig:Ex}
\end{figure}

In order to facilitate the forthcoming analysis, we need a way to compare discrete finite graphs with a continuous graphon. As we will review in this subsection, the pixel plot of an adjacency matrix naturally leads to a step-function representation of a finite graph as a graphon. This means that graphs can be represented in three equivalent ways: 1) geometrically as a collection of vertices and edges, 2) through the adjacency matrix, and 3) as a step-graphon. We refer the reader to Figure~\ref{fig:Ex} which provides these three equivalent presentations of a finite graph using a simple example. 

To be more precise, for a graph on $n$ vertices its adjacency matrix $A = [A_{i,j}]_{1 \leq i,j \leq n}$ can be used to make a graphon $W_n$ which is a step function over the domain $I_n\times I_n$ where $I_n = \{[(i-1)/n,i/n): i = 1,2,\dots,n\}$. Its values are defined as 
\begin{equation}\label{StepGraphon}
    W_n(x,y) = A_{i,j}\ \text{ for } \  (x,y)\in [(i-1)/n,i/n)\times  [(j-1)/n,j/n).
\end{equation}
By expressing finite graphs as step-graphons one is then able to compare their distances in the cut or operator norms. In particular, our main results in the following section require a sequence of step graphons converging to a given graphon in the cut norm. Let us now briefly comment on when this hypothesis can be guaranteed to hold using both the deterministic and random adjacency matrices coming from a single graphon. 

We begin with the deterministic case. For a graphon $W$, we have shown how one can generate a weighted deterministic graph on $n\geq 1$ vertices by sampling the graphon via \eqref{DetAdj}. The corresponding step graphon, here denoted $W_n^d$, is then constructed as in \eqref{StepGraphon}. Thus, $W_n^d$ is simply a step function approximation of $W$. If $W$ is almost everywhere continuous (as in all examples herein), the dominated convergence theorem gives that 
\begin{equation}
    \lim_{n \to \infty}\|W_n^d - W\|_1 \to 0.
\end{equation}
As presented in Section~\ref{sec:GraphonIntro} we have that $\|W\|_\square \leq \|W\|_1$ for all graphons $W$, and so we find that $\|W_n^d - W\|_\square \to 0$ as $n \to \infty$ for the case of deterministic weighted graphs generated by an almost everywhere continuous graphon $W$.

The case of random graphs is not nearly as straightforward, but we will again demonstrate that convergence of the step graphons to the generating graphon in the cut norm can be obtained {\em with high probability}\footnote{ {\em with high probability} means that the probability of the event occurring goes to $1$ as $n\to\infty$}. Following \eqref{RandomGraph}, we generate a random graph on $n$ vertices and denote $W_n^r$ to be its corresponding step graphon, again via \eqref{StepGraphon}. Notice now that $W_n^r$ only takes on values of $0$ or $1$, but nothing in between, and so in general we do not have pointwise convergence of $W_n^r$ to $W$ as $n \to \infty$. Nonetheless, \cite[Lemma~10.16]{lovasz12} provides the following useful result.

\begin{lem}[\cite{lovasz12}]\label{lem:RandConvergence}
    Let $n \geq 1$ and let $W$ be a graphon. Then with probability at least $1 - \mathrm{exp}(-n/\log(n))$ we have 
    \begin{equation}\label{CutConvergence}
        \|W_n^r - W\|_\square \leq \frac{22}{\sqrt{\log(n)}}.
    \end{equation}
\end{lem}

Thus we see from the above lemma that $W_n^r$ converges to $W$ in probability and so $\|W_n^r - W\|_\square \to 0$ as $n \to \infty$ can be expected to hold with high probability for random graphs. Although the cut norm can be unintuitive on first glance, Figure~\ref{fig:ER} provides random realizations of Erd\H{o}s-R\'enyi random graphs with $n = 10, 100$, and $1000$ vertices that help to visualize convergence in the cut norm. For large $n$ the blending of black (1s) and white (0s) plateaus of the step graphon begin to resemble the limiting graphon $W(x,y) = 1/2$ that would appear uniformly gray. 

Another measure of convergence of graphon sequences is in the degree of the vertices of the graph. The degree is an important centrality measure which is the sum of all edge weights incident to a vertex. In the language of graphons, the degree of a graphon is a function $d_W:[0,1] \to [0,1]$, given by
\begin{equation}
    d_W(x) = \int_0^1 W(x,y)\drm y,
\end{equation}
for all $x \in [0,1]$. Our hypotheses below will require that the degree functions of sequences of step graphons converge uniformly to the degree of a graphon, and so we briefly comment on when this can be verified in practice. We begin by providing the result \cite[Lemma~3]{vizuete2021laplacian} and then discuss its ramifications for our assumptions presented in the next section.

\begin{lem}[\cite{vizuete2021laplacian}]\label{lem:DegreeConvergence}
    Let $W$ be a graphon such that $d_W(x) \geq d_0 > 0$ for all $x \in [0,1]$. There exists an $N\geq 1$ such that for all $n \geq N$, with probability at least $1 - \nu$ we have
    \begin{equation}\label{DegreeConvergence}
        \|d_{W_n^r} - d_{W_n^d}\|_\infty \leq \sqrt{\frac{\log(2n/\nu)}{nd_0}}.
    \end{equation}
\end{lem}

Lemma~\ref{lem:DegreeConvergence} shows that the degree functions of the random step graphon and its deterministic counterpart converge uniformly with high probability so long as $W$ satisfies the minimal degree assumption $d_W(x) \geq d_0 > 0$ for all $x$. Although this result relates the degrees of the random and deterministic graphs, it does not necessarily imply uniform convergence to $d_W$ as $n \to \infty$. However, if we can show that 
\begin{equation}\label{DetDegreeConv}
    \lim_{n\to \infty} \|d_{W_n^d} - d_W\|_\infty = 0 
\end{equation}
is true, then an application of the triangle inequality with \eqref{DegreeConvergence} can show that $\|d_{W_n^r} - d_W\|_\infty \to 0$ with high probability as $n \to \infty$. There are many situations where \eqref{DetDegreeConv} can be confirmed, such as when the graphon is continuous or when $d_W(x)$ is independent of $x$; see for example  \cite{bramburger2021pattern}. Importantly, \eqref{DetDegreeConv} holds for Erd\H{o}s--R\'eyni graphons and ring graphons, as well as many other well-studied graphons in the literature.

\section{Main Results}\label{sec:mainresults}

With the introduction of graphons in the previous section, we are now in a position to provide our results. We begin by properly formulating the problem so that we can easily transition between the finite-dimensional setting of \eqref{eq:main} and the infinite-dimensional graphon equation \eqref{eq:main2}. 

\subsection{Problem setting}\label{sec:ProblemSetting}

Prior to stating our main results, we aim to fix the notation that will be used throughout this manuscript. Our eventual goal is to show that steady-state solutions to \eqref{eq:main} can be analyzed for $n \geq 1$ through the limiting infinite-dimensional graphon dynamical system \eqref{eq:main2}. Since our objective is to analyze a finite-dimensional ordinary differential equation using a non-local functional equation, we first seek to provide the appropriate definitions and terminology to move back and forth between the two settings.   

Steady-state solutions of \eqref{eq:main2} solve $F(u;W)=0$, where  
\begin{equation}\label{F}
    F(u;W): = f(u) + \int_0^1 W(x,y)D(u(x),u(y)) \drm y. 
\end{equation}
To ease notation we will simply write $F(u)$ instead of $F(u;W)$ in what follows. However, since we will consider \eqref{F} with a family of step-graphons $W_n$ converging to $W$ (see Hypothesis~\ref{hyp:Graphon} below), we will further introduce the short-hand
\begin{equation}\label{Fn}
    F_n(u) := f(u) + \int_0^1 W_n(x,y)D(u(x),u(y))\, \drm y.
\end{equation}
Clearly $F_n(u) = F(u;W_n)$ by definition.

If one restricts $u$ to lie in the set of step-functions defined over the partition $I_n$ of the interval $[0,1]$, then solving $F_n(u)=0$ is equivalent to solving a finite-dimensional problem. Indeed, one is only required to identify the value of $u$ on each of the $n$ intervals that make up the partition $I_n$. Thus, for each $n\geq 1$ and $\textbf{u} = (u_1,\dots,u_n)^T \in \R^n$, solving $F_n(u) = 0$ for a step function $u$ is equivalent to solving $G_n(\textbf{u}) = (G_n(\textbf{u})_1,\dots,G_n(\textbf{u})_n) = \mathbf{0}$, where 
\begin{equation}\label{Gn}
    G_n(\textbf{u})_i = f(u_i) + \frac{1}{n} \sum_{j = 1}^n (A_n)_{i,j} D(u_i,u_j)
\end{equation}
and $(A_n)_{i,j}$ is the value of the step graphon $W_n$ on the square $[(i-1)/n,i/n)\times[(j-1)/n,j/n)$. Notice that we have now arrived at the right-hand-side of \eqref{eq:main}, in this case derived from the graphon system \eqref{F}. Importantly, the structure of the connections in the networked dynamical system comes from the adjacency matrix $A_n$, which simply represents the $n^2$ values taken on the steps of $W_n$ whose structure is endowed by the limiting graphon $W$.

\subsection{Assumptions and Main Results}

With the problem setting and notation fixed by the previous subsection, we are now in a position to state our main results. We begin with a sequence of assumptions on the functions in the differential equation, the graphon, and on the existence of steady-states to the graphon equation \eqref{F}. We begin with the following assumption that is a standard starting point for the investigation of both \eqref{eq:main} and \eqref{eq:main2}.

\begin{hyp}\label{hyp:Smooth}
    The functions $f:\mathbb{R} \to \mathbb{R}$ and $D:\mathbb{R}\times\mathbb{R} \to \mathbb{R}$ that make up $F$ in \eqref{F} are smooth with locally Lipschitz derivatives.
\end{hyp}

Our next assumption provides that we have a sequence of graphs converging to a graphon as the size of the network grows without bound.

\begin{hyp}\label{hyp:Graphon}
    There exists a sequence of adjacency matrices $A_n \in \mathbb{R}^{n\times n}$ for all $n\geq 1$ and graphon $W$ so that the following hold:
    \begin{enumerate}
        \item The step graphons $W_n$ over $I_n\times I_n$ corresponding to each $A_n$ are such that $\|W_n - W\|_\square\to 0$ and $\|d_{W_n} - d_W\|_\infty \to 0$ as $n\to \infty$,
        \item For any $\varepsilon > 0$ there exists a $\delta > 0$ so that for every $x_0 \in [0,1]$ we have that
        \begin{equation}
            \int_0^1|W(x,y) - W(x_0,y)|\, \text{d}y <\varepsilon    
        \end{equation} 
        when $|x - x_0|<\delta$ and $x \in [0,1]$.
    \end{enumerate}
\end{hyp}

We remind the reader that the discussion in Section~\ref{sec:StepGraphons} provides scenarios for when Hypothesis~\ref{hyp:Graphon}(1) will hold. Precisely, if the $A_n$ are deterministic weighted graphs generated from an almost everywhere continuous $W$ then we have $\|W_n - W\|_\square \to 0$ as $n \to \infty$ and we need only verify the degree convergence. Similarly, if the $A_n$ are random graphs drawn from the graphon $W$ then the convergence $\|W_n - W\|_\square\to 0$ as $n\to \infty$ can be expected with high probability. The condition Hypothesis~\ref{hyp:Graphon}(2) is slightly less intuitive, but is indeed necessary for our proofs in this manuscript. If the graphon $W$ is continuous we can easily satisfy this condition using the fact that the domains $x,y \in [0,1]$ are compact and arguing from the uniform continuity of $W$. Furthermore, in the appendix we provide a proof that this condition can be shown to hold for, potentially discontinuous, ring graphons, thus broadening the class of graphons to which our work is applicable.  

We now present our final hypothesis which posits the existence of a solution to the graphon equation \eqref{F}, as well as the invertibility of the linearization about this solution.

\begin{hyp}\label{hyp:ContSol}
    There exists a continuous $u^*(x)$ satisfying $F(u^*)=0$. Furthermore, the linear operator $DF(u^*)$ is invertible on $C[0,1]$ with bounded inverse and the function
    \begin{equation}
        Q(x) := -f'(u^*(x))-\int_0^1 W(x,y) D_1(u^*(x),u^*(y)) \mathrm{d} y,   
    \end{equation} 
    satisfies $Q(x)>0$ for all $x \in [0,1]$, where $D_1$ indicates the first partial derivative of $D$ with respect to its first variable. 
\end{hyp}

In the above hypothesis the function $Q$ comes from the decomposition of the linearization $DF(u^*)$ into
\begin{equation}
    DF(u^*)v=-Q(x)v+\int_0^1 W(x,y) D_2(u^*(x),u^*(y)) v(y) \mathrm{d} y,
\end{equation}
where $D_2$ indicates the first partial derivative of the interaction function $D$ with respect to its second variable. Observe that $Q$ is the multiplication component of the linear operator. As we show in our examples below, the condition that $Q(x)>0$ can be verified in examples and, as proven in Lemma~\ref{lem:fredholm} below, reflects stability of the essential spectrum of the linearized operator $DF(u^*)$. Furthermore, Corollary~\ref{cor:Qcont} establishes that $Q(x)$ is continuous in $x$, essentially following from the assumption Hypothesis~\ref{hyp:Graphon}(2). Another piece to note is that while Hypothesis~\ref{hyp:ContSol} uses the Banach space $C[0,1]$, Lemma~\ref{lem:DF on C = DF on L2} below shows that the spectrum of $DF(u^*)$ is equivalent on $C[0,1]$ and $L^2$. This is important for applications where identifying the spectrum on $L^2$ is sometimes easier than on $C[0,1]$.

With the above Hypotheses we provide the following theorem that gives the persistence of the steady-state solution from the graphon equation to large networks with adjacency matrices converging to the graphon.

\begin{thm}\label{thm:existence} 
    Assume Hypotheses~\ref{hyp:Smooth}, \ref{hyp:Graphon}, and \ref{hyp:ContSol}. Then, there exists a constant $\rho_*>0$ such that for any $\rho\in (0,\rho_*)$
    there exists an $N\geq 1$ such that for all $n\geq N$ there is a vector $\unstar\in \mathbb{R}^n$  satisfying $G_n(\unstar;A_n) = 0$ and $\|u_n^* - u^*\|_\infty<\rho$; where $u_n^*(x)$ is the step function representation of the vector $\unstar$ over $I_n$.
\end{thm}

Our second main result concerns the stability of this steady-state. We adopt the usual convention that a linear operator is {\bf stable} if its spectrum is entirely contained in the left half of the complex plane and bounded away from the imaginary axis. The reader should recall that stability of the linearization of a finite-dimensional dynamical system about a steady-state gives local asymptotic stability, thus providing insight into the nonlinear dynamics of the network system \eqref{eq:main}.   

\begin{thm}\label{thm:stability} 
Under the same assumptions as Theorem~\ref{thm:existence}, there exists an $M \geq 1$ such that $\unstar$ is a stable equilibrium solution of (\ref{eq:main}) for all $n \geq \max\{N,M\}$ if $DF(u^*;W)$ is stable.
\end{thm}

The proof of Theorem~\ref{thm:existence} will be presented in Section~\ref{sec:proof1} while the proof of Theorem~\ref{thm:stability} will be presented in Section~\ref{sec:proof2}.

\begin{rmk}  
Our applications of the above theorems in Section~\ref{sec:examples} primarily deal with dynamics on random graphs. However, we note that our results are stated independently of these considerations and therefore can be adapted to be applied much more broadly. For example, one could consider a fixed finite graph on $n$ vertices. Minor tweaks to the proof of Theorem~\ref{thm:existence} can be made to show that steady-states can be shown to be robust with respect to perturbations of the network, such as edge addition or deletion, provided that the resulting change in the cut norm is sufficiently small.  
\end{rmk}

\section{Proof of Theorem~\ref{thm:existence}}\label{sec:proof1}

The goal of this section is to prove Theorem~\ref{thm:existence}. Precisely, we seek the existence of solutions to the finite-dimensional system $G_n(\mathbf{u},A_n)=0$, which from Section~\ref{sec:ProblemSetting} represent steady-state solutions of the original differential equation \eqref{eq:main}.  As described above, this finite-dimensional problem can be embedded into a continuous graphon problem using step-graphons which we have denoted $F_n(u)=0$, per \eqref{Fn}.  

To achieve our goal, we equivalently demonstrate the existence of a unique fixed point for the operator 
\begin{equation}\label{Tn}
    \mathcal{T}_n[u]=u-(DF(u^*))^{-1} F_n(u).
\end{equation}
For each $n \geq 1$ we will consider $\mathcal{T}_n$ as an operator on the Banach space of piecewise continuous functions $X_n$ 
\begin{equation}
    X_n=\left\{ u\in L^\infty \ \bigg| \ u(x) \ \text{is continuous on each interval} \ \left[\frac{i-1}{n},\frac{i}{n} \right)\right\},  
\end{equation}
equipped with the supremum, or $C[0,1]$, norm. Notice that for each $n \geq 1$ we have $C[0,1] \subset X_n$ and so in turn $u^* \in X_n$ for any $n$ since Hypothesis~\ref{hyp:ContSol} gives that $u^*$ is continuous.  

By construction, fixed points of $\mathcal{T}_n[u]$ in \eqref{Tn} are solutions of $F_n(u)$. The natural approach is to show that $\mathcal{T}_n$ is a contraction on the Banach space $X_n$.  However, for reasons that we will elaborate on below this cannot always be shown to be the case. Simply, this comes from the fact that $\mathcal{T}_n$ does not generally map balls of sufficiently small radius $\rho > 0$, denoted $B_\rho(u^*)\subset X_n$, back to itself. As it turns out, the operator $\mathcal{S}_n=\mathcal{T}_n\circ\mathcal{T}_n$, consisting of a two-fold application of the operator $\mathcal{T}_n$, can be shown to be a contraction on small enough balls in $X_n$ centered at $u^*$. Thus, once we are able to establish that $\mathcal{S}_n$ is a contraction for sufficiently large $n$, we obtain a unique fixed point of $\mathcal{S}_n$, which we will then show is also a fixed point of $\mathcal{T}_n$. Finally, we will demonstrate that this fixed point of $\mathcal{T}_n$, i.e. the piecewise continuous solution to $F_n(u)=0$, in $X_n$ is in fact piecewise constant so that the existence of a solution to the finite-dimensional problem is truly obtained.  

As a road map for what follows, this section has the following breakdown:
\begin{enumerate} 
    \item {\bf Section~\ref{sec:DFinv}}: We establish Fredholm properties of the operator $DF(u^*)$ and verify that $DF(u^*):X_n\to X_n$ is both well-posed and invertible for all $n \geq 1$ (Lemma~\ref{lem:fredholm} and Corollary~\ref{cor:DFinvertible}). Furthermore, expansions of the operator $DF(u^*)^{-1}$ are obtained using a Neumann type series expansions (Lemma~\ref{lem:MinvProjected}).
    
    \item {\bf Section~\ref{sec:twointegrals}}: We establish two general estimates on integrals of the form 
    \[
        \bigg| \int_0^1 \left[ W(x,y)-W_n(x,y)\right] \phi(x,y)\mathrm{d} y  \bigg|,
    \]
    and
    \[
        \bigg| \int_0^1 \int_0^1 \left[ W(z,y)-W_n(z,y)\right] \psi(x,z)\phi(z,y)v(y) \mathrm{d} y  \bigg|
    \]
for continuous functions $\phi$ and show that their supremum can be controlled by $\|d_W-d_{W_n}\|_{\infty}$ and $\|W-W_n\|_\square$, respectively (Lemmas~\ref{lem:WWnappliedtocontissmall} and \ref{lem:W-Wnappliedtwice}). 

    \item {\bf Section~\ref{sec:keyestimates}}: The results and estimates of the previous subsection are combined to obtain estimates on $DF(u^*)^{-1}F_n(u^*)$ and $DF(u^*)^{-1}[ DF(u^*)-DF_n(z)]v$ (Lemmas~\ref{lem:Fnsmall} and \ref{lem:DFdiffdef}).
    
    \item {\bf Section~\ref{sec:Tanal}}: These estimates are employed to study the action of operator $\mathcal{T}_n$ on a ball (Lemma~\ref{lem:T maps Bp to Bnp}).
    \item {\bf Section~\ref{sec:Sanal}}: In the general case, $\mathcal{T}_n$ will fail to be a contraction mapping, so we will instead show that $\mathcal{S}_n[u]=\mathcal{T}_n[\mathcal{T}_n[u]]$ is a contraction on $X_n$ (Lemma~\ref{lem:Scontraction}). 
    
    \item {\bf Section~\ref{sec:proofwrapup} }: We conclude by showing that the fixed point of the operator $\mathcal{S}_n$ implies the existence of a solution to the finite-dimensional problem $G_n(\mathbf{u},A)$ for $n$ sufficiently large.  This involves two steps: showing that the fixed point for $\mathcal{S}_n$ is also a fixed point of $\mathcal{T}_n$ thereby generating a solution of the non-local problem $F_n(u)=0$ (Lemma~\ref{lem:unfixedpointT}).  This fixed point is an element of $X_n$ and our final step is to show that it is constant on each sub-interval in $X_n$ and therefore also implies the existence of a solution of the finite-dimensional problem $G_n(\mathbf{u},A_n)=0$ (Lemma~\ref{lem:PiecewiseConstant}). 
\end{enumerate}

 Before we proceed with the proof of Theorem~\ref{thm:existence}, we briefly comment on why we cannot necessarily establish that $\mathcal{T}_n$ is a contraction, thus necessitating the use of $\mathcal{S}_n$. To apply the contraction mapping theorem to the operator $\mathcal{T}_n$ we need to control the operator norm difference of operators $DF(u^*)-DF_n(z)$ for some $z\in B_\rho(u^*)$ applied to an arbitrary vector $v \in B_\rho(u^*)$, for some $\rho > 0$.  It turns out that almost all terms in this difference can be controlled by making $\rho$ small or taking $n$ large. The exception is a term that takes (in the simplest case) the following form
\begin{equation} \label{eq:problemterm}
    \int_0^1 \left(W(x,y)-W_n(x,y)\right)v(y)\mathrm{d} y.
\end{equation}
To see why we cannot necessarily make this term small in the supremum norm by taking $n$ large, let  $x\in [0,1]$ and consider the piecewise continuous function $v(y)=\rho\ \mathrm{sign} [W(x,y)-W_n(x,y)]$.  Then (\ref{eq:problemterm}) will be 
\[ \rho \int_0^1 \left| W(x,y)-W_n(x,y)\right|\mathrm{d}y, \]
and, for the applications that we are interested in, this quantity will not typically tend to zero when $n\to \infty$.   

Alternatively, in the operator $\mathcal{S}_n[u]$ we show that the dominant term takes the form (presented again in the simplest case for the sake of exposition) of the double integral
\begin{equation}
\int_0^1 \left(W(x,z)-W_n(x,z)\right) \int_0^1 \left(W(z,y)-W_n(z,y)\right)v(y)\mathrm{d} y\mathrm{d} z,
\end{equation}
whose supremum norm with respect to $x$ can be controlled by the cut-norm difference $\|W-W_n\|_\square$, which by our assumptions can be made arbitrarily small for $n$ sufficiently large.

\subsection{Preliminary Facts}

Before proceeding to an analysis of the fixed point operators $\mathcal{T}_n$ and $\mathcal{S}_n$ we will need to compile some facts to have at our disposal.

\subsubsection{Properties of $DF(u^*)$ and its inverse}\label{sec:DFinv}

In this section we consider the linear operator $DF(u^*)$.  The goal of this section is two-fold.  We need to verify  that the mapping $\mathcal{T}_n:X_n\to X_n$ is well defined for $n \geq 1$.  This requires that $DF(u^*)$ is invertible on the space $X_n$.  Additionally, the remaining analysis will require estimates on and expansions of the operator $DF(u^*)^{-1}$.

Recall from Hypothesis~\ref{hyp:ContSol} that $DF(u^*)$ is assumed to be invertible on $C[0,1]$. We will show that this operator is also invertible on the larger space $X_n$.  In fact, we will prove a stronger result that the spectrum of $DF(u^*)$ is equivalent whether the operator is considered on $C[0,1]$ or $X_n$.  The spectrum of $DF(u^*)$ can be characterized in terms of Fredholm properties of $DF(u^*)-\lambda I$.  We will say that $\lambda\in\mathbb{C}$ is an element of the {\em essential spectrum} of $DF(u^*)$ if $DF(u^*)-\lambda I$ is either not Fredholm or is Fredholm with non-zero index.  Conversely, $\lambda$ lies in the {\em point spectrum} of $DF(u^*)$ if and only if $DF(u^*)-\lambda I$ is Fredholm with index zero and the kernel of this operator is nontrivial. The set of all such $\lambda$ belonging to the point spectrum of $DF(u^*)$ is denoted $\sigma_{pt}(DF(u^*))$, while the essential spectrum is denoted $\sigma_{ess}(DF(u^*))$. This leads to our first result. 
 
\begin{lem}\label{lem:fredholm}
The following dichotomy holds for any $n\geq 1$:
 \begin{itemize}
     \item If $\lambda \notin \mathrm{Rng}(-Q(x))$ then $DF(u^*)-\lambda I$ is Fredholm as an operator on $X_n$ with index zero and $\lambda\in \sigma_{pt}(DF(u^*))$ if and only if $\mathrm{ker} (DF(u^*)-\lambda I) \neq \emptyset$.
     \item If $\lambda\in \mathrm{Rng}(-Q(x))$ then $DF(u^*)-\lambda I$ is not Fredholm as an operator on $X_n$ and $\lambda\in \sigma_{ess}(DF(u^*))$.
 \end{itemize}  
\end{lem}

\begin{proof}
Throughout this proof we fix $n \geq 1$ since the arguments apply equally to any $X_n$. Then to begin, recall that for each $v \in X_n$ we have 
\[ 
    DF(u^*)v-\lambda v= -(Q(x)+\lambda) v+\int_0^1 W(x,y)D_2(u^*(x),u^*(y))v(y)\mathrm{d}y. 
\]
If $\lambda\notin \mathrm{Rng}(-Q(x))$ then $Q(x)+\lambda \neq 0$ and the multiplication operator $v\mapsto -(Q(\cdot)+\lambda)v$ is invertible on $X_n$ and hence Fredholm with index zero. Conversely, if  $\lambda\in \mathrm{Rng}(-Q(x))$ then the multiplication operator $v\to -(Q(\cdot)+\lambda)v$ is not invertible on $X_n$. Indeed, letting $x^*\in [0,1]$ be one value where $Q(x^*)+\lambda =0$, we get that the co-range of $-(Q(\cdot)+\lambda)$ includes all functions $w\in X_n$ for which $w(x^*)\neq 0$. The set of all such functions is infinite-dimensional, thus implying that the operator $v\mapsto -(Q(\cdot)+\lambda)v$ is not Fredholm if $\lambda\in \mathrm{Rng}(-Q(x))$.  

We will now verify that the integral operator $\int_0^1 W(x,y)D_2(u^*(x),u^*(y))v(y)\mathrm{d}y$ is compact. In doing so, we have that $DF(u^*)-\lambda I$ is a compact perturbation of the multiplication operator $v\to -(Q(\cdot)+\lambda)v$ and so we obtain that $DF(u^*)-\lambda I$ is Fredholm with index zero if and only if $-(Q(\cdot)+\lambda)v$ is too; see \cite[Theorem IV.5.26]{kato} for full details. 

To establish compactness, Hypothesis~\ref{hyp:Graphon}(2) is key, as outlined in \cite{graham79}. Let $v_j$ be a sequence of functions in $X_n$ with $\|v_j\|_\infty=1$ and set
\[ 
    \Psi_j(x)=\int_0^1 W(x,y)D_2(u^*(x),u^*(y))v_j(y) \mathrm{d}y.
\]
Let $\varepsilon>0$ and consider an arbitrary $a\in [0,1]$. Then the triangle inequality gives
\begin{eqnarray*}
    |\Psi_j(x)-\Psi_j(a)|&\leq & \bigg| \int_0^1 [W(x,y)-W(a,y)]D_2(u^*(x),u^*(y))v_j(y)\mathrm{d}y\bigg| \\
    &+& \bigg| \int_0^1 W(a,y)[D_2(u^*(x),u^*(y))-D_2(u^*(a),u^*(y))]v_j(y)\mathrm{d}y \bigg|
\end{eqnarray*}
Since the function $D_2(u^*(x),u^*(y))$ is uniformly continuous in $x,y \in [0,1]$ it is also uniformly bounded by some constant $L > 0$. Furthermore, there exists a $\delta > 0$ such that for any $a \in [0,1]$, if $|x - a| < \delta$ we have
\[
    |D_2(u^*(x),u^*(y))-D_2(u^*(a),u^*(y))| < \frac{\varepsilon}{2}
\]
and, from Hypothesis~\ref{hyp:Graphon}(2), 
\[
    \int_0^1 |W(x,y)-W(a,y)|\mathrm{d}y< \frac{\varepsilon}{2L}.
\]
Thus, for all $a \in [0,1]$, $j \geq 1$, and any $x$ such that $|x - a| < \delta$ we get  
\begin{equation}\label{eq:psicont}
    \begin{split}
    |\Psi_j(x)-\Psi_j(a)|\leq L \int_0^1 |W(x,y)-W(a,y)|\mathrm{d}y + \frac{\varepsilon}{2} < \frac{\varepsilon}{2} + \frac{\varepsilon}{2} = \varepsilon, 
    \end{split}
\end{equation}
using the fact that $\|v_j\|_\infty = 1$ and $|W(x,y)| \leq 1$ for all $x,y \in [0,1]$. Hence, the above bounds show that $\Psi_j$ is an equicontinuous family of functions and so the Arzel\'a--Ascoli theorem guarantees the existence of a convergent subsequence. This therefore implies that the operator $v\to \int_0^1 W(x,y)D_2(u^*(x),u^*(y)v(y))\mathrm{d}y$ is  compact and $DF(u^*)$ is a compact  perturbation of the multiplication operator $- (Q(x) + \lambda)$. Hence, $DF(u^*)-\lambda I$ is Fredholm if and only if $- (Q(x) + \lambda)$ is, per \cite[Theorem IV.5.26]{kato}. This concludes the proof. 
\end{proof}

From the previous proof we also obtain the following facts. 

\begin{cor}\label{cor:Qcont}  
The functions 
\[
    Q(x) = -f'(u^*(x)) - \int_0^1 W(x,y)D_1(u^*(x),u^*(y))\text{d}y
\]
and
\[
    \int_0^1 W(x,y)D_2(u^*(x),u^*(y))v(y)\text{d}y
\]
are continuous in $x$ for any function $v \in X_n$.
\end{cor}

\begin{proof}
Uniform continuity of the second function was verified in \eqref{eq:psicont}.  For $Q(x)$, we have that $f'(u^*(x))$ is continuous since both $f'$ and $u^*$ are.  The proof of continuity of the integral part follows from an analogous computation as that of $\int_0^1W(x,y)D_2(u^*(x),u^*(y))v(y)\mathrm{d}y$ and so we omit the details.   
\end{proof}

\begin{lem}\label{lem:Spectrum}
    The spectrum of $DF(u^*)$   posed on $X_n$ is equivalent to the spectrum on $C[0,1]$, i.e. $\sigma(DF(u^*))|_{C[0,1]}=\sigma(DF(u^*))|_{X_n}$.  Furthermore, the generalized eigenspaces associated to any element of the point spectrum is spanned by continuous functions.   
\end{lem}

\begin{proof}

The essential spectrum of $DF(u^*)$ characterized independently of $n$ in Lemma~\ref{lem:fredholm}. In particular, the essential spectrum is the same whether the operator is posed on $C[0,1]$ or $X_n$ with any $n\geq 1$. Thus, it only remains to examine the point spectrum.

Since $C[0,1] \subset X_n$ for all $n\geq 1$, we immediately get the inclusion $\sigma_{pt}(DF(u^*))|_{C[0,1]}\subseteq\sigma_{pt}(DF(u^*))|_{X_n}$. Now, to show the opposite inclusion we assume that $\lambda\in \sigma_{pt}(DF(u^*))|_{X_n}$. We will show that any kernel element of $DF(u^*)-\lambda I$ must be a continuous function, giving the equality $\sigma_{pt}(DF(u^*))|_{C[0,1]}=\sigma_{pt}(DF(u^*))|_{X_n}$. 

To prove the above statement, let us assume that $\psi\in \mathrm{ker}\left(DF(u^*)-\lambda I\right) \cap X_n$ for some $n \geq  1$. Our goal is to show $\psi \in C[0,1]$ as well. By definition of $\psi$ we have 
\begin{equation}\label{kernelSpec}
    -(Q(x)+\lambda)\psi(x)+\int_0^1 W(x,y) D_2(u^*(x),u^*(y))\psi(y)\mathrm{d}y=0.
\end{equation}
Since $\lambda\notin \sigma_{ess}(DF(u^*))|_{X_n}$, Lemma~\ref{lem:fredholm} gives that $Q(x)+\lambda\neq 0$ and Corollary~\ref{cor:Qcont} further gives that $Q(x)+\lambda$ is continuous. Hence, rearranging \eqref{kernelSpec} gives that 
\begin{equation}\label{kernelSpec2}
    \psi(x) = \frac{1}{Q(x)+\lambda}\int_0^1 W(x,y) D_2(u^*(x),u^*(y))\psi(y)\mathrm{d}y    
\end{equation}
and from Corollary~\ref{cor:Qcont} we have that 
\[ 
    \int_0^1 W(x,y) D_2(u^*(x),u^*(y))\psi(y)\mathrm{d}y\in C[0,1].
\] 
Therefore, the right-hand-side of \eqref{kernelSpec2} is continuous, thus giving that $\psi\in C[0,1]$, as desired.

Finally, continuity of any generalized eigenfunction can be established in a similar manner.  Indeed, consider an eigenfunction $\psi\in C[0,1]$ of $DF(u^*)-\lambda I$ and let $\eta\in X_n$ be a generalized eigenfunction which satisfies
\[-(Q(x)+\lambda)\eta(x)+\int_0^1 W(x,y) D_2(u^*(x),u^*(y))\eta(y)\mathrm{d}y=\psi(x). \]
Arguing as above, we have that $\int_0^1 W(x,y) D_2(u^*(x),u^*(y))\eta(y)\mathrm{d}y$ is continuous and we therefore obtain that $\eta\in C[0,1]$.  The argument extends to any element of the generalized eigenspace.
\end{proof} 

This leads to the following corollary which follows from Lemma~\ref{lem:Spectrum} and Hypothesis~\ref{hyp:ContSol} that assumes that $DF(u^*)$ is invertible on $C[0,1]$. 

\begin{cor}\label{cor:DFinvertible}
    For all $n\geq 1$, the operator $DF(u^*):X_n \to X_n$ is well-defined and invertible.    
\end{cor}

From Lemma~\ref{lem:Spectrum} and Corollary~\ref{cor:DFinvertible}, the spectrum is the same regardless of whether we pose $DF(u^*)$  on $C[0,1]$ or $X_n$ for any $n\geq 1$. Thus, we can drop all notation indicating the underlying space. That is, in what follows we simply write $\sigma(DF(u^*))$ to denote the spectrum of $DF(u^*)$, regardless of the space that the operator is posed on.

We now turn our attention to interpreting the inverse of  $DF(u^*)$ on the spaces $X_n$. By definition, we are required to solve 
\[ 
    DF(u^*)v=-Q(x)v+\int_0^1 W(x,y)D_2(u^*(x),u^*(y))v(y)\mathrm{d}y =w,
\]
for any $w\in X_n$. Dividing this equation by the non-zero continuous function $-Q(x)$ we obtain the equivalent formulation 
\[
    \left(v-\int_0^1 W(x,y)\frac{D_2(u^*(x),u^*(y))}{Q(x)}v(y)\mathrm{d}y \right) =-\frac{w(\cdot)}{Q(\cdot)}.
\]
The operator on the left hand side can be written as $(I-T_K)$ where we introduce the notation
\begin{equation} \label{eq:TK}
    T_Kv=\int_0^1 W(x,y)\frac{D_2(u^*(x),u^*(y))}{Q(x)}v(y)\mathrm{d}y.
\end{equation}
Therefore, inversion of $DF(u^*)$ is equivalent to solving a Fredholm integral equation of the second kind.  This leads to the following lemma whose proof is left to Appendix~\ref{app:BigProof}.  

\begin{lem}\label{lem:MinvProjected}
    The following holds:
    \begin{enumerate}
        \item For all $n \geq 1$ the operator $T_K:X_n \to X_n$, as defined in \eqref{eq:TK}, is compact.
        \item The spectrum of $T_K$ as an operator on $X_n$ is independent of $n$ and has at most a finite number of eigenvalues with real part greater than one. We denote the spectrum over any $X_n$ as $\sigma(T_K)$.
        \item Let $\{\lambda_j\}_{j=1}^J$ denote the (finitely many) eigenvalues of $T_K$ with real part greater than or equal to  $1$ and let $m_j \geq 1$ denote their algebraic multiplicity. Let $\{\varphi_{j,k}(x)\}_{k=1}^{m_j}$ be a set of continuous and linearly independent functions that span the generalized eigenspace associated to the eigenvalues $\lambda_j$. There exists a complementary set of continuous, linearly independent functions  $\{\psi_{j,k}(y)\}_{k=1}^{m_j}$  such that the spectral projection $P:X_n\to X_n$ associated with these $J$ eigenvalues is expressed as  
        \begin{equation}\label{eq:P}
            Pv=\sum_{j=1}^J\sum_{k=1}^{m_j} \int_0^1 \varphi_{j,k}(x)\psi_{j,k}(y)v(y) \mathrm{d} y. 
        \end{equation}
        \item Letting $\tilde{P}=I-P$ be the spectral projection associated to the subset of the spectrum whose real part is less than one, there exists a $\xi>0$ such that $DF(u^*)^{-1}:X_n \to X_n$ can be represented as
        \begin{equation}\label{eq:Minvgen}
            DF(u^*)^{-1}w= \sum_{k=0}^\infty \left(\frac{T_K+\xi}{1+\xi}\right)^k \tilde{P} \left(\frac{-w}{Q(\cdot)(1+\xi)}\right) + \sum_{j=1}^J\sum_{k,l=1}^{m_j} c_{j,k,l}\varphi_{j,k}(x)\int_0^1 \psi_{j,l}(y)\left(\frac{-w(y)}{Q(y)}\right)\mathrm{d}y ,
        \end{equation}
        for some constants $c_{j,k,l}$ and every $n \geq 1$.
    \end{enumerate}
\end{lem}



We will sometimes require a general bound on the inverse operator. 

\begin{lem} \label{lem:mbound}  There exists an $m>0$, independent of $n$,  such that the following bound holds
    \begin{equation} \label{eq:DFinvbd}
   \| DF(u^*)^{-1}\|_{X_n\to X_n} \leq m,
\end{equation}
for all $X_n$.
\end{lem}

\begin{proof}
    Consulting the formula for the inverse derived in \eqref{eq:Minvgen}, we observe that the only possible obstruction to the uniform bound in \eqref{eq:DFinvbd} stems from the infinite sum.  Indeed, uniform bounds on the finite sum coming from the inversion on the finite-dimensional space $PX_n$ are readily obtained.  We therefore consider the following operator
\begin{equation} \label{eq:MinvKeyPart1}
    \frac{1}{1+\xi}\sum_{k=0}^\infty \left(\frac{T_K+\xi}{1+\xi}\right)^k \tilde{P}.  
\end{equation}
The spectral radius of the operator $T_\xi=\frac{T_K+\xi}{1+\xi}$ on the space $\tilde{X}_n=\mathrm{Rng}(\tilde{P})$ is strictly less than one by construction and independent of $n$ by Lemma~\ref{lem:MinvProjected}.  Therefore, for any $n$ the series converges absolutely.  

We will now write (\ref{eq:MinvKeyPart1}), considered as an operator on $\tilde{X}_n=\mathrm{Rng}(\tilde{P})$ in the form
\[ 
    \frac{1}{1+\xi}\left(I-T_\xi\right)^{-1}=(I-T_K)^{-1}|_{\tilde{X}_n}=I+V,
\]
for some operator $V$.  Re-arranging, 
\[ (I-T_K)^{-1}|_{\tilde{X}_n}\left(I-(I-T_K)V\right)=I, \]
which implies that, restricted to $\tilde{X}_n$,  
\[ I-V+T_KV=I-T_K,\]
after which we find that $V:\tilde{X}_n\to \tilde{X}_n$ has the expression 
\[ 
    V= (I-T_K)^{-1}|_{\tilde{X}_n}T_K. 
\]
With this expansion, we see that \eqref{eq:MinvKeyPart1} has a uniform bound in $n$ as follows. From Corollary~\ref{cor:Qcont}, $T_K$ takes elements of $X_n$ into $C[0,1]$ and so is easily bounded independent of $n$ since all spaces have the same norm. Furthermore, $(I-T_K)^{-1}|_{\tilde{X}_n}$ is bounded as an operator on $C[0,1]\cap \tilde{X}_n$, leading to a uniform operator bound on \eqref{eq:MinvKeyPart1} that is independent of $n$. This completes the proof.
\end{proof}

For future reference we will write
\[ 
    V(w)=H[T_K(w)],
\]
where $H=(I-T_K)^{-1}|_{\tilde{X}_n}\tilde{P}$ and, from Lemma~\ref{lem:mbound}, is a bounded linear operator on $C[0,1]$. Precisely, there exists $b > 0$ so that 
\begin{equation}\label{eq:Bbound} 
    \|H[w]\|_\infty\leq b \|w\|_\infty
\end{equation} 
for all $w \in C[0,1]$.

\subsubsection{Two useful integral bounds}\label{sec:twointegrals}

In the analysis throughout the remainder of this section we will often have to estimate integrals involving the difference $W(x,y)-W_n(x,y)$. In what follows we establish two auxiliary lemmas that will be employed to control certain terms in $\mathcal{T}_n[u]$ in the following subsection. 

The first of our lemmas shows that when integrating $W-W_n$ against a continuous function $\phi(x,y)$ the resulting function can be made arbitrarily small in the supremum norm by taking $n$ sufficiently large, based on the assumed convergence of $\|W-W_n\|_\square$ and $\|d_{W_n}-d_W\|_\infty$ as $n\to \infty$. The result is as follows. 

\begin{lem}\label{lem:WWnappliedtocontissmall}
Suppose that $\phi\in C([0,1]\times [0,1])$.   Then for any $\varepsilon > 0$ there exists a constant $C_1(\varepsilon,\phi)>0$, independent of $n$, such that 
    \begin{equation}
         \sup_{x\in[0,1]} \bigg| \int_0^1 \left[ W(x,y)-W_n(x,y)\right] \phi(x,y)\mathrm{d} y  \bigg| < \varepsilon+C_1(\varepsilon,\phi) \|d_{W_n}-d_W\|_\infty .
    \end{equation}
\end{lem}

\begin{proof} 
Let us fix an $\varepsilon>0$ and begin by recalling that  $\phi(x,y)$ is continuous on $[0,1]\times [0,1]$. Therefore, for any $\varepsilon>0$ there exists $M\in\mathbb{N}$ such that, 
\[ 
\phi(x,y)=\sum_{i,j=1}^M \phi_{ij}\zeta_i(x)\zeta_j(y) +\Delta \phi(x,y) 
\]
where $\zeta_i$ is an indicator function defined as 
\begin{equation} \label{eq:zeta} 
    \zeta_i(x)=\left\{ \begin{array}{cc} 1 & x\in \left[\frac{i-1}{M},\frac{i}{M}\right) \\
0 & \text{otherwise} \end{array}\right., 
\end{equation}
and $|\Delta\phi(x,y)|<\varepsilon$. We use notation $I_j=\left[\frac{j-1}{M},\frac{j}{M}\right)$ and consider an arbitrary $x\in I_j\subset [0,1]$. Then, we expand 
\begin{equation}\label{phiI_mIntegral}
    \begin{split}
        \bigg|\int_0^1 &[W_n(x,y)-W(x,y)] \phi(x,y) \mathrm{d}y \bigg| \\
        &=\bigg|\sum_{j=1}^M \left(\int_{I_j} [W_n(x,y)-W(x,y)]\phi_{ij}  \mathrm{d}y+ \int_{I_j} [W_n(x,y)-W(x,y)] \Delta \phi(x,y)\mathrm{d}y\right)\bigg| \\
        &\leq\bigg|\sum_{j=1}^M \left(\int_{I_j} [W_n(x,y)-W(x,y)]\phi_{ij} \mathrm{d}y \bigg|+ \sum_{j=1}^M\bigg|\int_{I_j} [W_n(x,y)-W(x,y)] \Delta \phi(x,y)\mathrm{d}y\right)\bigg|,
    \end{split}
\end{equation}
The first summation in the above expression can be bounded by 
\begin{equation}
\bigg|\sum_{j=1}^M \phi_{ij}\int_{I_j} [W_n(x,y)-W(x,y)] \mathrm{d}y\bigg|\leq M\|d_{W_n}-d_W\|_{L^\infty(I_j)}\sup |\phi_{ij}| .
\end{equation}
For the second summation in \eqref{phiI_mIntegral} we use the fact that $|W_n(x,y)-W(x,y)|\leq 1$ to get 
\begin{equation}
    \bigg|\int_{I_j} [W_n(x,y)-W(x,y)] \Delta \phi(x,y) \mathrm{d}y\bigg|\leq \bigg|\int_{I_j} [W_n(x,y)-W(x,y)] \Delta \phi(x,y)\mathrm{d}y\bigg|\leq \frac{\varepsilon}{M},
\end{equation}
since each interval $I_j$ has length $1/M$. Finally, summing over the $M$ sub-intervals and using the inequality \eqref{phiI_mIntegral} we get
\begin{equation}
    \bigg|\int_0^1 [W_n(x,y)-W(x,y)] \phi(x,y) \mathrm{d}y\bigg|\leq \varepsilon+ M\sup_j(\|d_{W_n}-d_W\|_{L^\infty(I_j)}) \sup_{i,j=1, \dots, M} |\phi_{ij}|.   
\end{equation} 
Since $\phi(x,y)$ is uniformly continuous on $[0,1]\times [0,1]$ the supremum is finite, while $\sup_j \{\|d_{W_n}-d_W\|_{L^\infty(I_j)}\} \leq \|d_{W_n}-d_W\|_\infty$ since the $I_j$ are a partition of $[0,1]$, thus removing all $j$-dependence on this quantity. Since $x\in[0,1]$ was arbitrary the stated estimate then follows.   
\end{proof} 

The next result that we present considers repeated integrals and again involves the difference $W(x,y)-W_n(x,y)$. In the proof we will make use of the following alternative version of the cut-norm (see \cite[Section~4]{janson2010graphons}), defined as
\begin{equation} \label{eq:2ndcutnorm}
    \|W\|_{\square,2} =\sup_{\|f\|_\infty \leq 1,\|g\|_\infty \leq 1} \bigg|\int_0^1 \int_0^1 W(x,y)f(x)g(y)\mathrm{d}y\mathrm{d} x \bigg|.
\end{equation}
Importantly, \cite{janson2010graphons} proves the inequality $\|W\|_{\square,2}\leq 4 \|W\|_{\square}$, and so it holds that convergence in the cut-norm we are working with (see (\ref{cutnorm})) implies convergence in the alternative norm $\|\cdot\|_{\square,2}$. This leads to the following lemma.

\begin{lem} \label{lem:W-Wnappliedtwice}
    Suppose that $\phi: [0,1]\times [0,1]\to \mathbb{R}$ is continuous, $\psi:[0,1]\times [0,1]\to \mathbb{R}$ is bounded and $v\in X_n$.  Then for any $\varepsilon>0$ there exists an constant $C_2(\varepsilon,\psi,\phi)>0$, independent of $n$, such that 
     \begin{equation} \label{eq:W-Wncontrolled}
        \begin{split}
         &\sup_{x\in[0,1]} \bigg| \int_0^1 \int_0^1 \left[ W(z,y)-W_n(z,y)\right] \psi(x,z)\phi(z,y)v(y) \mathrm{d}z \mathrm{d} y  \bigg| \\ 
         & \qquad\qquad\qquad < \varepsilon \|v\|_\infty+C_2(\varepsilon,\psi,\phi)\|W-W_n\|_\square \|v\|_\infty.
        \end{split}
    \end{equation}
\end{lem}

\begin{proof}  
The proof of this result mimics the previous lemma.  First, for any $\varepsilon>0$, there exists a $M\in\mathbb{N}$ such that
\begin{equation} \label{eq:phidef}
\phi(z,y)=\sum_{i,j=1}^M \phi_{ij}\zeta_i(z)\zeta_j(y) +\Delta \phi(z,y),  
\end{equation}
with $\zeta_{i,j}$ defined in \eqref{eq:zeta} and with the remainder terms obeying the bound 
\[ 
|\Delta\phi(z,y)|<\frac{\varepsilon}{\|\psi\|_\infty }.
\]  
Here we have implicitly assumed that $\|\psi\|_\infty \neq 0$ as otherwise the result would be trivial.

Now take an arbitrary $x\in [0,1]$. Using the expansion \eqref{eq:phidef}, the integral in \eqref{eq:W-Wncontrolled} can be expressed as 
\begin{eqnarray*}
    \sum_{i,j=1}^M \int_0^1\int_0^1 [W(z,y)-W_n(z,y)]\psi(x,z)\phi_{ij}\zeta_i(z)\zeta_j(y)v(y)\mathrm{d}z\mathrm{d}y \\
    +\int_0^1\int_0^1[W(z,y)-W(z,y)] \psi(x,z)\Delta\phi(z,y)v(y)\mathrm{d}z\mathrm{d}y.  
\end{eqnarray*}
The first integral is in an adequate form for a  direct application of (\ref{eq:2ndcutnorm}) and we obtain 
\[ 
    \begin{split}
        \bigg|\sum_{i,j=1}^M \int_0^1\int_0^1 &[W(z,y)-W_n(z,y)]\psi(x,z)\phi_{ij}\zeta_i(z)\zeta_j(y)v(y)\mathrm{d}z\mathrm{d}y\bigg|
        \\ &\leq  M^2 \sup_{i,j}|\phi_{ij}| \|W-W_n\|_{\square,2} \|\psi\|_\infty \|v\|_\infty.  
        \\ &\leq 4 M^2 \sup_{i,j}|\phi_{ij}| \|W-W_n\|_\square \|\psi\|_\infty \|v\|_\infty.  
    \end{split}
\]
Since the right-hand bound is independent of the $x$ being considered, it follows that taking the supremum over $x \in [0,1]$ over the left-hand-side obeys the same bound. For the integral involving the remainder terms $\Delta \phi$, we can estimate directly using the facts that $|W(x,y) - W_n(x,y)|\leq 1$ and $|\Delta \phi (x,y)|\leq \frac{\varepsilon}{\|\psi\|_\infty}$ for all $(x,y) \in [0,1]\times[0,1]$ to get
\[ 
    \bigg| \int_0^1\int_0^1[W(z,y)-W(z,y)] \psi(x,z)\Delta\phi(z,y)v(y)\mathrm{d}z\mathrm{d}y\bigg| \leq \varepsilon \|v\|_\infty.
\]
We therefore have obtained \eqref{eq:W-Wncontrolled} with $C_2(\varepsilon,\psi,\phi)=4 M^2\sup_{ij} |\phi_{ij}|\|\psi\|_\infty$, which is independent of $n$ as claimed.  
\end{proof}

\subsection{Residual Estimates}\label{sec:keyestimates}

Building towards an application of the contraction mapping theorem, we now obtain estimates on the function $F_n(u^*)$ and the operator $DF(u^*)^{-1}[DF(u^*)-DF_n(u)]$.  Our work here will involve applications of Lemma~\ref{lem:WWnappliedtocontissmall} and Lemma~\ref{lem:W-Wnappliedtwice}.

{\bf Notation}: In the remaining analysis we will consider balls of radius $\rho$ in $X_n$, denoted $B_\rho(u^*)$. We will always assume that $\rho<1$.  From Hypothesis~\ref{hyp:Smooth}, $f$, $D$ and their derivatives are continuous. Taking inputs in a bounded region provides that $f$, $D$ and their derivatives $f'$, $D_1$ and $D_2$ are globally bounded by some uniform constant and Lipschitz continuous. To simplify notation we will let $L > 0$ denote a uniform bound on these functions, their derivatives and their Lipschitz constants over all inputs with $\|u - u^*\|_\infty \leq 1$.  We emphasize that for functions with two inputs this Lipshitz bound is taken to be with respect to the sup norm on $\mathbb{R}^2$ so that, for example,  $|D(a_1,b_1)-D(a_2,b_2)|\leq L \max\{ |a_1-a_2|,|b_1-b_2|\} $.

We begin this section by showing that for $n$ sufficiently large the continuous function $u^*\in C[0,1]$ is an approximate solution to the problem $F_n(u)=0$ in the sense that $\|F_n(u^*)\|_\infty$ is small.  

\begin{lem}\label{lem:Fnsmall} For any $\varepsilon > 0$ there exists an $N\in\mathbb{N}$ such that for all $n\geq N$ 
\[ 
    \|F_n(u^*)\|_\infty<\varepsilon.
\]
\end{lem}

\begin{proof} 
First, Hypothesis~\ref{hyp:ContSol} gives that $F(u^*) = 0$ and so $\|F_n(u^*)\|_\infty = \|F_n(u^*)-F(u^*)\|_\infty$. Since $F_n$ and $F$ only differ by the graphons $W_n$ and $W$, the result holds if we can show 
\begin{equation}\label{Fnu^*_residue}
    \sup_{x\in[0,1]} \bigg| \int_0^1\left[W_n(\cdot,y)-W(\cdot,y)\right]D(u^*(x),u^*(y))\mathrm{d} y \bigg| < \varepsilon .
\end{equation}
To achieve this inequality we simply apply Lemma~\ref{lem:WWnappliedtocontissmall}. Indeed, let $\varepsilon>0$ and take $\phi(x,y)=D(u^*(x),u^*(y))$. Then, there exists a $C>0$, independent of $n$, such that 
\begin{equation}
    \sup_{x\in[0,1]} \bigg| \int_0^1\left[W_n(\cdot,y)-W(\cdot,y)\right]D(u^*(x),u^*(y))\mathrm{d} y \bigg| < \frac{\varepsilon}{2}+C\| d_{W_n}-d_W\|_\infty .
\end{equation}
Since $C$ is independent of $n$, Hypothesis~\ref{hyp:Graphon}(1) guarantees that there is an $N$ sufficiently large so that $C\| d_{W_n}-d_W\|_\infty<\frac{\varepsilon}{2}$ for all $n\geq N$. This therefore completes the proof.  
\end{proof} 

Now that we have shown $F_n(u^*)$ is small in the sup-norm when $n$ is large, we turn to showing something similar for the operator norm of $DF(u^*)^{-1}[DF(u^*)-DF_n(u)]$. To begin, recall that $DF(u^*)$ acts on $v \in X_n$, for any $n \geq 1$, by
\[ 
    DF(u^*)v = -Q(x)v+\int_0^1 W(x,y)D_2(u^*(x),u^*(y)) v(y)\mathrm{d} y. 
\]
Recall further that Hypothesis~\ref{hyp:ContSol} gives that $Q(x)>0$ for all $x \in [0,1]$, while Corollary~\ref{cor:Qcont} shows that $Q$ is continuous.  We will express the linearization $DF_n(u)$ in a similar fashion,
\[ 
    DF_n(u)v=-Q_n(u(x))v+\int_0^1 W_n(x,y)D_2(u(x),u(y)) v(y)\mathrm{d} y,
\]
where $W$ is replaced by $W_n$ in the definition of $Q$ to get $Q_n$. We now proceed by obtaining bounds on the difference $DF(u^*)-DF_n(u)$ in two steps: first showing that $|Q(x)-Q_n(u(x))|$ can be made small by taking $n$ sufficiently large and then considering the (more challenging) integral parts of the operators.  

\begin{lem}\label{lem:Qsmall} 
For any $\varepsilon>0$ there exists a $C>0$, independent of $n$, such that for any $u\in B_1(u^*)\subset X_n$ we have 
\begin{equation} \label{eq:Q-Qn}
    \sup_{x\in [0,1]} | Q(x)-Q_n(u(x))|< \varepsilon+C\|d_W-d_{W_n}\|_\infty+ 2L \|u-u^*\|_\infty, 
\end{equation}
where 
\[ 
Q_n(u(x))=-f'(u(x))-\int_0^1 W_n(x,y)D_1(u(x),u(y))\mathrm{d} y. 
\]
Furthermore, there exists a $\rho>0$ and an $N\in\mathbb{N}$ such that $Q_n(u(x))\neq 0$ for all $x \in [0,1]$, all $n\geq N$, and any $u \in X_n$ with $\|u-u^*\|_\infty<\rho$.
\end{lem}

\begin{proof} 
First, adding and subtracting the term $W_n(x,y)D_1(u^*(x),u^*(y))$ and using the triangle inequality, for any $x \in [0,1]$ we have the bound
\[
    \begin{split}
        | Q(x)-Q_n(u(x))| &\leq |f'(u^*(x))-f'(u(x))| + \bigg|\int_0^1 W_n(x,y)\left[D_1(u(x),u(y))-D_1(u^*(x),u^*(y))\right]\mathrm{d} y \bigg| \\
        &+ \bigg|\int_0^1 \left[W_n(x,y)-W(x,y)\right] D_1(u^*(x),u^*(y))\mathrm{d} y \bigg|.
    \end{split}
\]
Then, Lipschitz continuity of $f'$ implies that 
\[ 
    |f'(u^*(x))-f'(u(x))| \leq L|u^*(x)-u(x)|\leq L\|u-u^*\|_\infty,
\]
for all $x \in [0,1]$. Then, since $0\leq W_n(x,y)\leq 1$ and $D_1$ is Lipschitz (since it is twice differentiable) it holds that 
\[  
    \sup_{x\in [0,1]} \bigg|\int_0^1 W_n(x,y)\left[D_1(u(x),u(y))-D_1(u^*(x),u^*(y))\right]\mathrm{d} y \bigg| \leq L\| u-u^* \|_\infty. 
\]
Finally, using the fact that $D_1(u^*(x),u^*(y))$ is continuous, we can appeal to Lemma~\ref{lem:WWnappliedtocontissmall} which guarantees that for any $\varepsilon>0$ there exists a constant $C_1(\varepsilon,D_1(u^*(x),u^*(y)))>0$, independent of $n$, for which 
\[  
    \sup_{x\in [0,1]} \bigg|\int_0^1 \left[W_n(x,y)-W(x,y)\right] D_1(u^*(x),u^*(y))\mathrm{d} y \bigg| <\varepsilon +C_1\|d_W-d_{W_n}\|_\infty . 
\]
Upon summing up the three bounds above, the estimate \eqref{eq:Q-Qn} then follows.  

We now turn to proving the second statement in the lemma. Since $Q(x)$ is continuous and non-zero we can also guarantee that $Q_n(u(x))\neq 0$ for any $u$ sufficiently close to $u^*$  and any $n$ sufficiently large. Specifically, let 
\[
    Q_*=\min_{x \in [0,1]} Q(x).
\]
Then, one can take $\varepsilon=Q_*/4$, $\rho=\frac{Q_*}{8L}$ and $N$ sufficiently large so that $C\|d_W-d_{W_n}\|_\infty<Q_*/4$ in \eqref{eq:Q-Qn}. This gives that $|Q(x)-Q_n(u(x))|<3Q_*/4$, which further implies $Q_n(u(x))>Q_*/4>0$. This concludes the proof. 
\end{proof}

We now turn our attention to the integral portion of the operator  $DF(u^*)-DF_n(u)$.  This part of the operator is expressed by the difference 
\[ 
    \int_0^1 \left[ W(x,y)D_2(u^*(x),u^*(y))-W_n(x,y)D_2(u(x),u(y))\right] v(y) \mathrm{d}y.
\]
We add and subtract $W_n(x,y)D(u^*(x),u^*(y))$ and separate this integral into two components: the first taking the form 
\begin{equation} \label{eq:inteasy} 
    \int_0^1 W_n(x,y)\left[D_2(u^*(x),u^*(y))-D_2(u(x),u(y))\right]v(y)\mathrm{d}y
\end{equation} 
and the second being
\begin{equation} \label{eq:inthard} 
    \int_0^1 \left[W(x,y)-W_n(x,y)\right]D_2(u^*(x),u^*(y))v(y)\mathrm{d}y .
\end{equation}
Bounds on the norm of \eqref{eq:inteasy} are presented in the following lemma. 

\begin{lem}\label{lem:inteasy}
    Suppose $u\in B_1(u^*)$ and $v\in X_n$, for any $n \geq 1$. Then 
    \begin{equation} \label{eq:WnD2-D2} 
        \sup_{x\in[0,1]} \bigg|\int_0^1 W_n(x,y)\left[D_2(u^*(x),u^*(y))-D_2(u(x),u(y))\right]v(y)\mathrm{d}y\bigg| \leq L \|u-u^*\|_\infty \|v\|_\infty  
    \end{equation}
\end{lem}

\begin{proof}
    Since $0\leq W_n(x,y)\leq 1$, $D_2$ is Lipschitz and $u\in B_1(u^*)$ we obtain the stated estimate directly.
\end{proof}

Having bounded \eqref{eq:inteasy}, bounding the integral \eqref{eq:inthard} remains.  As we will demonstrate, it is not possible, in general, to make the operator  norm of this operator small through some combination of taking $\|u-u^*\|_\infty$ small or $n$ large. 

To illustrate why this could be the case, consider a fixed $x\in [0,1]$ and the function  $v\in B_\rho(0)$ defined by
\[ 
    v(y)= \rho \ \mathrm{sign}\ \left( \left[W_n(x,y)-W(x,y)\right]D_2(u^*(x),u^*(y))\right).
\]
Then \eqref{eq:inthard} with this choice of $v$ will not necessarily be small relative to the magnitude of $v$\footnote{If $D_2(u^*(x),u^*(y))$ changes sign then $v$ may not be an element of $X_n$ as it fails to be continuous on some sub-interval.  Nonetheless, for $n$ sufficiently large one can approximate this function arbitrarily well as $n\to\infty$ with elements of $X_n$ and the issue remains for large $n$. }.  Looking ahead, control of this integral will be key in our attempt to show that $\mathcal{T}_n$ is a contraction mapping and demonstrating that this operator has a norm that cannot  exceed one turns out to be something we cannot show with the level of generality endowed by Hypotheses~\ref{hyp:Smooth}-\ref{hyp:ContSol}.  

Dealing with the operator in \eqref{eq:inthard} will be the main technical issue moving forward.  For that reason, and since we will have to keep careful track of this term in the the subsequent analysis we introduce the following notation for the operator, 
\begin{equation}\label{eq:Xidefn} 
    \Xi(W_n,W,u^*)v = \int_0^1 \left[W_n(x,y)-W(x,y)\right]D_2(u^*(x),u^*(y))v(y)\mathrm{d} y.  
\end{equation}
With this notation, we gather the estimates from our work so far in the following corollary.

\begin{cor}\label{cor:DF-DFn}   
For any $\varepsilon>0$ there exists a $C_1>0$, independent of $n$, such that for any $n\geq 1$, $u\in B_1(u^*)\subset X_n$, and $v\in X_n$ it holds that  
    \begin{equation}
        [DF(u^*)-DF_n(u)]v=\Xi(W_n,W,u^*)v+R_D[v]
    \end{equation}
    where 
    \[ 
    \|R_D[v]\|_\infty\leq \left(\varepsilon+C_1(\varepsilon,D_1)\|d_W-d_{W_n}\|_\infty+ 3L \|u-u^*\|_\infty\right) \|v\|_\infty 
    \]
\end{cor}

We now build upon Corollary~\ref{cor:DF-DFn} with the goal of understanding the operator $ DF(u^*)^{-1}\left[ DF(u^*)-DF_n(u) \right] $, for various choices of $u$. This operator will appear later when applying the contraction mapping theorem and so the following result provides an expansion of this operator for any $u\in B_1(u^*)\subset X_n$.

\begin{lem} \label{lem:DFdiffdef} 
    Let $\varepsilon>0$. Then, there exists positive constants $C_1(\varepsilon)$ and $C_2(\varepsilon)$, independent of $n$, such that for any $n\geq 1$, $u\in B_1(u^*)\subset X_n$, and $v\in X_n,$ it holds that 
    \begin{equation}  \label{eq:DFinvDF-DFn} DF(u^*)^{-1}\left[ DF(u^*)-DF_n(u) \right] v =-\frac{\Xi(W_n,W,u^*)}{Q(x)}v+R_u[v],  \end{equation}
where 
\begin{equation} \label{eq:Rbounds}  
    \left\| R_u[v]\right\|_\infty\leq \left(\varepsilon+C_1\|d_W-d_{W_n}\|_\infty + C_2\|W-W_n\|_\square + 3Lm\|u-u^*\|_\infty\right) \|v\|_\infty 
\end{equation}
\end{lem}

\begin{proof} 
In Lemma~\ref{lem:MinvProjected}, we showed that $DF(u^*)$ is invertible as a linear operator acting on $X_n$, for any $n\geq 1$. Furthermore, Lemma~\ref{lem:mbound} showed that the operator norm of the inverse has a uniform bound independent of $n$.  

Now, according to Corollary~\ref{cor:DF-DFn}, we have $[DF(u^*)-DF_n(u)]v=\Xi(W_n,W,u^*)v+R_D[v]$ and so we seek to estimate $DF(u^*)^{-1}\left(\Xi(W_n,W,u^*)v+R_D[v]\right)$. 
Using the uniform bound on \eqref{eq:DFinvbd} and Corollary~\ref{cor:DF-DFn}, for any $\varepsilon_{D}>0$ we have that there exists a $C_1(\varepsilon_{D},D_1) > 0$, independent of $n$, so that
\[ 
    \|DF(u^*)^{-1}R_D[v]\|_\infty\leq \left(m\varepsilon_D+mC_1(\varepsilon_D,D_1)\|d_W-d_{W_n}\|_\infty+3mL \|u-u^*\|_\infty\right) \|v\|_\infty. 
\]
We thus have reduced the problem to estimating $DF(u^*)^{-1}\Xi(W_n,W,u^*)v$, which consists of two pieces. Recall the formula for $DF(u^*)^{-1}$ presented in (\ref{eq:Minvgen}).  First, using Lemma~\ref{lem:MinvProjected} we consider 
\[ 
    \sum_{j=1}^J\sum_{k,l=1}^{m_j} c_{j,k,l}\varphi_{j,k}(x)\int_0^1 \psi_{j,l}(z)\left(\frac{-\Xi(W_n,W,u^*)v}{Q(z)}\right)\mathrm{d}z. 
\]
For any $j,k,l$, the corresponding term in the sum, after expanding $\Xi$ according to its definition \eqref{eq:Xidefn}, takes the form
\[ 
    c_{j,k,l} \int_0^1\int_0^1 [W_n(z,y)-W(z,y)] \varphi_{j,k}(x) \psi_{j,l}(z)\frac{D_2(u^*(z),u^*(y))}{Q(z)} v(y)\mathrm{d}y\mathrm{d}z.
\]
For each $j,k,l$, consider an arbitrary $\varepsilon_{j,k,l}>0$. Since $\frac{D_2(u^*(z),u^*(y))}{Q(z)}$ is continuous, Lemma~\ref{lem:W-Wnappliedtwice} guarantees the existence of a $C_{j,k,l} > 0$ which is independent of $n$ so that 
\begin{eqnarray} 
    \left\| c_{j,k,l} \int_0^1\int_0^1 [W_n(z,y)-W(z,y)] \varphi_{j,k}(x) \psi_{j,l}(z)\frac{D_2(u^*(z),u^*(y))}{Q(z)}v(y) \mathrm{d}y \mathrm{d}z\right\|_\infty \nonumber  \\
    \leq \left(
    \varepsilon_{j,k,l}+C_{j,k,l}
    \|W-W_n\|_\square\right) \|v\|_\infty.
\end{eqnarray}
The constant  $C_{j,k,l}$ depends on $\varepsilon_{j,k,l}$, $\phi(z,y)=\frac{D_2(u^*(z),u^*(y))}{Q(z)}$, and $\psi(x,z)=c_{j,k,l}\varphi_{j,k}(x)\psi_{j,l}(z)$, but we suppress this dependence in the remainder of this proof.  We emphasize once again that $C_{j,k,l}$ is independent of $n$.  
Summing over all $j,k,l$ provides the bound 
\begin{equation} 
    \begin{split}
        \bigg\| \sum_{j=1}^J\sum_{k,l=1}^{m_j} c_{j,k,l}\varphi_{j,k}(x) &\int_0^1 \psi_{j,l}(z)\left(\frac{-\Xi(W_n,W,u^*)v}{Q(z)}\right)\mathrm{d}y\mathrm{d}z\bigg\|_\infty \\
        &\leq \sum_{j=1}^J\sum_{k,l=1}^{m_j} \left(
    \varepsilon_{j,k,l}+C_{j,k,l}  \|W-W_n\|_\square\right) \|v\|_\infty.
    \end{split}
\end{equation}

It now remains to estimate the contribution stemming from  the Neumann series portion of $DF(u^*)^{-1}$; see again (\ref{eq:Minvgen}). 
\begin{equation}\label{eq:MinvKeyPart}
    \frac{1}{1+\xi}\sum_{k=0}^\infty \left(\frac{T_K+\xi}{1+\xi}\right)^k \tilde{P} \left(\frac{-\Xi(W_n,W,u^*)v}{Q(\cdot)}\right). 
\end{equation}
Using the decomposition of the operator obtained in Lemma~\ref{lem:mbound} we have
\[ 
    \frac{1}{1+\xi}\sum_{k=0}^\infty \left(\frac{T_K+\xi}{1+\xi}\right)^k \tilde{P}w=\tilde{P}w+H[T_Kw],
\]
with $H$ a bounded operator on $C[0,1]$ (recall additionally that $T_K:X_n\to C[0,1]$ is bounded) with $\|Hw\|_\infty\leq b \|w\|_\infty$ for all $w \in C[0,1]$. Expanding the action of $T_K$ gives
\[ 
    T_K \left(\frac{-\Xi(W_n,W,u^*)v}{Q(\cdot)} \right)=\int_0^1\int_0^1  \frac{W(x,z)}{Q(x)} D_2(u^*(x),u^*(z))\frac{W_n(z,y)-W(z,y)}{Q(z)} D_2(u^*(z),u^*(y))v(y)\mathrm{d}y\mathrm{d}z. 
\]
Then Lemma~\ref{lem:W-Wnappliedtwice} gives that for any $\varepsilon_T>0$, there exists $C_T$ that depends on $\varepsilon_T > 0$ (we again suppress the dependence of $C_T$ on the terms appearing in the previous integrand), but not $n$, so that 
\[ 
    \left\| H\left[T_K\left(\frac{-\Xi(W_n,W,u^*)v}{Q(\cdot)} \right)\right]\right\|_\infty \leq \left(b\varepsilon_T+bC_T \|W-W_n\|_\square\right) \|v\|_\infty. 
\]
Finally, we have only to consider 
\[ 
    \tilde{P}\left(-\frac{\Xi(W_n,W,u^*)v}{Q(\cdot)}\right) =-\frac{\Xi(W_n,W,u^*)v}{Q(\cdot)} +P\left(\frac{\Xi(W_n,W,u^*)v}{Q(\cdot)}\right). 
\]
This first term is precisely the leading order estimate of $DF(u^*)^{-1}\left[ DF(u^*)-DF_n(u) \right]$ provided in the statement of the lemma that we wish to obtain. We therefore focus on the second part of the expression. 

Working term-by-term in the definition of $P$, we again apply Lemma~\ref{lem:W-Wnappliedtwice} and obtain, for any $\varepsilon_{j,k}>0$ there exists $C_{j,k}$, depending on $\varepsilon_{j,k}$ but not $n$, so that 
\[ 
    \left\|P\left(\frac{\Xi(W_n,W,u^*)v}{Q(\cdot)}\right)\right\|_\infty \leq \sum_{j=1}^J \sum_{k=1}^{m_j} \left(\varepsilon_{j,k}+C_{j,k}\|W-W_n\|_\square \right) \|v\|_\infty 
\]
Hence, combining all of our estimates we have that the expression \eqref{eq:DFinvDF-DFn} holds with 
\begin{eqnarray} \label{eq:Rvestimate}
    \|R_u[v]\|&\leq& \left( m\varepsilon_D+\sum_{j,k,l} \varepsilon_{j,k,l} +b\varepsilon_T+\sum_{j,k}\varepsilon_{j,k}\right)\|v\|_\infty +mC_1(\varepsilon_D,D_1)\|d_W-d_{W_n}\|_\infty)\|v\|_\infty  \nonumber \\
    &+&   \left( \sum_{j,k,l} C_{j,k,l} +bC_T+\sum_{j,k}C_{j,k}\right)\|W-W_n\|_\square \|v\|_\infty+3Lm\|u-u^*\|_\infty\|v\|_\infty.
\end{eqnarray}
Moreover, for any $\varepsilon>0$  we can select the (finite collection of) constants $\varepsilon_D$, $\varepsilon_{j,k,l}$, $\varepsilon_T$, $\varepsilon_{j,k} > 0$ in a manner such that  
\[ 
    \left( m\varepsilon_D+\sum_{j,k,l} \varepsilon_{j,k,l} +b\varepsilon_T+\sum_{j,k}\varepsilon_{j,k}\right)<\varepsilon.
\]
With these values fixed, the constant $C_1$ and all constants $C_{j,k,l}$, $C_T$ and $C_{j,k}$ are fixed quantities.  We therefore aggregate the constants appearing in \eqref{eq:Rvestimate} to obtain the estimate \eqref{eq:Rbounds}, concluding the proof.
\end{proof}

\subsection{Analysis of $\mathcal{T}_n$}
\label{Analysis of T}\label{sec:Tanal}

Having presented the necessary preliminary results, we now return to our analysis of the operator $\mathcal{T}_n[u]$. In general, this operator will fail to be a contraction on $X_n$.  To establish that $\mathcal{T}_n$ is a contraction mapping it is required that its operator norm is strictly less than one.  However, as we now demonstrate, there is no reason that this bound should hold in general.  The problem arises due to the term
\[ 
    \frac{\Xi(W_n,W,u^*)}{Q(x)}(u_1-u_2) = \frac{1}{Q(x)}\int_0^1 \left[W_n(x,y)-W(x,y)\right]D_2(u^*(x),u^*(y))(u_1(y)-u_2(y))\mathrm{d} y,
\]
coming from Lemma~\ref{lem:DFdiffdef}, which arises in the expansion of $\mathcal{T}_n[u_1]-\mathcal{T}_n[u_2]$ for general $u_1,u_2 \in X_n$. Letting 
\begin{equation}  \label{eq:eta}
    \eta=\max\left\{ 1, \sup_{x\in[0,1]} \left(\frac{1}{Q(x)}\int_0^1 | D_2(u^*(x)),u^*(y))|\mathrm{d}y \right) \right\} 
\end{equation} 
it can be shown that $\eta$ provides a coarse bound on the operator norm of $\frac{\Xi(W_n,W,u^*)}{Q(\cdot)}$, which will not be less than one in general.  

Our main result regarding $\mathcal{T}_n$ is the following Lemma.   

\begin{lem}\label{lem:T maps Bp to Bnp} 
    Let $\eta$ be defined as defined in \eqref{eq:eta}. Then, there exists a positive $\rho^*\leq \frac{1}{2\eta}$ such that for any $\rho \in (0,\rho^*)$ there exists an $N(\rho) \geq 1$ such that for any $n\geq N(\rho)$ the operator $\mathcal{T}_n$ maps the ball $B_\rho(u^*) \subset X_n$ into the ball $B_{2\eta \rho}(u^*)\subset B_1(u^*) \subset X_n$, i.e. $\mathcal{T}_n[B_\rho(u^*)] \subseteq B_{2\eta \rho}(u^*)$.  
\end{lem}

\begin{proof}  
For an $n\geq 1$ and $\rho > 0$, consider a $u\in B_\rho(u^*)\subset X_n$. Using the mean-value theorem we can expand 
\begin{equation} \label{eq:Tminusustar}
    \begin{split}
        \mathcal{T}_n[u]-u^*&= u-u^* -DF(u^*)^{-1}F_n(u) \\ 
        &=u-u^* -DF(u^*)^{-1}F_n(u^*)-DF(u^*)^{-1}DF_n(z)(u-u^*) \\
        &= -DF(u^*)^{-1}F_n(u^*)+DF(u^*)^{-1}\left(DF(u^*)-DF_n(z)\right) (u-u^*),
    \end{split}
\end{equation}
for some $z\in B_\rho(u^*)$. We use the uniform bound for $DF(u^*)^{-1}$ guaranteed by Lemma~\ref{lem:mbound} and apply  Lemma~\ref{lem:Fnsmall} to obtain that for any $\varepsilon_1>0$ there exists a $N_1 = N_1(\varepsilon_1)$ such that for any $n\geq N_1$ it holds that 
\[ 
    \|DF(u^*)^{-1} F_n(u^*)\|_\infty\leq m \varepsilon_1. 
\]
Next, Lemma~\ref{lem:DFdiffdef} guarantees that for any $\varepsilon_2>0$ there exists $C_1(\varepsilon_2), C_2(\varepsilon_2) > 0$, independent of $n$, so that 
\begin{equation}
    \begin{split}
        \|DF(u^*)^{-1} & \left(DF(u^*)-  DF_n(z)\right) (u-u^*)\|_\infty \\ &\leq \left\| \frac{\Xi(W_n,W,u^*)}{Q(\cdot)}\right\|_{X_n\to X_n }\|u-u^*\|_\infty \\ &+\left(\varepsilon_2+C_1(\varepsilon_2)\|d_W-d_{W_n}\|_\infty + C_2(\varepsilon_2)\|W-W_n\|_\square\right)\|u-u^*\|_\infty \\
        &+ 3Lm\|z-u^*\|_\infty \|u-u^*\|_\infty
    \end{split}
\end{equation}
Recall from the beginning of this subsection that $\eta$ is defined in \eqref{eq:eta} so that

\begin{equation} \label{Xi less than eta}
    \left\| \frac{\Xi(W_n,W,u^*)}{Q(\cdot)}\right\|_{X_n\to X_n }\leq \eta.
\end{equation}

Thus, we will have derived that $\mathcal{T}_n:B_\rho\to B_{2\eta\rho}$ if we can show that
\begin{equation} \label{eq:Tinequality}
    m\varepsilon_1+\left(\varepsilon_2+C_1(\varepsilon_2)\|d_W-d_{W_n}\|_\infty + C_2(\varepsilon_2)\|W-W_n\|_\square \\
    + 3Lm\rho\right) \rho<\eta\rho, 
\end{equation}
where we have used the fact that $z\in B_\rho(u^*)$, making $\|z - u^*\|_\infty < \rho$.  
To guarantee that \eqref{eq:Tinequality} holds, take
\begin{equation}\label{eq:rhostar} 
    \rho^*\leq \min\left\{ \frac{1}{2\eta}, \frac{\eta}{9Lm}\right\}.  
\end{equation}
and consider any $\rho\in (0,\rho^*)$. Since $\varepsilon_1$ and $\varepsilon_2$ were arbitrary, we take them to satisfy $\varepsilon_1<\frac{\rho \eta}{6m}$ and $\varepsilon_2<\frac{\eta}{6}$. With $\varepsilon_2$ chosen $C_1(\varepsilon_2)$ and $C_2(\varepsilon_2)$ are fixed constants and we can take $N(\rho)\geq N_1(\varepsilon_1)$ so that inequality (\ref{eq:Tinequality}) is satisfied and the result follows.  
\end{proof}

\begin{rmk}
    If we were able to show that the operator norm of $\Xi(W_n,W,u^*)/Q$ on $X_n$ was less than one, potentially for all suitably large $n$, then it would be possible to show that $\mathcal{T}_n$ is a contraction. As argued at the beginning of this section  with the introduction of the operator $\mathcal{T}_n$ this in turn would give the existence of a solution to $F_n(u)=0$. As we will see in the examples studied in Section~\ref{sec:examples} this operator norm is not necessarily less than 1 and so we will proceed in the general setting of Lemma~\ref{lem:T maps Bp to Bnp} and turn our attention to the second iterate of $\mathcal{T}_n$ which we prove be a contraction on a sufficiently small ball centered at $u^*$ in $X_n$.  
\end{rmk}

\subsection{Analysis of $\mathcal{S}_n$}\label{sec:Sanal}

In the previous subsection we highlighted our inability to prove that $T_n$ is a contraction on any suitably small ball in $X_n$ centered at $u^*$. We therefore turn our attention to demonstrating that $\mathcal{S}_n = \mathcal{T}_n\circ\mathcal{T}_n$, the composition of $\mathcal{T}_n$ with itself, is a contraction mapping on $B_\rho(u^*)$ with sufficiently small $\rho > 0$ and large $n\geq 1$. On top of this result, we further prove that it is also a contraction on the larger ball $B_{2\eta\rho}(u^*)$ from Lemma~\ref{lem:T maps Bp to Bnp}. This will allow us to conclude in the next subsection that $\mathcal{T}_n$ has a unique fixed point in the smaller ball $B_\rho(u^*)$, thus achieving the goal in introducing the operator $\mathcal{T}_n$ at the onset of this section.

We begin with the following lemma.

\begin{lem} \label{lem:Scontraction}
    There exists a $\rho_S > 0$ such that for every $\rho\in (0,\rho_S)$ there exists a $N(\rho) \geq 1$ such that for any $n\geq N(\rho)$ the operator $\mathcal{S}_n:B_\rho(u^*)\cap X_n\to B_\rho(u^*)\cap X_n$ is a contraction mapping.    
\end{lem}

\begin{proof}
The proof is broken down into two components: first showing that $\mathcal{S}_n:B_\rho(u^*)\cap X_n\to B_\rho(u^*)\cap X_n$ is well-posed and second showing that it is a contraction. 

{\bf Well-posedness}:  $\mathcal{S}_n:B_{\rho}(u^*)\cap X_n\to B_{\rho}(u^*)\cap X_n$

For any $n\geq 1$ and $\rho > 0$, consider $u\in B_\rho(u^*)\cap X_n$. From Lemma~\ref{lem:T maps Bp to Bnp} we have that there exists a $\rho^* > 0$ so that if we restrict $\rho \in (0,\rho^*)$ and take $n$ sufficiently large we guarantee that  $\mathcal{T}_n[u]\in B_{2\eta\rho}(u^*)\subset B_1(u^*)$, where we recall that $\eta$ is defined in \eqref{eq:eta}. Keeping these restrictions on $\rho$ and $n$, we use the mean value theorem to arrive at the expansion
\begin{eqnarray}\label{Su - ustar}
    \mathcal{S}_n[u]-u^*&=&\mathcal{T}_n\circ \mathcal{T}_n[u]-u^*\nonumber \\
    &=&\mathcal{T}_n[u]-u^*-DF(u^*)^{-1}F_n(\mathcal{T}_n[u]) \nonumber \\
    &=& \mathcal{T}_n[u]-u^*-DF(u^*)^{-1}F_n(u^*)-DF(u^*)^{-1}DF_n(z_1)\left(\mathcal{T}_n[u]-u^*\right) \nonumber  \\
    &=& [I - DF(u^*)^{-1}DF_n(z_1)](\mathcal{T}_n[u] - u^*) - DF(u^*)^{-1}F_n(u^*) \nonumber \\
    &=& DF(u^*)^{-1}\left[DF(u^*)-DF_n(z_1)\right]\left(\mathcal{T}_n[u]-u^*\right) - DF(u^*)^{-1}F_n(u^*) 
\end{eqnarray}
for some $z_1\in B_{2\eta\rho}(u^*)$. 

First, recall that there exists a $z_2\in B_\rho(u^*)$ such that  
\begin{equation}
    \begin{split}
        \mathcal{T}_n[u]-u^*&= DF(u^*)^{-1}(DF(u^*)-DF_n(z_2))(u-u^*)-DF(u^*)^{-1}F_n(u^*) \\
        &= -\frac{\Xi(W_n,W,u^*)}{Q(\cdot)} (u-u^*) - DF(u^*)^{-1}F_n(u^*) +R_{z_2}[u-u^*], 
    \end{split}
\end{equation}
see the
expansion of $\mathcal{T}_n[u]$ derived in \eqref{eq:Tminusustar} combined with the estimate of $DF(u^*)^{-1}(DF(u^*)-DF_n(z_2))(u-u^*)$ provided in Lemma~\ref{lem:DFdiffdef}. 
Then 
\begin{equation}\label{eq:Sexp}
    \begin{split}
        \mathcal{S}_n[u]-u^*&= \frac{\Xi(W_n,W,u^*)}{Q(\cdot)}\frac{\Xi(W_n,W,u^*)}{Q(\cdot)}(u-u^*) +\frac{\Xi(W_n,W,u^*)}{Q(\cdot)} DF(u^*)^{-1}F_n(u^*) \\
        &\quad - \frac{\Xi(W_n,W,u^*)}{Q(\cdot)}R_{z_2}[u-u^*] +R_{z_1}\left[\mathcal{T}_n[u]-u^*\right] - DF(u^*)^{-1}F_n(u^*) 
    \end{split}
\end{equation}
Consider the first term which consists of two-fold application of the operator $\frac{\Xi(W,W_n,u^*)}{Q(\cdot)}$. To condense notation, we will write $\tilde{W}_n(x,y)=W_n(x,y)-W(x,y)$. Then, 
\[  
    \begin{split}
    &\frac{\Xi(W_n,W,u^*)}{Q(\cdot)}\frac{\Xi(W_n,W,u^*)}{Q(\cdot)}(u-u^*) \\
    &= \int_0^1 \tilde{W}_n(x,z)\frac{D_2(u^*(x),u^*(z))}{Q(x)} \int_0^1 \tilde{W}_n(z,y)\frac{D_2(u^*(z),u^*(y))}{Q(z)}\left(u(y)-u^*(y)\right) \mathrm{d} y\mathrm{d} z.
    \end{split}
\]
The continuity of $D_2(u^*(z),u^*(y))$ allows for the application of Lemma~\ref{lem:W-Wnappliedtwice}, which implies that for any $\varepsilon_1>0$ there exists a constant $C_\Xi(\varepsilon_1)$ 
(again suppressing the functional dependence), independent of $n$, such that  
\begin{equation} \label{eq:Xi Squared}
    \left\| \left(\frac{\Xi(W,W_n,u^*)}{Q(x)} \right)\left( \frac{\Xi(W,W_n,u^*)}{Q(z)} \right) (u-u^*)\right\|_\infty\leq \left(\varepsilon_1+C_\Xi(\varepsilon_1) \|W-W_n\|_\square \right) \|u-u^*\|_\infty.  
\end{equation}

We now simultaneously consider the second and final terms on the right hand side of Equation \eqref{eq:Sexp}.  From Lemma~\ref{lem:Fnsmall} we have that there exists an $\varepsilon_2>0$ and a $N_2 = N_2(\varepsilon_2)$ such that for any $n\geq N_2$ we have
\[ 
    \left\|\left(I-\frac{\Xi(W,W_n,u^*)}{Q(x)} \right) DF(u^*)^{-1}F_n(u^*) \right\|_\infty \leq (1+\eta) m \varepsilon_2,   
\]
where we have used that $\eta$, defined in \eqref{eq:eta}, is a bound for the operator norm of $\frac{\Xi(W_n,W,u^*)}{Q(\cdot)}$ on $X_n$.
Next, for any $\varepsilon_3>0$, Lemma~\ref{lem:DFdiffdef} gives that there exists constants $C_1(\varepsilon_3)$ and $C_2(\varepsilon_3)$ such that 
\[ 
    \begin{split}
        &\left\|\left(\frac{\Xi(W,W_n,u^*)}{Q(x)} \right) R_{z_2}[u-u^*] \right\|_\infty \\ &\quad \leq \eta \left(\varepsilon_3+C_1(\varepsilon_3)\|d_W-d_{W_n}\|_\infty + C_2(\varepsilon_3)\|W-W_n\|_\square + 3Lm\|z_2-u^*\|_\infty\right) \|u-u^*\|_\infty. 
    \end{split}
\]
Finally, we have that 
\[ 
    \begin{split}
        &\left\|R_{z_1}\left[\mathcal{T}_n[u]-u^*\right]]\right\|_\infty \\ 
        & \quad \leq \left(\varepsilon_4+C_1(\varepsilon_4)\|d_W-d_{W_n}\|_\infty + C_2(\varepsilon_4)\|W-W_n\|_\square + 3Lm\|z_1-u^*\|_\infty\right) \|\mathcal{T}_n[u]-u^*\|_\infty.  
    \end{split}
\]
Combining these estimates and recalling that $u\in B_\rho(u^*) \cap X_n$, $z_1\in B_{2\eta \rho}(u^*) \cap X_n$, $z_2\in B_\rho(u^*) \cap X_n$ and $\mathcal{T}_n[u]\in B_{2\eta\rho}(u^*) \cap X_n$ we get 
\begin{equation}
    \begin{split}
    \|\mathcal{S}_n[u]-u^*\|_\infty&\leq \varepsilon_1\rho+(1+\eta) m \varepsilon_2+\eta\varepsilon_3\rho+2\varepsilon_4\eta \rho \\
    &+ (\eta \rho C_1(\varepsilon_3) +2\eta\rho C_1(\varepsilon_4)) \|d_W-d_{W_n}\|_\infty \\
    &+ (\rho C_\Xi(\varepsilon_1)+\eta \rho C_2(\varepsilon_3) +2\eta\rho C_2(\varepsilon_4)) \|W-W_n\|_\square \\
    &+ (3Lm+12L\eta^2m) \rho^2
    \end{split}  
\end{equation}
Select $\rho_S > 0$ so that  
\begin{equation}\label{eq:rhoS} 
    \rho_S < \frac{1}{2}\min\left\{ \rho^*,\frac{1}{3(3Lm +12L\eta^2 m)},\frac{1}{9Lm(1+\eta+2\kappa\eta)}\right\},  
\end{equation}
where $\kappa>3Lm\rho^*$ is a fixed constant arising in (\ref{eq:coarseRz2bound}).

Now consider any $\rho\in (0,\rho_S)$. Take $\varepsilon_{1,2,3,4} > 0$ so that $\varepsilon_1\rho+(1+\eta) m \varepsilon_2+\eta\varepsilon_3\rho+2\varepsilon_4\eta\rho<\frac{\rho}{3}$.  Then with $\varepsilon_{1,2,3,4}(\rho)$ fixed, we can select $N(\rho)$ sufficiently large so that for any $n\geq N(\rho)$ it holds that 
\[ 
    \|\mathcal{S}_n[u]-u^*\|_\infty\leq \frac{\rho}{3}+\frac{\rho}{3}+\frac{\rho}{3}=\rho 
\]
for all $u \in B_{\rho}(u^*)\cap X_n$. Thus, $\mathcal{S}_n: B_{\rho}(u^*)\cap X_n \to B_{\rho}(u^*)\cap X_n$ is well-defined.

{\bf Contraction}: We now proceed to show that $\mathcal{S}_n$, for $\rho$ suffiently small and $n$ sufficiently large,  is a contraction on $B_{\rho}(u^*) \cap X_n$. As above, we will consider $\rho\in (0,\rho_S)$ and let $u_{1,2}\in B_{\rho}(u^*) \cap X_n$.  Then, consider 
\[ 
    \left\| \mathcal{S}_n[u_1]-\mathcal{S}_n[u_2]\right\|_\infty = \left\| DF(u^*)^{-1}\left[DF(u^*)(\mathcal{T}_n[u_1]-\mathcal{T}_n[u_2]) - F_n(\mathcal{T}_n[u_1]) +F_n(\mathcal{T}_n[u_2])\right]\right\|_\infty.  
\]
The mean value theorem guarantees that there exists some $z_1\in B_{2\eta \rho}(u^*)$ such that 
\[ 
    F_n(\mathcal{T}_n[u_1]) - F_n(\mathcal{T}_n[u_2])=DF_n(z_1) \left( \mathcal{T}_n[u_1]-\mathcal{T}_n[u_2] \right),
\]
so that we then have 
\[ 
    \mathcal{S}_n[u_1]-\mathcal{S}_n[u_2] = DF(u^*)^{-1}\left(DF(u^*)-DF_n(z_1)\right) (\mathcal{T}_n[u_1]-\mathcal{T}_n[u_2]) .   
\]
Similarly, we have that there exists a $z_2\in B_\rho(u^*)$ such that 
\[ 
    \mathcal{T}_n[u_1]-\mathcal{T}_n[u_2] = DF(u^*)^{-1}\left(DF(u^*)-DF_n(z_2)\right)(u_1-u_2).  
\]
Using Lemma~\ref{lem:DFdiffdef}, we then have that
\[ 
    \mathcal{S}_n[u_1]-\mathcal{S}_n[u_2] = \left(\left(\frac{\Xi(W,W_n,u^*)}{Q(\cdot)} \right)+R_{z_1}\right)\left(\left(\frac{\Xi(W,W_n,u^*)}{Q(\cdot)} \right)+R_{z_2}\right)(u_1-u_2). 
\]

The analysis now resembles that performed above to demonstrate well-posedness. Indeed, for any $\varepsilon_1>0$ , there exist a constant $C_\Xi(\varepsilon_1)$, independent of $n$, such that 
\[ 
    \left\| \left(\frac{\Xi(W,W_n,u^*)}{Q(x)} \right)\left( \frac{\Xi(W,W_n,u^*)}{Q(z)} \right) (u_1-u_2)\right\|_\infty\leq \left(\varepsilon_1+C_\Xi(\varepsilon_1) \|W-W_n\|_\square \right) \|u_1-u_2\|_\infty.  
\]
For any $\varepsilon_2>0$ , there exist constants $C_1(\varepsilon_2)$ and $C_2(\varepsilon_2)$, both independent of $n$, such that 
\[ 
    \begin{split}
        &\left\|\left(\frac{\Xi(W,W_n,u^*)}{Q(x)} \right) R_{z_2}[u_1-u_2] \right\|_\infty \\
        & \quad \leq \eta \left(\varepsilon_2+C_1(\varepsilon_2)\|d_W-d_{W_n}\|_\infty + C_2(\varepsilon_2)\|W-W_n\|_\square + 3Lm\|z_2-u^*\|_\infty\right) \|u_1-u_2\|_\infty. 
    \end{split}
\]
And again, for any $\varepsilon_3>0$ , there exist constants $C_1(\varepsilon_3)$ and $C_2(\varepsilon_3)$, both independent of $n$, such that
\[ 
    \begin{split}
        &\left\| R_{z_2} \left(\frac{\Xi(W,W_n,u^*)}{Q(x)} \right)[u_1-u_2] \right\|_\infty \\ 
        & \quad\leq  \eta \left(\varepsilon_3+C_1(\varepsilon_3)\|d_W-d_{W_n}\|_\infty + C_2(\varepsilon_3)\|W-W_n\|_\square + 3Lm\|z_1-u^*\|_\infty\right) \|u_1-u_2\|_\infty. 
    \end{split}
\]

The final term involves the composition of the operators: $R_{z_1}( R_{z_2}[u_1-u_2])$, which each satisfy estimate 
 \[  \left\| R_u[v]\right\|_\infty\leq \left(\varepsilon+C_1(\varepsilon)\|d_W-d_{W_n}\|_\infty + C_2(\varepsilon)\|W-W_n\|_\square + 3Lm\|u-u^*\|_\infty\right) \|v\|_\infty,  \]
recall \eqref{eq:Rbounds}. 
Using that $z_2\in B_{\rho}(u^*)\subset B_{\rho^*}(u^*)$ we can obtain -- for a fixed $\varepsilon_b>0$ and $N(\varepsilon_b)$ sufficiently large that there exists $\kappa\geq 4Lm\rho^*$ such that 
\begin{equation} \label{eq:coarseRz2bound}
    \|R_{z_2}[u_1-u_2]\|_\infty \leq \kappa \|u_1-u_2\|_\infty. 
\end{equation}
Then for any $\varepsilon_4>0$ the following estimate holds 
\begin{equation}
\begin{split}
     &\left\| R_{z_1}( R_{z_2}[u_1-u_2]) \right\|_\infty
     \\ &\leq \left(\varepsilon_4+C_1(\varepsilon_4)\|d_W-d_{W_n}\|_\infty + C_2(\varepsilon_4)\|W-W_n\|_\square + 3Lm\|z_1-u^*\|_\infty\right)\left\|R_{z_2}[u_1-u_2]\right\|_\infty \\
     & \leq \kappa \left(\varepsilon_4+C_1(\varepsilon_4)\|d_W-d_{W_n}\|_\infty + C_2(\varepsilon_4)\|W-W_n\|_\square + 3Lm\|z_1-u^*\|_\infty\right)\left\|u_1-u_2\right\|_\infty
     \end{split}
\end{equation}

Aggregating the estimates above and recalling that $z_1\in B_{2\eta\rho}(u^*)$ while $z_2\in B_{\rho}(u^*)$ we arrive at
\begin{equation}
    \begin{split}
    \left\|\mathcal{S}_n[u_1]-\mathcal{S}_n[u_2]\right\|_\infty &\leq \left(\varepsilon_1+\eta\varepsilon_2+\eta\varepsilon_3+\kappa \varepsilon_4 \right) \|u_1-u_2\|_\infty \\
    &+ \left(\eta C_1(\varepsilon_2)+\eta C_1(\varepsilon_3)+\kappa C_1(\varepsilon_4)\right)\|d_W-d_{W_n}\|_\infty \|u_1-u_2\|_\infty \\
    &+ \left(C_\Xi(\varepsilon_1)+\eta C_2(\varepsilon_2)+\eta C_2(\varepsilon_3)+\kappa C_2(\varepsilon_4)\right) \|W-W_n\|_\square \|u_1-u_2\|_\infty \\
    &+ 3Lm(1+2\eta+2\kappa\eta) \rho  \|u_1-u_2\|_\infty.
    \end{split}
\end{equation}
To obtain a contraction on $B_\rho(u^*)$ we first notice; 
\[ 3Lm(1+2\eta +2\kappa \eta)\rho<\frac{1}{3},\]
see again (\ref{eq:rhoS}).    Next, we select 
$\varepsilon_{1,2,3,4}$ in such a way that 
\[ \varepsilon_1+\eta\varepsilon_2+\eta\varepsilon_3+\kappa \varepsilon_4<\frac{1}{3}. \]
With these quantities fixed, the constants $C_\Xi(\varepsilon_1)$ and  $C_{1,2}(\varepsilon_j)$ become fixed and therefore by taking $N(\rho)$ sufficiently large we can guarantee that 
\[ 
    \left\|\mathcal{S}_n[u_1]-\mathcal{S}_n[u_2]\right\|_\infty \leq  \frac{3}{4}  \|u_1-u_2\|_\infty,
\]
showing that $\mathcal{S}_n$ is a contraction mapping on the ball $B_\rho(u^*) \cap X_n$.
\end{proof}

The contraction mapping theorem in conjunction with the previous lemma immediately implies the existence of a unique fixed point of the operator $\mathcal{S}_n$ in the ball $B_\rho(u^*)$.  We state this fact as the following corollary. 

\begin{cor} \label{Banach FPT p}
    There exists a $\rho_S > 0$ such that for any $\rho \in (0,\rho_S)$ there exists a $N(\rho) \geq 1$ such that for any $n\geq N$ there exists a unique function $u_n^*(x)\in X_n$ which is a fixed point of the mapping $\mathcal{S}_n[u_n]=u_n$ and satisfies
    \[ 
        \sup_{x\in [0,1]} | u_n^*(x)-u^*(x)|< \rho.  
    \]
\end{cor}

\subsection{Finite-dimensional solution }\label{sec:proofwrapup}

We have now established, for all $n$ sufficiently large, that there exists a $u_n^*\in X_n$ which is a fixed point of $\mathcal{S}_n[u]$. From the contraction mapping theorem, this fixed point is unique in a small ball centered at $u^*$ with respect to the sup norm.  We will now show that the existence of this fixed point implies the existence of a steady-state for the finite-dimensional problem $G_n(\mathbf{u},A)$.  This involves two steps: first showing that the fixed point of $\mathcal{S}_n$ is also a fixed point of $\mathcal{T}_n$ and then verifying that this function is constant on each sub-interval in the partition of $[0,1]$.

\begin{lem}\label{lem:unfixedpointT}
    There exists a $\rho_T>0$ such that for any $\rho\in (0,\rho_T)$ there exists an $N(\rho) \geq 1$ such that for any $n\geq N(\rho)$ there exists a unique function $u_n^*\in X_n$ satisfying $\mathcal{T}_n[u_n^*]=u_n^*$ and
    \[ 
        \sup_{x\in [0,1]} | u_n^*(x)-u^*(x)|< \rho.  
    \]
    As a consequence, $u_n^*$ is such that $F_n(u_n^*)=0$.
\end{lem}

\begin{proof}
Let $\rho_S > 0$ be the bound on the ball radius guaranteed by Lemma~\ref{lem:Scontraction} and let $\rho_T>0$ be taken small enough to satisfy $2\eta\rho_T<\rho_S$, where we recall $\eta$ is defined in \eqref{eq:eta}. From Corollary~\ref{Banach FPT p} we have a unique function $u_n^*\in B_{\rho_T}(u^*)$ such that $\mathcal{S}_n[u_n^*] = \mathcal{T}_n[\mathcal{T}_n[u_n^*]] = u_n^*$ for all $n$ sufficiently large. There are now two cases: (1) $u_n^*$ is a fixed point of $\mathcal{T}_n$, yielding the result, or (2) $u_n^*$ is not a fixed point of $\mathcal{T}_n$. 

We proceed by assuming case (2) for the purpose of contradiction. Set $\mathcal{T}_n[u_n^*] = v_n^*$ with the assumption that $v_n^* \neq u_n^*$. Then, Lemma~\ref{lem:T maps Bp to Bnp} gives that $v_n^* \in B_{2\eta\rho_T}(u^*)$ because $u_n^* \in B_{\rho_T}(u^*)$. Using the fact that $u_n^*$ is a fixed point of $\mathcal{S}_n = \mathcal{T}_n\circ \mathcal{T}_n$, we have that $\mathcal{T}_n[v_n^*] = u_n^*$ as well. Hence,
\[
    \mathcal{S}_n[v_n^*] = \mathcal{T}_n[\mathcal{T}_n[v_n^*]] = \mathcal{T}_n[u_n^*] = v_n^*, 
\]
showing that $v_n^*$ is also a fixed point of $\mathcal{S}_n$. However,  $\mathcal{S}_n$ is a contraction on $B_{2\eta\rho_T}(u^*)$ for all $n$ sufficiently large; see again Lemma~\ref{lem:Scontraction}. Therefore, the contraction mapping theorem guarantees the uniqueness of the fixed point of $\mathcal{S}_n$ on the larger ball $B_{2\eta\rho_T}(u^*)$, meaning that $u_n^*=v_n^*$. Thus, only possibility (1) above remains, meaning that $u_n^*$ is a fixed point of $\mathcal{T}_n$. 

Lastly, from the definition of $\mathcal{T}_n$, a fixed point of $\mathcal{T}_n$ is a solution of $DF(u^*)^{-1}F_n(u_n^*)=0$. Since Corollary~\ref{cor:DFinvertible} proved that $DF(u^*)$ is invertible on $X_n$ for all $n\geq 1$, this implies $F_n(u_n^*)=0$.  
\end{proof}

At this point, we only know that the solution $u_n^*\in X_n$ is continuous on each sub-interval $[(i-1)/n,i/n)$ of $[0,1]$. As a final piece of the proof of Theorem~\ref{thm:existence} we will show that this solution is constant on each sub-interval and so the solution $F_n(u_n^*)=0$ corresponds to a solution  of the  original finite-dimensional problem $G_n(\mathbf{u}_n^*,A_n)=0$ with $\mathbf{u}_n^*\in\mathbb{R}^n$ the vector of the values on the steps of $u_n^*$. 

\begin{lem}\label{lem:PiecewiseConstant} 
    The fixed point $u_n^*\in X_n$ of $\mathcal{T}_n$, guaranteed by Lemma~\ref{lem:unfixedpointT} for sufficiently large $n$, is piecewise constant on each interval $\left[\frac{i-1}{n},\frac{i}{n}\right)$.
\end{lem}

\begin{proof}
We will argue by contradiction. Fix $n$ sufficiently large so that $u_n^*\in X_n$ is the steady-state solution guaranteed by Lemma~\ref{lem:unfixedpointT}. Suppose that for some $n$, there exists a sub-interval $J_* := \left[\frac{i-1}{n},\frac{i}{n}\right)$ for which $u_n^*$ is not constant on it.  

Since $u_n^*$ solves $F_n(u_n) = 0$, it follows that for each $a\in J_*$ we have 
\[ 
    0=f(u_n^*(a))+\int_0^1 W_n(a,y) D(u_n^*(a),u_n^*(y))\mathrm{d} y. 
\]
Now, recall from the definition of $X_n$ that $u_n^*$ is continuous on $J_*$ and that $W_n(a,y)=W_n(b,y)$ for all $a,b\in J_*$. Since we have assumed that $u_n^*$ is non-constant on $J_*$, we may further suppose that $u_n^*(a)$ is neither a local maximum nor minimum of $u_n^*$. Therefore, there exists a $\delta>0$ such that 
\[ 
    0=f(\zeta)+\int_0^1 W_n(a,y) D(\zeta,u_n^*(y))\mathrm{d} y
\]
for any $\zeta\in [u_n^*(a)-\delta,u_n^*(a)+\delta]$.

Let us now define the function 
\[ 
    \Pi(\zeta)=f(\zeta)+\int_0^1 W_n(a,y) D(\zeta,u_n^*(y))\mathrm{d} y,
\]
which from above satisfies $\Pi(\zeta)=0$ for all $\zeta \in [u_n^*(x_1)-\delta,u_n^*(x_1)+\delta]$.  In turn, this implies that $\frac{\mathrm{d} \Pi}{\mathrm{d} \zeta}=0$ on this interval, where differentiability of $\Pi$ with respect to $\zeta$ is a consequence of Hypothesis~\ref{hyp:Smooth}. The derivative $\frac{\mathrm{d} \Pi}{\mathrm{d} \zeta}$ is further found to be
\[ 
    \frac{d \Pi}{d\zeta} =f'(\zeta)+\int_0^1 W_n(x_1,y) D_1(\zeta,u_n^*(y))\mathrm{d} y.
\]
Precisely, evaluating $\frac{\mathrm{d} \Pi}{\mathrm{d} \zeta}$ at $u_n^*(a)$ is exactly $Q_n(u_n^*(a))$, which is proven, for $n$ sufficiently large to be non-zero in Lemma~\ref{lem:Qsmall}. Thus, we have a contradiction, which proves the claim.
\end{proof}

We now summarize our contributions so far with the following proof of Theorem~\ref{thm:existence}.

\begin{proof}[Proof of Theorem~\ref{thm:existence}]
Recall the definitions of $\rho_S$ from Lemma~\ref{lem:Scontraction} and recall that the constant $\rho_T$ is selected such that $2\eta\rho_T<\rho_S$.  
For any $\rho\in (0,\rho_T)$, there exists an $N(\rho) \geq 1$ such that for any $n\geq N(\rho)$ the contraction mapping result in Lemma~\ref{lem:Scontraction} gives the existence of a fixed point of $\mathcal{S}_n$, as summarized in Corollary~\ref{Banach FPT p}.  Denoting these fixed points as $u_n^*(x) \in X_n$, we have that they are locally unique functions that satisfy $F_n(u_n^*)=0$, again for any $n\geq N(\rho)$ (see Lemma~\ref{lem:unfixedpointT}).  These functions are shown to be piece-wise constant (perhaps by restricting to $N(\rho)$ larger) in Lemma~\ref{lem:PiecewiseConstant} and therefore correspond to vector solutions of the finite-dimensional problem $G_n(\unstar,A_n)=0$. This concludes the proof.
\end{proof}

\section{Proof of Theorem~\ref{thm:stability}}\label{sec:proof2}

In this section we prove Theorem~\ref{thm:stability}. Throughout we will denote the steady-state $\unstar \in \mathbb{R}^n$ to be that which is guaranteed by Theorem~\ref{thm:existence}. Then, our approach in this section will be to first study the stability of $u_n^*\in X_n$, the step function version of $\unstar$ on $[0,1]$, as a solution of the non-local problem $F_n(u_n^*)=0$ and leverage this stability result to the discrete finite-dimensional setting.  Recall that our standing assumption in this section is that $DF(u^*)$ as an operator on $C[0,1]$ has spectrum which is contained entirely in the left half of the complex plane and bounded away from the imaginary axis.  

The linearization of $F_n(u)$ at the steady-state $u_n^*$ acts on $v \in C[0,1]$ according to 
\begin{equation}
    DF_n(u_n^*)v=-Q_n(u_n^*(x))v+\int_0^1 W_n(x,y)D_2(u_n^*(x),u_n^*(y))v(y)\mathrm{d}y, 
\end{equation}
where we recall from the previous section that 
\begin{equation} 
    Q_n(u_n^*(x))=-f'(u_n^*(x))-\int_0^1 W_n(x,y)D_1(u_n^*(x),u_n^*(y))\mathrm{d}y. 
\end{equation}
Our goal is to use spectral convergence results to show that the spectrum of $DF_n(u_n^*)$ is close to that of $DF(u^*)$ for $n$ sufficiently large.  The space $X_n$ turns out to be insufficient for this purpose. This, once again, stems from the fact the $L^\infty \to L^\infty$ operator norm of $\Xi(W_n,W,u^*)$ does not tend to zero as $n \to \infty$. Instead, we will consider the spectrum of $DF_n(u_n^*)$ as an operator on $L^2$ where we have estimates such as \eqref{eq:TWcutnormbounds} available. To carry out this argument we will first show that the spectrum of $DF(u^*)$ is the same on both $C[0,1]$ and $L^2$; see also \cite[Lemma 3.4]{delmas22} for a similar finding.  Operator norm convergence of $DF_n(u_n^*)$ to $DF(u^*)$ is then obtained in $L^2$. This ultimately allows us to conclude spectral stability of $DF_n(u_n^*)$ based upon stability of $DF(u^*)$, finally also obtaining stability of the spectrum of the discrete operator $DG_n(\unstar)$.   


\begin{lem} \label{lem:DF on C = DF on L2}
    The spectrum of the operator $DF(u^*)$ posed on $L^2$ is equivalent to the spectrum on $C[0,1]$, i.e.
    \[
    \spec(DF(u^*))|_{C[0,1]} = \spec(DF(u^*))|_{L^2}.
    \]
\end{lem}

\begin{proof}

We begin with the essential spectrum.  Recall that in Lemma~\ref{lem:fredholm} 
we showed that $\lambda \in \spec_{ess}(DF(u^*))|_{C[0,1]}$ if and only if $\lambda \in \mathrm{Rng}(-Q(x))$. This characterization of the essential spectrum also carries over to the space $L^2$, which for completeness we now demonstrate.

Suppose that $\lambda+Q(x)\neq 0$, meaning $\lambda \notin \mathrm{Rng}(-Q(x))$. Since $Q(x)$ is continuous on $[0,1]$, it follows that the function $\frac{1}{\lambda+Q(x)}$ is bounded. Thus, the multiplication operator $v\to -(\lambda+Q(x)) v$ is invertible on $L^2$ and therefore Fredholm with index zero.  The remaining integral portion of the operator $DF(u^*)$ is a Hilbert--Schmidt integral operator and therefore compact. Since the Fredholm index is preserved under compact perturbations, it follows that if $\lambda \notin \mathrm{Rng}(-Q(x))$ then $DF(u^*) - \lambda$ is Fredholm with index zero and $\lambda$ is not an element of the essential spectrum.  

Suppose now that there exists a $c\in [0,1]$ such that $-(\lambda +Q(c))=0$. We will now verify that the multiplication operator $Q_\lambda v = -(\lambda + Q(\cdot))v$ on $L^2$ is not Fredholm for this choice of $\lambda$.  For any $\varepsilon>0$, let $I_\varepsilon= (c-\varepsilon,c+\varepsilon) \cap [0,1]$ and let $\chi_\varepsilon:[0,1] \to \mathbb{R}$ be the indicator function associated to this interval. Then, since $Q(x)$ is continuous the mean value theorem for integrals gives that for any $\varepsilon$ sufficiently small it holds that 
\begin{equation} 
    \frac{\|Q_\lambda\chi_\varepsilon\|_2^2}{\|\chi_\varepsilon\|_2^2} = \frac{1}{2\varepsilon}\int_{c-\varepsilon}^{c+\varepsilon} \left(Q(y)+\lambda\right)^2\mathrm{d}y = \left(Q(\theta(\varepsilon))+\lambda\right)^2,
\end{equation}
for some $\theta \in (c-\varepsilon,c+\varepsilon)$. The right hand side tends to zero as $\varepsilon\to 0$ showing that $Q_\lambda$ is not bounded from below. Therefore, the open mapping theorem gives that $Q_\lambda$ is not onto $L^2$. If the zero set of $Q_\lambda$ has zero measure then the nullspace of the multiplication operator $Q_\lambda$ is trivial, and so by the closed range theorem \cite[Theorem VII.5.1]{yosida}, it follows that the range of $Q_\lambda$ is not closed in $L^2$.  Alternatively, if there is an open interval over which $Q_\lambda(x)\equiv 0$ then the kernel of $Q_\lambda$ does not have finite dimension. This implies that $Q_\lambda$, and by extension $DF(u^*)-\lambda I$, are not Fredholm and thus $\lambda \in \sigma_{ess}(DF(u^*))|_{L^2}$ for any $\lambda\in \mathrm{Rng}(-Q(\cdot))$. Thus, $\sigma_{ess}(DF(u^*))|_{L^2} = \sigma_{ess}(DF(u^*))|_{C[0,1]}$.  

We now turn to the point spectrum of $DF(u^*)$ on $L^2$.  We will show that the point spectrum of $DF(u^*)$ is equivalent on $L^2$ and $C[0,1]$ by verifying that all eigenfunctions and generalized eigenfunctions in $L^2$ are also continuous.  Let $\lambda\in \spec_{pt}(DF(u^*))|_{L^2}$, which from above means that $\lambda \notin \mathrm{Rng}(-Q(x)).$ Then there exists an associated eigenfunction $v\in L^2$ such that 
\begin{equation}
    (-Q(x) - \lambda)v(x) + \int_0^1W(x,y)D_2(u^*(x),u^*(y))v(y)\mathrm{d}y = 0.
\end{equation}
Rearranging this expression means that $v$ satisfies 
\begin{equation} \label{eqn:v rearranged}
    v(x) = \frac{1}{Q(x) + \lambda}\int_0^1W(x,y)D_2(u^*(x),u^*(y))v(y)\mathrm{d}y.
\end{equation}
Since $\lambda \notin \mathrm{Rng}(-Q(x))$ and both $W$ and $D_2$ are bounded, we take the supremum of the right-hand side of \eqref{eqn:v rearranged} to bound $v$ pointwise by
\begin{align*}
    |v(x)| &\leq \sup_{x\in [0,1]} \left| \frac{1}{Q(x) + \lambda}\int_0^1W(x,y)D_2(u^*(x),u^*(y))v(y)\mathrm{d}y \right| \\
    &\leq \sup_{x\in [0,1]} \left| \frac{1}{Q(x) + \lambda}\right| \int_0^1 \left| W(x,y)D_2(u^*(x),u^*(y))v(y) \right| \mathrm{d}y \\
    &\leq \eta \|v\|_1\\
    &\leq \eta \|v\|_2
\end{align*}
for some $\eta>0$. In the above we have used the fact that functions in $L^2$ also belong to $L^1$ and satisfy the inequality $\|v\|_1 \leq \|v\|_2$, coming from the fact that $[0,1]$ has finite (Lesbegue) measure. Thus, we see that the eigenfunction $v\in L^2$ satisfies $\|v\|_\infty \leq \eta \|v\|_2$, showing that $v \in L^\infty$. We can further apply Corollary~\ref{cor:Qcont} to see that both $\frac{1}{Q(x)+\lambda}$ and $\int_0^1W(x,y)D_2(u^*(x),u^*(y))v(y)\mathrm{d}y$ are continuous functions of $x$. Given the representation for $v$ in \eqref{eqn:v rearranged}, we see that $v$ is a product of continuous functions and thus continuous. A similar argument works for generalized eigenfunctions and therefore $\spec_{pt}(DF(u^*))|_{L^2} \subseteq \spec_{pt}(DF(u^*))|_{C[0,1]}$.   

Since elements of $C[0,1]$ are bounded, they also belong to $L^2$ and so trivially we have that $\spec_{pt}(DF(u^*))|_{C[0,1]} \subseteq \spec_{pt}(DF(u^*))|_{L^2}$, which together with the above implies that $\spec_{pt}(DF(u^*))|_{L^2} = \spec_{pt}(DF(u^*))|_{C[0,1]}$. Since the spectrum is decomposed into the point and essential spectrum, we have shown that $\spec(DF(u^*))|_{L^2} = \spec(DF(u^*))|_{C[0,1]}$, concluding the proof.
\end{proof}


\begin{lem} \label{lem: 2,2 operator convergence}
    For any $\varepsilon>0$ there exists an $N\in \N$ such that for all $n\geq N$ we have
    \[
        \|DF_n(u^*_n) - DF(u^*)\|_{2\to 2} <\varepsilon.
    \]
\end{lem}

\begin{proof}
Using the triangle inequality we express this difference as
\begin{equation} \label{eqn: DFu - DFnUn Triangle Inequality}
    \|DF_n(u_n^*) - DF(u^*)\|_{2\to 2} \leq  \|DF_n(u_n^*) - DF_n(u^*)\|_{2\to 2} + \|DF_n(u^*) - DF(u^*)\|_{2 \to 2}.
\end{equation}
We will work to show that each term on the right hand side of \eqref{eqn: DFu - DFnUn Triangle Inequality} can be made small individually for large enough $n$, which will in turn prove the lemma.  

Letting $v \in L^2$, using the definition of $DF_n(u_n^*(x))$ and $DF_n(u^*(x))$ we get   
\begin{align*}
    [DF_n(u_n^*(x)) - DF_n(u^*(x))]v(x) &= [-f'(u_n^*(x)) + f'(u^*(x))]v(x)\\
    & \quad - v(x)\int_0^1W_n(x,y) [D_1(u^*_n(x),u^*_n(y)) - D_1(u^*(x),u^*(y))]\mathrm{d}y\\
    & \quad + \int_0^1W_n(x,y) [D_2(u^*_n(x),u^*_n(y)) - D_2(u^*(x),u^*(y))]v(y)\mathrm{d}y.
\end{align*}
Now, recall that Lemma~\ref{lem:unfixedpointT} gives the existence of a $\rho_T > 0$ so that for any $\rho\in (0,\rho_T)$ there exists a $N(\rho) \geq 1$ such that $\|u_n^*-u^*\|_\infty<\rho$ for all $n\geq N(\rho)$. Furthermore, Hypothesis~\ref{hyp:Smooth} imply that $f'$, $D_1,$ and $D_2$ are all Lipschitz continuous.  For simplicity of notation, we assume that all Lipschitz constants are simply $L > 0$. Thus, for any $n \geq N(\rho)$ with $\rho < \rho_T$, the following bounds are immediate:
\begin{align*}
    \|[-f'(u_n^*) + f'(u^*)]v\|_2 &\leq L\rho \|v\|_2 \\
    \left\| v(\cdot) \int_0^1W_n(\cdot,y) [D_1(u^*_n(\cdot),u^*_n(y)) - D_1(u^*(\cdot),u^*(y))]\mathrm{d}y \right\|_2 & \leq L \rho \|v\|_2
\end{align*}
and 
\begin{equation}
    \begin{split}
    \left\| \int_0^1W_n(\cdot,y) [D_2(u^*_n(\cdot),u^*_n(y)) - D_2(u^*(\cdot),u^*(y))]v(y)\mathrm{d}y \right\|_2 &\leq L \rho \left\| \int_0^1 v(y) \mathrm{d}y \right\|_2\\
    &= L \rho \left| \int_0^1 v(y) \mathrm{d}y \right|\\
    &\leq L \rho \left\| v \right\|_1\\
    &\leq L \rho \left\| v \right\|_2,
    \end{split}
\end{equation}
where, as in the proof of the previous lemma, we have used the fact that $\|v\|_1 \leq \|v\|_2$ since $[0,1]$ is a space of finite measure.

We now consider the second term of inequality \eqref{eqn: DFu - DFnUn Triangle Inequality}. Letting $v\in L^2$, we have
\begin{align} \label{eqn: (DFn - DF)v Stability}
    [DF_n(u^*(x)) - DF(u^*(x))]v(x) &= - v(x)\int_0^1[W_n(x,y) - W(x,y)] D_1(u^*(x),u^*(y))\mathrm{d}y \nonumber \\ 
    & \quad + \int_0^1 [W_n(x,y) - W(x,y)] D_2(u^*(x),u^*(y))v(y)\mathrm{d}y.
\end{align}
Observe that $D_1(u^*(x),u^*(y)) \in C([0,1]\times [0,1])$ and directly apply Lemma~\ref{lem:WWnappliedtocontissmall}. So, for any $\varepsilon>0$ there exists a $C_1>0$ such that
\begin{equation}
    \begin{split}
     \bigg\| v\int_0^1[W_n(\cdot,y) - &W(\cdot,y)] D_1(u^*(\cdot),u^*(y))\mathrm{d}y \bigg\|_2 \\ &\leq  \sup_{x\in [0,1]}\left|\int_0^1[W_n(x,y) - W(x,y)] D_1(u^*(x),u^*(y))\mathrm{d}y \right| \|v\|_2\\
     &< \left(\varepsilon + C_1 \|d_{W_n}-d_W\|_\infty\right) \|v\|_2.
    \end{split}
\end{equation}
To control the second term on the right-hand side of (\ref{eqn: (DFn - DF)v Stability}) we require an $L^2$ version of Lemma~\ref{lem:W-Wnappliedtwice} which we construct here. We will also make use of the bound \cite[Lemma~E.6]{janson2010graphons}
\begin{equation} \label{eqn:TW cutnorm bound}
    \|T_W\|_{2,2} \leq 2\sqrt{2}\|W\|_\square^{1/2},
\end{equation}
where $T_W:v\to \int_0^1 W(x,y)v(y)\mathrm{d}y$, as defined in \eqref{eq:2ndcutnorm}.  
We now proceed as in the proof of Lemma~\ref{lem:W-Wnappliedtwice}. 

Take any $\varepsilon>0$. Then, from the continuity of $D_2$ there exists an $M\in\N$ such that 
\begin{equation}\label{eq:D1 step}
    D_2(u^*(x),u^*(y)) = \sum_{i,j = 1}^M D_{ij}\zeta_i(x)\zeta_j(y) + \Delta D(x,y)
\end{equation}
with $\zeta_{i,j}$ defined in \eqref{eq:zeta} above and the remainder term bounded by $|\Delta D(x,y)| < \varepsilon.$ Taking any $x\in [0,1]$, we rewrite the last term of \eqref{eqn: (DFn - DF)v Stability} using the new expansion \eqref{eq:D1 step} to get
\begin{eqnarray*}
    \sum_{i,j=1}^M \int_0^1 [W(x,y)-W_n(x,y)]D_{ij}\zeta_i(x)\zeta_j(y)v(y)\mathrm{d}y + \int_0^1 [W(x,y) - W_n(x,y)] \Delta D(x,y)v(y)\mathrm{d}y.  
\end{eqnarray*}
The first term in the above can be bounded as
\begin{align*}
    \left\| \sum_{i,j=1}^M \int_0^1 [W(\cdot,y)-W_n(\cdot,y)]D_{ij}\zeta_i(\cdot)\zeta_j(y)v(y)\mathrm{d}y \right\|_2 &\leq M^2 \sup_{i,j} |D_{ij}| \|T_{W_n - W}\|_{2,2} \|v\|_2\\
    &\leq M^2 \sup_{i,j} |D_{ij}| 2\sqrt{2}\|W_n - W\|_\square^{1/2} \|v\|_2,
\end{align*}
where we have used \eqref{eqn:TW cutnorm bound}. The second term containing the remainder $\Delta D$ can be controlled directly using $|W(x,y) - W_n(x,y)|\leq 1$ and $|\Delta D(x,y)| < \varepsilon$, wherein one has
\begin{equation*}
    \left\| \int_0^1 [W(x,y) - W_n(x,y)] \Delta D(x,y)v(y)\mathrm{d}y \right\|_2 \leq \varepsilon \left\| \int_0^1 v(y)\mathrm{d}y \right\|_2
    = \varepsilon \left| \int_0^1 v(y)\mathrm{d}y \right|
    \leq \varepsilon \|v\|_2.
\end{equation*}
We conclude that for any $\varepsilon>0$ there is a $C_2$, independent of $n$, such that 
\begin{equation}\label{}
    \left\| \int_0^1 [W(x,y) - W_n(x,y)] D_2(u^*(x),u^*(y)) v(y) \mathrm{d}y \right\|_2 \leq (\varepsilon + C_2\|W - W_n\|^{1/2}_\square) \|v\|_2.
\end{equation}

Now we assemble all these pieces to arrive at the proof of this Lemma. Let $\varepsilon > 0$ and $v\in L^2$. Recall the Lipschitz bound $L > 0$ and take $\rho < \min\{\varepsilon/(7L),\rho_T\}$. Then from Theorem~\ref{thm:existence} there is some $N_1\in \N$ such that for $n\geq N_1$ we have that $\|u_n^* - u^*\|_\infty <\rho.$ Furthermore, there are constants $C_1,C_2 > 0$, independent of $n$, such that 
\begin{align*}
    \left\| v\int_0^1[W_n(\cdot,y) - W(\cdot,y)] D_1(u^*(\cdot),u^*(y))\mathrm{d}y \right\|_2 &\leq \left( \frac{\varepsilon}{7} + C_1 \|d_{W_n}-d_W\|_\infty \right) \|v\|_2\\
    \left\| \int_0^1 [W_n(\cdot,y) - W(\cdot,y)] D_2(u^*(\cdot),u^*(y))v(y)\mathrm{d}y \right\|_2 & \leq \left(\frac{\varepsilon}{7} + C_2\|W - W_n\|^{1/2}_\square \right) \|v\|_2.
\end{align*}
From Hypothesis \ref{hyp:Graphon} there is an $N_2\in \N$ such that for $n\geq N_2$ we have both $\|d_{W_n}-d_W\|_\infty < \varepsilon/(7C_1)$ and $\|W - W_n\|^{1/2}_\square < \varepsilon/(7C_2).$ So, putting this all together gives that for $n\geq\max\{N_1,N_2\}$ we have
\begin{align*}
    \|[DF_n(u^*_n) - DF(u^*)]v\|_2 &=  \|[DF_n(u_n^*) - DF_n(u^*)]v + [DF_n(u^*) - DF(u^*)]v\|_2\\
    &\leq \left( 3L\rho  + \frac{\varepsilon}{7} + C_1 \|d_{W_n}-d_W\|_\infty + \frac{\varepsilon}{7} + C_2\|W - W_n\|^{1/2}_\square \right) \|v\|_2\\
    &< \left( \frac{3L\varepsilon}{7L} + \frac{2\varepsilon}{7} + \frac{C_1\varepsilon}{7C_1} + \frac{C_2\varepsilon}{7C_2} \right) \|v\|_2\\
    &= \varepsilon\|v\|_2.
\end{align*}
Since this holds for any $v\in L^2$, it follows that $\|[DF_n(u^*_n) - DF(u^*)]v\|_{2\to 2} < \varepsilon$, concluding the proof.
\end{proof}


\begin{cor} \label{cor: upper semicontinuous}
    If there exists $\gamma > 0$ so that 
    \[
        \spec(DF(u^*))|_{C[0,1]} \subset \{z \in \mathbb{C}|\ \mathrm{Re}(z)  < -\gamma\},
    \]
    then there exists an $N \in \N$ such that for all $n\geq N$ we have 
    \[
        \spec(DF_n(u_n^*))|_{L^2} \subset \{z \in \mathbb{C}|\ \mathrm{Re}(z)  < -\gamma\},
    \]
\end{cor}

\begin{proof}
From Lemma~\ref{lem: 2,2 operator convergence} we conclude that $\|DF_n(u_n^*)\|_{2,2}$ is uniformly bounded by some $R>0$ and so its spectrum lies inside the ball of radius $R$ centered at 0 in the complex plane, denoted $B_R(0)$. We define the half-plane $\Lambda_\gamma := \{z\in \C : \mathrm{Re}(z) < -\gamma\}$ and the compact set $K_\gamma = B_R(0)\setminus \Lambda_\gamma$. Clearly $K_\gamma$ lies inside the resolvent set $\mathbb{C}\setminus\sigma(DF(u^*))_{C[0,1]} = \mathbb{C}\setminus\sigma(DF(u^*))_{L^2}$. Then \cite[Theorem IV.3.1]{kato} gives that there is a $\delta>0$ such that the resolvent set $\mathbb{C}\setminus\sigma(DF_n(u_n^*))$ belongs to $K_{\gamma}$ if $\|DF_n(u^*_n) - DF(u^*)\|_{2,2} < \delta$. Lemma~\ref{lem: 2,2 operator convergence} guarantees this is possible for $n$ sufficiently large, so we conclude $\spec(DF_n(u_n^*))|_{L^2} \subseteq \Lambda_{\gamma}$.
\end{proof}

Finally we relate these results about the step-graphon operator $DF_n(u_n^*)$ back to the original discrete operator $DG_n(\textbf{u}_n^*)$ which we claim in Theorem~\ref{thm:stability} has only stable eigenvalues.

\begin{lem}\label{lem:spec DGn in spec DFn}
    Any eigenvalue $\lambda$ of the matrix $DG_n(\textbf{u}_n^*))\in\mathbb{R}^{n\times n}$ necessarily belongs to $\spec(DF_n(u_n^*))_{L^2}$.
\end{lem}

\begin{proof}
    Fix a value of $n\geq 1$ and suppose $\lambda$ is an eigenvalue of $DG_n(u_n^*)$ with associated eigenvector $\textbf{v}\in \R^n$. Using this eigenvector we define the step function $v(x)\in X_n$ which takes the constant value of the $i$th component of $\textbf{v}$ on the subinterval $[(i-1)/n,i/n)$, for each $i = 1 ,\dots, n$. By construction $v\in L^2$ since it is bounded. For an arbitrary $i = 1,\dots,n$ and $x\in [(i-1)/n,i/n)$ we have
    \[
        \begin{split}
        [DF_n(u_n^*)]v(x) &= f'(u_n^*(x))v(x) + v(x) \int_0^1 W_n(x,y)D_1(u_n^*(x),u_n^*(y)) \mathrm{d}y \\ &\qquad + \int_0^1 W_n(x,y)D_2(u_n^*(x) ,u_n^*(y)) v(y)\mathrm{d}y\\
        &= f'((\textbf{u}_n^*)_i)\textbf{v}_i + \frac{\textbf{v}_i}{n}\sum_{j=1}^n (A_n)_{i,j} D_1((\textbf{u}_n^*)_i,(\textbf{u}_n^*)_j) + \frac{1}{n}\sum_{j=1}^n (A_n)_{i,j} D_2((\textbf{u}_n^*)_i,(\textbf{u}_n^*)_j)\textbf{v}_j\\
        &= (DG_n(\textbf{u}_n^*)\textbf{v})_i\\
        &= \lambda \textbf{v}_i\\
        &= \lambda v(x).
        \end{split}
    \]
    This is true for any $x\in [(i-1)/n,i/n)$ and any $i = 1,\dots,n$, thus giving that $v$ is an eigenfunction of $DF_n(u_n^*)$. Hence, $\lambda \in \spec(DF_n(u_n^*))|_{L^2}$, completing the proof. 
\end{proof}


We conclude the section by summarizing how our work here proves Theorem~\ref{thm:stability}.

\begin{proof}[Proof of Theorem~\ref{thm:stability}]
    Assume the conditions for Theorem \ref{thm:existence} are met and let $N\in \N$ be such that for all $n\geq N$ there exists a $\textbf{u}_n^*$ satisfying $G_n(\textbf{u}_n^*) = 0.$ Further assume that $DF(u^*)$ is stable, specifically that there exists a $\gamma>0$ such that $\spec(DF(u^*))_{C[0,1]} \subseteq \Lambda_\gamma := \{z\in \C: \mathrm{Re}(z) < -\gamma\}$. Lemma~\ref{lem:DF on C = DF on L2} guarantees that $\spec(DF(u^*))_{L^2} \subseteq \Lambda_\gamma$. Then from Lemma~\ref{lem: 2,2 operator convergence} and Corollary~\ref{cor: upper semicontinuous}, we know that there exists an $M\in \N$ such that for $n\geq M$ we have $\spec(DF_n(u_n^*))\subseteq \Lambda_\gamma.$ Finally we can apply Lemma~\ref{lem:spec DGn in spec DFn} to conclude that $\spec(DG_n(\textbf{u}_n^*)) \subseteq \Lambda_\gamma$. Therefore, for $n\geq \max(N,M)$ the steady-state $\textbf{u}_n^*$ to the dynamical system \eqref{eq:main} is stable.
\end{proof}

\begin{rmk}
    While we do not state a specific result here, we remark that the operator norm convergence presented in Lemma~\ref{lem: 2,2 operator convergence} could be used to prove spectral convergence results for steady-states which have a finite number of unstable eigenvalues. We refer to Section~\ref{sec: wilsoncowan} for an example of such a system.   
\end{rmk}


\section{Examples}\label{sec:examples}

In this section, we study several examples that will illustrate our main results.  In Section~\ref{sec:Kuramoto}, we study the famous Kuramoto model and show the existence of twisted state solutions for a certain class of ring graphons. In Section~\ref{sec: wilsoncowan} we study the existence of steady-states in a version of the Wilson-Cowan model. Section~\ref{sec:lotka} provides an application of our results to a Lotka--Volterra model from population ecology. The main goal of this example is to demonstrate that for certain parameters we can show that $\mathcal{T}_n$ is itself a contraction without going to $\mathcal{S}_n$ as in our proofs. Our final example in Section~\ref{sec:bipartite} is meant to demonstrate how our results can be extended to situations beyond those covered by our main results. That is, we again study a Lotka--Volterra model, but now on a bipartite graphon which does not directly satisfy Hypothesis~\ref{hyp:Graphon}(2). We show that if the problem is broken up properly then the individual components do satisfy our hypotheses, allowing for the extension of our results.  


\subsection{Twisted states in the Kuramoto Model}\label{sec:Kuramoto}

As discussed in the introduction, graphons have long been applied to the study of coupled oscillators. Our goal here is therefore to showcase how our results can be applied to complement and extend well-known results from the literature. In particular, we return to the illustrative Kuramoto model \cite{kuramoto84} from the introduction, which describes the behavior of a system of oscillators whose coupling is encoded as a graph on $n$ vertices as follows
\begin{equation}\label{FiniteKuramoto}
    \frac{du_i}{dt} = \frac{1}{n} \sum_{j = 1}^n A_{ij}\sin(2\pi(u_j - u_i)).
\end{equation}
Here $\textbf{u} = (u_1,\dots,u_n)\in \R^n$ describes the phases, taken modulo 1, of $n$ oscillators with pairwise coupling strength which is encoded in the $n\times n$ adjacency matrix $A = [A_{i,j}]_{1 \leq i,j \leq n}$. 

A common setting for \eqref{FiniteKuramoto} is to arrange the oscillators on a ring with distance-dependent coupling between them. Thus, our application here will involve ring graphons to mimic this scenario. The graphon analogue of \eqref{FiniteKuramoto} is given by
\begin{equation}\label{GraphonKuramoto}
    \frac{du}{dt} = \int_0^{1} R(|x-y|)\sin(2\pi(u(y) - u(x)))\, dy,
\end{equation}
where $R$ denotes our ring graphon, as given in Section \ref{sec:GraphonExamples}. This equation was studied previously in \cite{SyncBasin}, although without reference to graphons, with the intention of having the results for \eqref{GraphonKuramoto} hold for \eqref{FiniteKuramoto} when the adjacency matrix is a deterministic weighted graph derived from $R$ and $n$ is large. Our results make this connection rigorous by showing that solutions in the infinite-dimensional setting persist to the large network. Moreover, our results further provide that the same results hold (with high probability) when the underlying network structure in \eqref{FiniteKuramoto} is a random graph whose connection probabilities come from the graphon $R$.

We now proceed by following the procedure of \cite{SyncBasin}. The equation \eqref{GraphonKuramoto} has an explicit family of stationary solutions given by $u^*(x)=m\left(x-\frac{1}{2}\right)$ for any integer $m$. These states are called {\em $m$-twisted states} as the solution covers the circle $m$ times in one cycle around the ring. Here the constant shift of $-m/2$ is chosen to guarantee that the solution has mean zero over $x \in [0,1]$, but the shift invariance of \eqref{FiniteKuramoto} allows for any shift to be chosen, including $0$ which was the choice in the introduction.

The linearization of \eqref{GraphonKuramoto} about a twisted state results in the linear operator $DF(u^*):C[0,1] \to C[0,1]$ acting on $v \in C[0,1]$ by
\begin{equation}
    DF(u^*)v = 2\pi \int_0^1 R(|x - y|)\cos(2\pi m(y-x))[v(y) - v(x)]\drm y.
\end{equation}
As detailed in Section~\ref{sec:GraphonExamples}, ring graphons have Fourier series expansions of the form
\begin{equation}\label{eq:Rfourier}
    R(|x-y|)=\sum_{k\in\mathbb{Z}} c_k \me^{2\pi \mbi k (x-y)},
\end{equation}
with $c_{k} = c_{-k} \in \mathbb{R}$. This allows for a precise characterization of the spectrum of $DF(u^*)$ since the Fourier basis functions are eigenfunctions\footnote{The Fourier basis functions completely characterize the spectrum since they form an orthogonal basis for $L^2$ and Lemma~\ref{lem:DF on C = DF on L2} proves that the spectrum of $DF(u^*)$ is equivalent on $C[0,1]$ and $L^2$.}. To observe this, note that  
\[ 
    DF(u^*)v=-2\pi c_m v +2\pi \int_0^1 \sum_{k\in\mathbb{Z}} c_k \mathrm{e}^{2\pi\mbi k (x-y)} \cos(2\pi m (y-x)) v(y)\mathrm{d} y,
\]
where we see that in this case we have $Q(x)=2\pi c_m$. Then, 
\begin{equation}\label{eq:kuramotoL}
    DF(u^*)\left( \me^{2\pi \mbi \ell x}\right)=\lambda_\ell \me^{2\pi \mbi \ell x}, \qquad    \lambda_\ell=\pi \left( c_{\ell+m}+c_{\ell-m}-2c_m\right).
\end{equation}
Note that $\lambda_0=0$ for any $m \in \mathbb{Z}$, coming from the symmetry $c_{m} = c_{-m}$. This is a consequence of the aforementioned translation invariance of solutions to \eqref{GraphonKuramoto}, giving that the spectrum of $DF(u^*)$ always includes 0. To account for this symmetry we restrict solutions to the space of mean-zero functions, defined as
\[  
    Y_n = \left\{ u\in X_n \ \bigg| \ \int_0^1 u(x) \mathrm{d} x =0 \right\}.   
\]
This removes the translational eigenvalue, and the twisted state is a spectrally stable solution of \eqref{GraphonKuramoto} if $\lambda_\ell<0$ for all $\ell\neq 0$. We comment on the details of using the subspace $Y_n$ instead of $X_n$ in our proofs in Appendix~\ref{sec:KuramotoApp}, while here only noting that all of our results go through after quotienting out the translational symmetry of \eqref{GraphonKuramoto}.     

Details for guaranteeing $\lambda_\ell < 0$ for all $\ell \in \mathbb{Z}\backslash \{0\}$ were worked out for the small-world graphon \eqref{SmallWorldGraphon} with $q = 0$ and $p,\alpha \in (0,1)$ in \cite{SyncBasin}. Precisely, it is shown that $\lambda_\ell < 0$ for all $\ell$ if $|m|\alpha < \mu\pi$, where $\mu \approx 0.6626$ is obtained by solving 
\begin{equation}
    \tan(\pi \mu) = \frac{2 \pi \mu}{2 - (\pi\mu)^2}.
\end{equation}
Furthermore, a proof of nonlinear stability of the twisted state, as a solution of the graphon equation \eqref{GraphonKuramoto} was obtained in \cite{wrightex}. These ideas combine to imply that there exist ring graphons $R(|x-y|)$ for which stable twisted state solutions of \eqref{GraphonKuramoto} are known to exist. These results can further be leveraged to apply the results of this manuscript. For example, in the introduction we saw Figure~\ref{twisted_state_graphs} that presented twisted states as solutions to \eqref{GraphonKuramoto} with $\alpha = 0.2$, $p = 1/(2\pi\alpha) \approx 0.8$, and $q = 0$, as well as the solutions guaranteed (with high probability) by Theorem~\ref{thm:existence} on a random graph with $n = 200$ vertices. With this value of $\alpha$ all twisted states in Figure~\ref{twisted_state_graphs} are stable, as expected from Theorem~\ref{thm:stability}. 

\begin{rmk} 
One ring graphon structure that can immediately be ruled out from the application of our results is that of Erd\H{o}s--R\'eyni graphons where $W(x,y)$ is constant. The reason for this can be traced to the fact that in \eqref{eq:Rfourier} the coefficients $c_k$ are all zero aside from $c_0=p$. Consulting \eqref{eq:kuramotoL} one observes that $\lambda=0$ is an eigenvalue of infinite multiplicity which prevents the existence of a bounded inverse for the linearization. It turns out that \eqref{FiniteKuramoto} posed on sequences of finite graphs converging to an Erd\H{o}s--R\'eyni graphon cannot be guaranteed to support the same structures as the limiting graphon equation. We refer the interested reader to \cite{Nonex} where \eqref{FiniteKuramoto} is considered on both complete and Paley graphs which converge to the same constant graphon in cut norm. On complete graphs the twisted states are shown to always be stable, while on Paley graphs they are always unstable. This discrepancy between graphs with the same limiting graphon highlights the necessity of $DF(u^*)$ having a bounded inverse in our results. 
\end{rmk}


\subsection{The Wilson-Cowan Model}\label{sec: wilsoncowan}

As a second example we consider a model of excitation in neural networks. It takes the form of a modified Wilson-Cowan model, as defined in \cite{sanhedrai23},
\begin{equation} \label{discrete neuro}
    \frac{\drm u_i}{\drm t} = -u_i + \frac{\lambda}{n} \sum_{j = 1}^n A_{ij} \frac{1}{1+\exp(\mu - \delta u_j)}, \qquad i = 1,\dots,n.
\end{equation}
Here $u_i$ represents the excitation level of the $i$-th neuron, which decays due to the linear self-interaction term and is sustained by couplings to other nodes. Here $\mu$ and $\delta$ are parameters that define the activation threshold. Natural settings for neuronal models such as \eqref{discrete neuro} are over large networks, which motivates analyzing the mean-field graphon version, given by
\begin{equation}\label{continuous neuro}
    \frac{du}{dt} = -u(x) + \lambda \int_0^1 \frac{W(x,y)}{1+\exp(\mu - \delta u(y))} \, dy.
\end{equation}  
In this subsection we will identify stable steady-states of \eqref{continuous neuro} and apply our results to demonstrate the existence of steady-states to \eqref{discrete neuro} over classes of large, random networks.

Prior to stating our first result, we recall that a graph whose vertices all have the same degree is called regular. Similarly a graphon can be called regular, or degree-constant, if $d_W(x) = \int_0^1 W(x,y)\, dy = \Omega\in \R$ for all $x \in [0,1]$. This family of graphons include many commonly studied types like Erd\H{o}s-R\'enyi and ring graphons.  The next lemma shows that this system has up to three constant steady-state solutions when $W$ is regular. We will omit the proof since it is a straightforward computation. 

\begin{lem} \label{lem:ConstantSolution}
    Let $W$ be a regular graphon so that $d_W(x) = \Omega \in [0,1]$. Then, \eqref{continuous neuro} has a family of constant steady-state solutions defined implicitly by
        \begin{equation} \label{eq:Gamma}
            u^* = \frac{\lambda\Omega}{1+\exp(\mu - \delta u^* )},
        \end{equation}
    which for each $\Omega\in [0,1]$ has either one, two, or three solutions.  
\end{lem}

\begin{figure}[t]
    \centering
    \includegraphics[width = 0.75\textwidth]{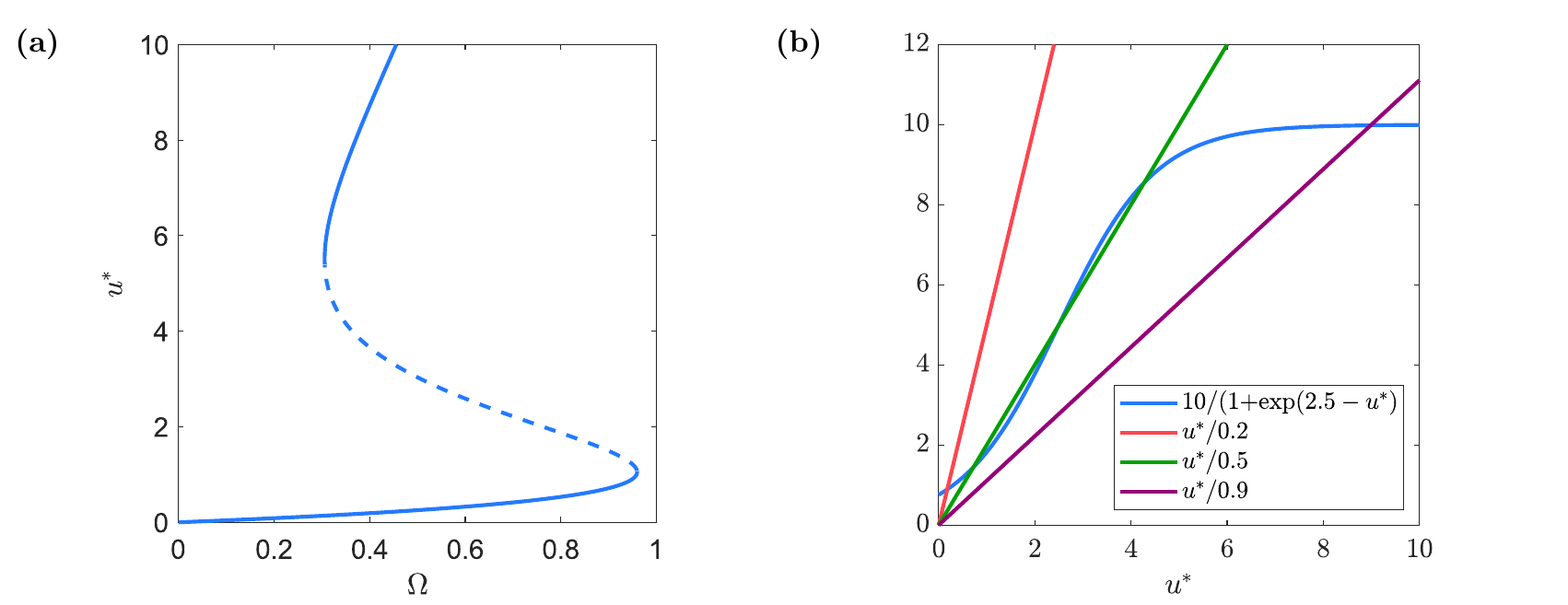}
    \caption{Homogeneous steady-states for the Wilson-Cowan type model (\ref{continuous neuro}).  The left panel plots homogeneous steady-states for degree constant graphons as a function of the degree.  Note the region of bistability. In the right panel we plot $\lambda/ (1+\mathrm{exp}(\mu-\delta u^*)$ and $\frac{u^*}{\Omega}$.  Intersections of these curves represent steady-states and bistability is observed as $\Omega$ is varied.
    }
    \label{Scurve}
\end{figure}

We can visualize the presence of multiple solutions to the implicit equation \eqref{eq:Gamma} by treating $\Omega$ as a bifurcation parameter. Indeed, for fixed $(\mu,\delta,\lambda)$ we can rearrange to get
\begin{equation}
    \Omega=\frac{1}{\lambda}u^*\left(1+e^{\mu-\delta u^*}\right).
\end{equation}
This bifurcation curve is plotted in Figure~\ref{Scurve} with $(\mu,\delta,\lambda) = (4,1,22)$, where one sees the presence of a region of bistability wherein two co-existing stable steady-states can be found. For example, at the value $\Omega = 0.5$ which is firmly inside the region of bistability, one finds a stable homogeneous state $u^* \approx 0.25$ of low excitation and another at $u^* \approx 11$ for high excitation. These stable states are separated by an intermediate unstable homogeneous state. 

For exposition, let us consider the simple case of an Erd\H{o}s--R\'eyni graphon $W(x,y) = p \in (0,1)$. Then $\Omega = p$, so we can investigate the stability of the homogeneous states for fixed parameters $(\mu,\delta,\lambda)$ depending on the value of $p$. The linearization about a homogeneous steady-state to \eqref{continuous neuro} takes the form
\begin{equation}
    [DF(u^*)v](x) = -v(x) + \lambda   \int_0^1 \frac{W(x,y) \delta \exp(\mu - \delta u^*)}{(1+\exp(\mu - \delta u^*))^2} v(y) \, dy.
\end{equation}
In the case of an Erd\H{o}s--R\'eyni graphon we can greatly simply the above linearization to get 
\begin{equation}
    [DF(u^*)]v(x) = -v(x) + R(u^*,p,\lambda,\delta,\mu) \int_0^1v(y)dy ,
\end{equation}
where we denote the constant, parametrically-dependent term
\begin{equation}
    R(u^*,p,\lambda,\delta,\mu)=\frac{\delta (u^*)^2 e^{\mu-\delta u^* }}{\lambda p}.
\end{equation}
As was the case with the Kuramoto model, the eigenfunctions of $DF(u^*)$ are the Fourier basis functions. We state our results on the spectrum of $DF(u^*)$ in the following lemma, which again is stated without proof.

\begin{lem}\label{lem:NeuroStability}
    Let $W(x,y) = p \in[0,1]$ for all $(x,y) \in [0,1]^2$ and suppose $u^*$ is a homogeneous steady-state solution to \eqref{continuous neuro}. Then, the spectrum of the linearization of \eqref{continuous neuro} consists of only two points $-1$ and $-1+R(u^*,p,\lambda,\delta,\mu)$. The multiplicity of the eigenvalue $-1+R(u^*,p,\lambda,\delta,\mu)$ is one with an eigenspace spanned by the constant function on $[0,1]$. The eigenvalue $-1$ has infinite multiplicity. 
\end{lem}

\begin{figure}
    \centering
    \includegraphics[width = \textwidth]{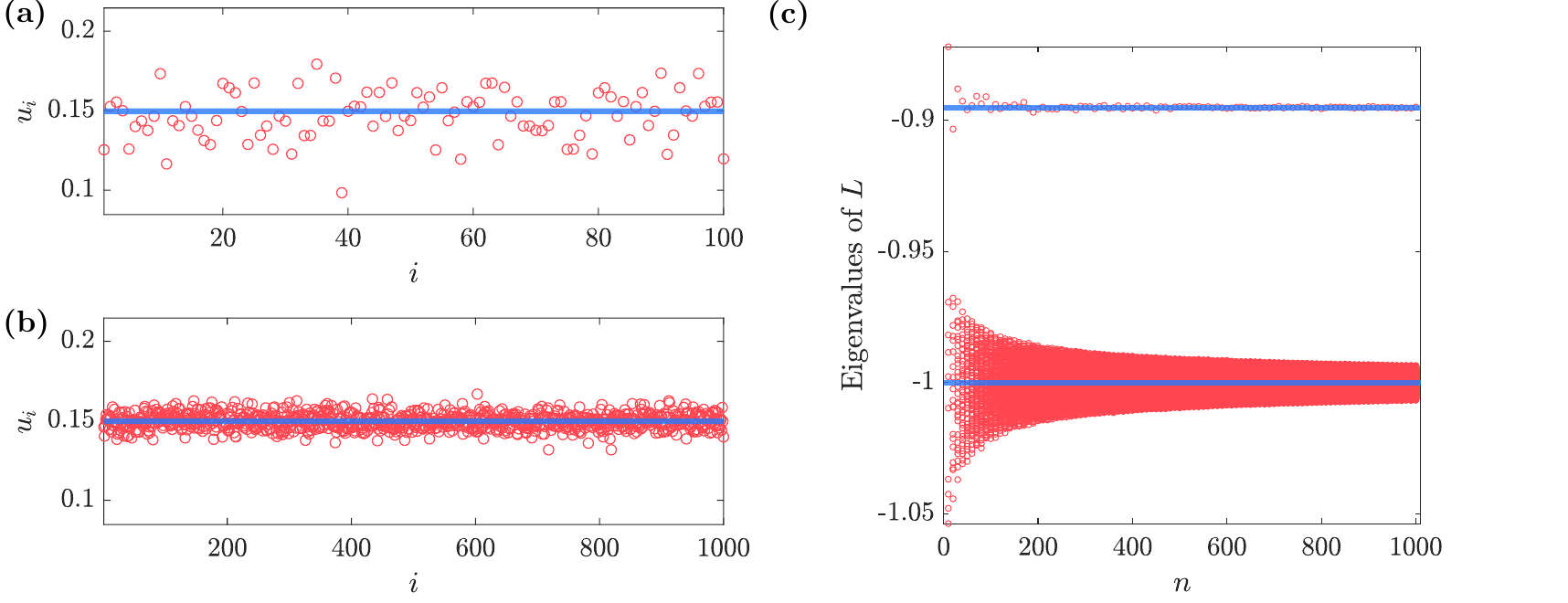}
    \caption{Left: Constant steady-state for (\ref{continuous neuro}) using $W(x,y) =p=0.5$ and $\lambda = \mu = \delta = 1$ gives $u^*\approx 0.15,$ shown in blue. Nearby steady-states for (\ref{discrete neuro}) with ER random graphs on 10 and 200 vertices shown in red. Right: Eigenvalues for $DF(u^*)$ using ER random graphs on $n$ vertices arrayed vertically for each $n$.}
    \label{HugeNeuro}
\end{figure}

Therefore, from Lemma~\ref{lem:NeuroStability} we see that if $R(u^*,p,\lambda,\delta,\mu)<1$ then the homogeneous steady-state $u^*$ to \eqref{continuous neuro} is stable. In this case, Theorems~\ref{thm:existence} and \ref{thm:stability} can be applied, guaranteeing, with high probability, the existence of a nearly homogeneous steady steady-state for \eqref{discrete neuro} posed on a sufficiently large Erd\H{o}s--R\'eyni random graph. We can further validate our analysis with numerical experiments. Some findings are presented Figure~\ref{HugeNeuro} where we compare identified steady-states of \eqref{discrete neuro} with $n = 100$ and $1000$ nodes. Realization of the network with fewer nodes results in a numerical solution that is not particularly close to the homogeneous solution of the continuum model \eqref{continuous neuro}. However, for larger values of $n$ the steady-state solution on the random graph is found to closely resemble the homogeneous solution from the graphon equation. Furthermore, in accordance with Theorem~\ref{thm:stability}, we see that the steady-state of the discrete problem \eqref{discrete neuro} inherits the stability of homogeneous solution and the eigenvalues of $DG_n$ appear to converge to the spectrum of $DF$, as depicted in Figure~\ref{HugeNeuro}.

\begin{rmk}
  As we mentioned above, system (\ref{continuous neuro}) exhibits bistability for some choices of parameters.  This is illustrated in Figure~\ref{Scurve} where two stable branches of homogeneous states are connected by a branch of unstable homogeneous states.  The branches meet at saddle-node bifurcations where $R(u^*,p,\lambda,\delta,\mu)=1$ and $DF(u^*)$ has a zero eigenvalue with multiplicity one.  Since $DF(u^*)$ is not invertible at these points our results do not apply. The study of tracking bifurcations from a continuum model back down to the discrete model will the focus of future research.  
\end{rmk}

\subsection{Lotka-Volterra Competition Model} \label{sec:lotka}

The Lotka--Volterra model can be employed to describe the interaction of competing or cooperating species \cite{volterra39,hu22}. Letting $u_i$ denote the abundance of the $i$th species, then a (competitive) Lotka--Volterra model for $n \geq 1$ interacting species is given by  
\begin{equation} \label{eq:finiteLV}
    \frac{\drm u_i}{\drm t}=u_i(1-u_i)-\frac{\lambda}{n} \sum_{j=1}^n A_{ij} u_iu_j . 
\end{equation}
Here $A = [A_{ij}]_{1 \leq i,j \leq n}$ is a matrix whose non-zero entries represent the existence of a competitive interaction between two species. Here we are assuming that each species undergoes simple logistic growth in the absence of the other species, while we have state-dependent quadratic interactions between the species.

As in the previous subsection, we will focus on the case of Erd\H{o}s--R\'eyni graphons. The purpose for this is primarily to illustrate that the proofs of our main results can be simplified in this limited scenario. First,  for pair-wise interactions that  occur randomly with some probability $p \in [0,1]$, the limiting graphon system for \eqref{eq:finiteLV} is given by
\begin{equation}\label{eq:LV}
    u_t = u(1-u)-\lambda p\int_0^1 u(x)u(y)\mathrm{d}y.  
\end{equation}
We summarize our findings for \eqref{eq:LV} with the following lemma, which is again stated without proof due to its simplicity.

\begin{lem}\label{lem:LV}
    The constant function $u^* = 1/(1 + \lambda p)$ is a steady-state solution to \eqref{eq:LV}. Furthermore, the linearization about this steady-state results in the linear operator whose action on $C[0,1]$ is given by
    \begin{equation}
        [DF(u^*)v](x) = \frac{-1}{1+\lambda p}v(x)-\frac{\lambda }{1+\lambda p}\int_0^1 p v(y)\mathrm{d}y, 
    \end{equation}
    and has spectrum consisting of two points: $\frac{-1}{1+\lambda p}$ and $-1$. The spectral element $\frac{-1}{1+\lambda p}$ has infinite multiplicity, while $-1$ is an eigenvalue of multiplicity 1 with eigenspace spanned by the constant functions.      
\end{lem}

\begin{figure} 
    \centering 
    \includegraphics[width = 0.9\textwidth]{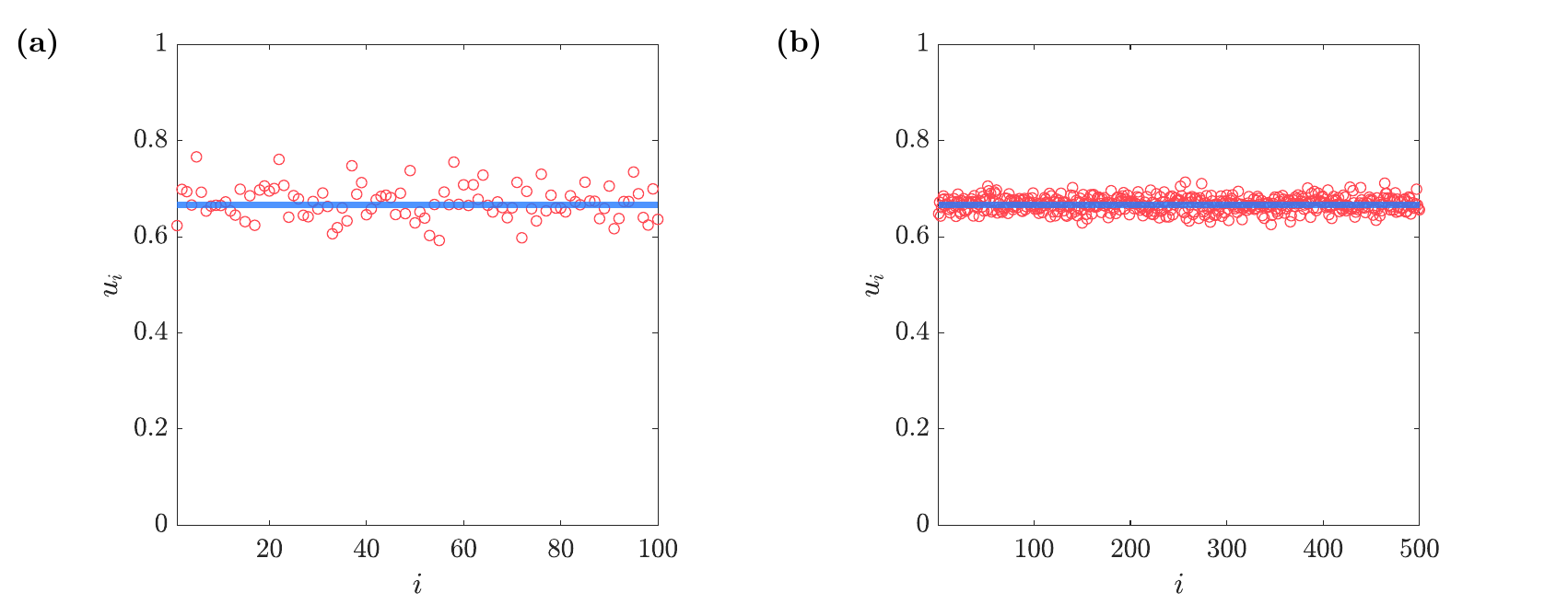}
    \caption{Left: steady-state solutions of the Lotka-Volterra model.  For a constant graphon with $p = 1/2$ and $\lambda = 1$ the graphon equation in (\ref{eq:LV}) has the solution $u^* = 2/3$ which is depicted in blue. For a random realization on $n=200$ nodes we find a nearby solution vector to (\ref{eq:finiteLV}) shown in red. Right: 
    The steady-state for the system (\ref{eq:finiteLV}) with a random realization on $n=1000$ nodes.  } \label{fig:LV}
\end{figure}

From Lemma~\ref{lem:LV} we see that for any $\lambda > 0$ we have that the homogeneous steady-state $u^* = 1/(1 + \lambda p)$ of \eqref{eq:LV} is stable. Theorem~\ref{thm:existence} and Theorem~\ref{thm:stability} then imply the existence and stability of steady-states of (\ref{eq:finiteLV}) for $n$ large and with high probability.  Steady state solutions of (\ref{eq:finiteLV}) are located numerically in Figure~\ref{fig:LV}.

In terms of the mathematical analysis presented in this paper, one interesting aspect of this example is that it sometimes allows for a simpler route to prove the  existence result Theorem~\ref{thm:existence}. The homogeneous steady-state in Lemma~\ref{lem:LV} exists for all $\lambda > 0$, but when $\lambda \in (0,1)$ we can show that $\mathcal{T}_n$ is a contraction on $X_n$, using the notation of Section~\ref{sec:proof1}. Indeed, here we have 
\[ 
    \frac{\Xi(W_n,W,u^*)}{Q(x)}v=\int_0^1 \lambda \left[p-W_n(x,y)\right]v(y)\mathrm{d}y. 
\]
If $W_n(x,y)$ only takes values of zero or one, as it would in the case of a random graph, then the operator norm
\[ 
    \left\|\frac{\Xi(W_n,W,u^*)}{Q(\cdot)}\right\|_{X_n\to X_n}\leq \lambda\ \mathrm{max}\{p,1-p\}   
\]
and the operator $\mathcal{T}_n$ will be a contraction if this operator norm is less one. Of course, one way to guarantee this is to have $\lambda \in (0,1)$, while one can take larger values of $\lambda$ so long as $p$ is taken so that $\lambda\ \mathrm{max}\{p,1-p\} < 1$. Therefore, in this limited case one does not require using the second-iterate mapping $\mathcal{S}_n$ to prove the result of Theorem~\ref{thm:existence}.


\subsection{Ecological Competition with Mutualistic Interactions}\label{sec:bipartite}

Our final example is one of practical interest, but also lies outside the direct scope of the analysis presented already.  We discuss this example briefly, both to illustrate the issues with directly applying our main result to this class of problems, but also to demonstrate how our methods could be generalized to study systems of this form.  

The model in question is one of ecological dynamics under mutualistic interactions motivated by models presented in \cite{gao16}.  We again consider a Lotka--Volterra model, similar to \eqref{eq:finiteLV}, but now with cooperative interaction and interaction matrix given in the form of a bi-partite graph. Precisely, the model takes the form 
\begin{equation} \label{eq:finitemutual}
    \frac{du_i}{dt}=u_i(1-u_i)+\frac{\lambda}{n} \sum_{j=1}^n A_{ij} u_iu_j,  
\end{equation}
and the non-local graphon counterpart is
\begin{equation} \label{eq:graphonmutual}
    u_t=u(1-u)+\lambda \int_0^1 W(x,y) u(x)u(y)\mathrm{d}y,  
\end{equation}
where we recall that the bi-partite graphon is defined for any $p \in [0,1]$ by
 \begin{equation}\label{eq:bipartiterepeat}
        W(x,y) = \begin{cases}
            p & \mathrm{if}\ \min\{x,y\}\leq \alpha,\ \max\{x,y\} > \alpha, \\
            0 & \mathrm{otherwise}.
        \end{cases} 
    \end{equation}
We summarize our findings with the follow lemma.

\begin{lem}\label{lem:LVBipart}
    Consider \eqref{eq:graphonmutual} and suppose that $(\lambda,p,\alpha)$ are such that $\lambda^2p^2\alpha (1-\alpha)<1$. Then, \eqref{eq:graphonmutual} has a piecewise constant steady-state solution given by
    \[
        u^*(x) = \begin{cases} 
        \frac{1+\lambda p (1-\alpha)}{1-\lambda^2p^2\alpha (1-\alpha) } & 0 \leq x < \alpha, \\
        \frac{1+\lambda p \alpha}{1-\lambda^2p^2\alpha (1-\alpha) } & \alpha \leq x \leq 1.
    \end{cases}
    \]
\end{lem}

\begin{proof}
    The result is obtained by assuming the form
    \[ 
        u^*(x)= \begin{cases} 
            u_1^* & 0 \leq x < \alpha \\
            u_2^* & \alpha \leq x \leq 1.
        \end{cases}
    \]
    for a solution to \eqref{eq:graphonmutual}. Plugging this into \eqref{eq:graphonmutual} results in the linear matrix equation
    \begin{equation}
        \left(\begin{array}{cc} -1 & \lambda p (1-\alpha) \\ \lambda p \alpha & -1\end{array}\right)\left(\begin{array}{c} u_1^*\\ u_2^* \end{array}\right)= \left(\begin{array}{c} -1\\ -1 \end{array}\right),
    \end{equation}
    for which the solution is given in the statement of the lemma. This concludes the proof.
\end{proof} 

The main obstacles in leveraging the result of Lemma~\ref{lem:LVBipart} to apply our results in Section~\ref{sec:mainresults} are as follows. First, bipartite graphons as given in \eqref{eq:bipartiterepeat} are neither degree constant when $\alpha \neq 1/2$ nor ring graphons and so it might not be the case that degree convergence in the $L^\infty$ norm can be obtained to satisfy Hypothesis~\ref{hyp:Graphon}(1). Precisely, if the discontinuities of the associated sequence of step graphons $W_n$ do not align with the one at $x = \alpha$ in the generating graphon $W$, then uniform degree convergence cannot be obtained. This would be the case if $x = \alpha$ lies between $(i-1)/n$ and $i/n$ for some $n \geq 1$ and $i = 1,\dots,n$. Second, it is no longer clear whether the operator $T_K$ in our proofs is compact since bipartite graphons do not satisfy Hypothesis~\ref{hyp:Graphon}(2). Again, this comes from the jump discontinuity at $x = \alpha$. Finally, there is the issue that the steady-state solution in Lemma~\ref{lem:LVBipart} is not continuous over $x \in [0,1]$, thus not satisfying Hypothesis~\ref{hyp:ContSol}.

Nonetheless, we believe that our analysis could be adapted to study the persistence of steady-states in \eqref{eq:graphonmutual}. The first step is properly constructing the step graphon $W_n(x,y)$ to achieve degree convergence. Fixing a number of vertices $n\geq 1$, define $n_1=\lfloor \alpha n\rfloor$ and $n_2=n-n_1$. One then defines a partition of $[0,1]$ using the points
\[ 
    x_i= \begin{cases}
        \alpha \bigg(\frac{i - 1}{n}\bigg) & i = 1,\dots, n_1, \\
        \alpha + (1 - \alpha)\bigg(\frac{i - n_1 - 1}{n}\bigg) & i = n_1 + 1, \dots, n.
    \end{cases}
\]
Notice that this discretizes the subintervals $[0,\alpha)$ and $[\alpha,1]$ separately with a potentially different step size for the different subintervals. However, this discretization allows for the construction of random graphs and associated step-graphons which will converge (with high probability) to the bipartite graphon; in both the cut norm and uniformly in the degree function. 

\begin{figure}
    \centering
    \includegraphics[width = 0.9\textwidth]{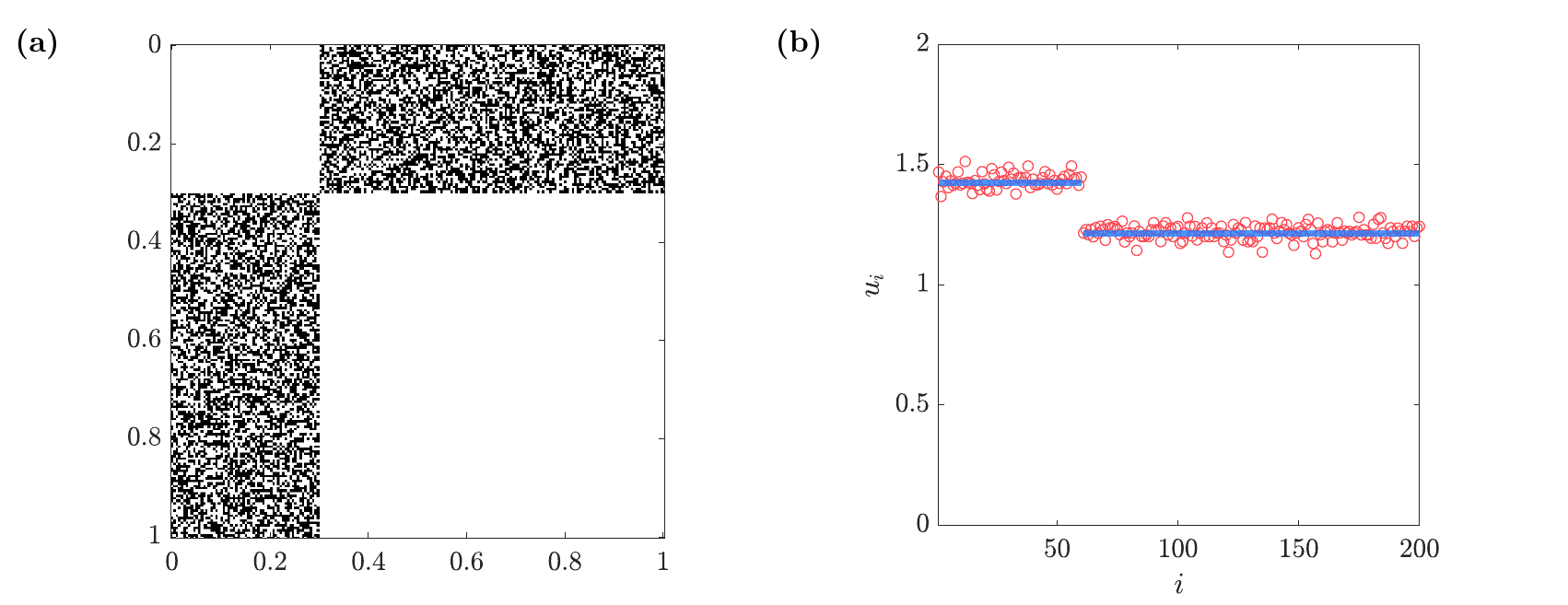}
    \caption{(a) A pixel plot for a random realization of a bipartite graph on 200 vertices generated from the graphon \eqref{eq:bipartiterepeat} with $(p,\alpha)=(0.5,0.3)$. (b) The steady-state solution from Lemma~\ref{lem:LVBipart} (blue, solid) compared to the steady-state solution to the discrete system \eqref{eq:finitemutual} on the random bipartite random graph from the left panel (red, dots).}
    \label{fig:LVbipartite}
\end{figure}

The other two outstanding issues pertain to the jump discontinuity at $x = \alpha$ in both the bipartite graphon and the steady-state $u^*(x)$ in Lemma~\ref{lem:LVBipart}. To circumvent this we could instead replace statements and assumptions using the Banach space $C[0,1]$ with the Banach space
\[ 
    Y=\left\{ u\in L^\infty[0,1] \  \bigg|  \ u(x) \ \text{is continuous on $[0,\alpha)$ and $[\alpha,1]$} \right\}. 
\]
Essentially this Banach space would break the analysis up into two pieces, one using continuous functions on the interval $[0,\alpha)$ and the other using continuous functions on $[\alpha,1]$. Since the bipartite graphon is continuous on each of these subintervals, Hypothesis~\ref{hyp:Graphon}(2) could be verified independently for each subinterval. Thus, we  anticipate that our results could be recovered by breaking the problem into two in this way, although we do not pursue this analysis here. In Figure~\ref{fig:LVbipartite} we present a numerical computation of a steady-state solution for \eqref{eq:finitemutual} defined on a bipartite graph as compared to the graphon equilibrium which provides further evidence for our hypothesis. Similar generalizations should be able to be made for a graphon with any finite number of jump discontinuities.

\section{Discussion}\label{sec:discussion}

Reaction-diffusion equations over networks of the form \eqref{eq:main} provide a general framework to model a number of systems of applied interest. Many realistic systems require complex interaction matrices $A$ that may only be described in a probabilistic sense, thus often making general statements about the system difficult. To avoid case-by-case investigations of these discrete reaction-diffusion systems it is common to (formally) let the number of nodes tend to infinity and study the associated mean-field equations, generally taking the form \eqref{eq:main2}. Working with the spatially continuous limiting problem is often easier (see Section~\ref{sec:examples} for a few examples) and it provides only one equation as opposed to the numerous variations possible for the choices of interaction matrices $A$ within a certain class. However, after completing the study of the mean-field equation \eqref{eq:main2} one is still tasked with leveraging this information to say something about the behavior of the finite-dimensional system \eqref{eq:main} that initiated the investigation, at least for sufficiently large networks.  

The purpose of this research study was to demonstrate that important structures of the mean-field equations provide the existence of related structures in large discrete systems that are close to the limiting problem. Importantly, these results allow one to study a single infinite-dimensional problem to almost surely provide information about networked reaction-diffusion equations with large random interaction matrices. In our case the structures of interest were steady-states and we used the now well-developed theory of graphons to arrive at our results. These results are complementary to related investigations in \cite{medvedev14} that leverage graphon theories to guarantee finite-time proximity of solutions between finite networked dynamical systems and their limiting mean-field graphon equation.

There are many possible extensions we believe are immediately amenable to our methods. First, while we focused only on systems of scalar equations, there is little doubt that these methods can be adapted to apply to settings where there are multiple dependent variables for each node $n$. Second, we sketched out in Section~\ref{sec:bipartite} that while our compactness assumption Hypothesis~\ref{hyp:Graphon}(2) does not hold for bipartite graphons, our approach can be adapted to handle these graphons as well. Thus, it seems that a relatively minor variation of our hypotheses could incorporate bipartite, stochastic block model, and various other graphons that arise in application which do not satisfy our assumptions in their current form. Third, we could have considered inhomogeneous reaction functions, i.e. state-dependent reaction terms of the form $f_i(u_i)$ for each $i$.

Beyond these extensions, there remains much larger questions that we wish to pursue in follow-up investigations. For example, there are now two notable extensions of graphons, graphops \cite{gkogkas2022graphop2} and embedded vertexon-graphons \cite{caines2022embedded}, that are able to capture more diverse graph structures. Thus, formulating our theory in terms of one or both of these graphon extensions would significantly extend the applicability of our results. In terms of dynamical assumptions, the most obvious generalization is to cases where $DF(u^*)$ fails to be invertible due to the existence of an isolated zero eigenvalue of finite multiplicity. Such investigations would likely include a parameter dependence in \eqref{eq:main} to unfold various bifurcations in the mean-field limit and provide insight into the behavior of the nearby finite-dimensional systems, similar to what was done in \cite{bramburger2021pattern}. Cases where the operator $DF(u^*)$ fails to be invertible due to the essential spectrum including 0 are more challenging and determining stability or instability without knowing particular information about the discrete network structure may not be possible, as illustrated in \cite{Nonex}.

\section*{Acknowledgements}  The research of MH and JW was partially supported by the National Science Foundation through DMS-2007759. JB was supported by an NSERC Discovery Grant.


\appendix

\section{Verifying Hypothesis~\ref{hyp:Graphon}(2) for ring graphons}

At first glance, it may appear that Hypothesis~\ref{hyp:Graphon}(2) is difficult to verify if given a graphon. Here we show that this assumption can always be shown to hold for ring graphons. We begin with their definition.

\begin{defn}
    A graphon $W:[0,1]\times [0,1] \to [0,1]$ is said to be a {\bf ring graphon} if there exists a function $R:[0,1] \to [0,1]$ which is 1-periodic, piecewise continuous, and satisfies $W(x,y) = R(|x - y|)$ for all $x,y \in [0,1]$.
\end{defn}

Notable examples of ring graphons are Erd\"{o}s--R\'eyni and small-world graphons, while bipartitie graphons are not rings. We now demonstrate that Hypothesis~\ref{hyp:Graphon}(2) holds for ring graphons.

\begin{lem}\label{lem: 1 disc ring case}
    Let $W$ be a ring graphon. Then for any $\varepsilon > 0$, there exists a $\delta > 0$ so that for every $x_0 \in [0,1]$ it holds that
    \[
         \int_0^1|W(x_0,y) - W(x,y)| \mathrm{d}y < \varepsilon
    \]
    when $|x - x_0| < \delta$, where $|\cdot|$ is taken modulo 1 by the periodicity of the ring graphon.
\end{lem}

\begin{proof}
    First, by definition since $W$ is a ring graphon, there exists a 1-periodic, piecewise continuous function $R$ so that $W(x,y) = R(|x - y|)$. This proof is then carried out by induction on the number of jump discontinuities in the function $R$.

    Let us begin by assuming that $R$ is continuous. That is, $R$ has 0 jump discontinuities. Continuity of $R:[0,1] \to [0,1]$ further implies uniform continuity since $[0,1]$ is compact, and the result follows immediately from this. Thus, the base case of $k = 0$ jump discontinuities holds.

    Now, let us assume that if $R$ has $k \in \mathbb{N} \cup \{0\}$ or less jump discontinuities the lemma holds. We now consider the case that there are $k + 1$ jump discontinuities in $R$. Let $\xi \in [0,1]$ be the location of one such jump discontinuity. Without loss of generality we can assume that $\xi \in(0,1)$ since $R$ is 1-periodic. Then, $R$ restricted to the sub-intervals $[0,\xi]$ and $[\xi,1]$ has $k$ or fewer jump discontinuities. Therefore, by our inductive hypothesis we have that for all $\varepsilon > 0$ there is a $\delta_1 > 0$ so that for all $x_0 \in [0,1]$ we have
    \[
        \int_0^{\xi - \delta_1}|R(|x_0 - y|) - R(|x - y|)| \mathrm{d}y < \frac{\varepsilon}{3}, 
    \]
    when $|x_0 - x| < \delta_1$ and $|x_0 - y| < \xi$, and 
    \[
        \int_{\xi+\delta_1}^1|R(|x_0 - y|) - R(|x - y|)| \mathrm{d}y < \frac{\varepsilon}{3}, 
    \]
    when $|x_0 - x| < \delta_1$ and $|x_0 - y| > \xi$.
    Furthermore, letting 
    \[
        l = \lim_{x \to \xi^-} R(x), \quad r = \lim_{x \to \xi^+} R(x),
    \]
    we further have that there exists a $\delta_2 > 0$ which guarantees that $2|r - l|\delta_2< \varepsilon/3$. Thus, setting $\delta = \min\{\delta_1,\delta_2,\frac{1-\xi}{2},\frac{\xi}{2}\}$, it follows that for all $|x_0 - x| < \delta$ we have
    
    \[
        \begin{split}
            \int_0^1|W(x_0,y) - W(x,y)|\, \text{d}y &= \int_0^1|R(|a-y|) - R(|x-y|)|\, \text{d}y \\
            &\leq \int_0^{\xi-\delta}|R(|x_0-y|) - R(|x-y|)|\, \text{d}y \\ 
            &+ \int_{\xi+\delta}^1|R(|x_0-y|) - R(|x-y|)|\, \text{d}y + 2|l-r|\delta\\
            &< \frac{\varepsilon}{3} + \frac{\varepsilon}{3} + \frac{\varepsilon}{3}\\
            &= \varepsilon,
        \end{split}   
    \]
    where we have used the fact that $|W(x_0,y) - W(x,y)| \leq 1$ for all $x_0,x,y \in [0,1]$. Thus, we have proven the inductive step and completed the proof.
\end{proof}

\section{Proof of Lemma~\ref{lem:MinvProjected}}\label{app:BigProof}

In this appendix we provide the proof of Lemma~\ref{lem:MinvProjected}. The proof is broken down into components that correspond to the items enumerated in the statement of the lemma.  We remark that the continuity of the functions $\psi_{j,k}(y)$ in the spectral projections is essential to our proof of Theorem~\ref{thm:existence}.  We suspect this fact is perhaps already known in the literature, but absent a suitable reference and for the benefit of the reader we provide a full proof here in part (3) below.  

(1)\ By the definition \eqref{eq:TK}, $T_K$ is given by 
\[ 
    T_Kv=\frac{1}{Q(x)}\int_0^1 W(x,y)D_2(u^*(x),u^*(y))v(y)\mathrm{d}y.
\]
Since $Q(x)$ is continuous, nonzero, and positive, we have that compactness of $T_K:X_n \to X_n$ will follow from compactness of the operator 
\[ 
    X_n \ni v \mapsto \int_0^1 W(x,y)D_2(u^*(x),u^*(y))v(y)\mathrm{d}y,
\]
which was proven to be compact in the proof of Lemma~\ref{lem:fredholm}. Thus, $T_K : X_n \to X_n$ for all $n \geq 1$ is compact, proving point (1).

(2)\ Since we have now shown that $T_K:X_n \to X_n$ is compact, the existence of only a finite number of eigenvalues with real part greater than or equal to one is an immediate corollary. Compactness further implies that each eigenvalue has a finite algebraic multiplicity. An analogous argument as performed in Lemma~\ref{lem:Spectrum} shows that the spectrum of $T_K$ is independent of the space $X_n$ and the corresponding eigenfunctions and generalized eigenfunctions are all continuous.  

(3) \ Let $J \geq 0$ be the number of eigenvalues of $T_K$ with real part greater than or equal to 1. Let $m_j \geq 1$ be the multiplicity of the $j$th eigenvalue, denoted $\lambda_j \in \mathbb{C}$, for each $j = 1,\dots, J$.  For any fixed $n \geq 1$, let $P_{\lambda_j}:X_n \to X_n$ be the spectral projection associated to $\lambda_j$ and define
\[ 
    P=\sum_{j=1}^J P_{\lambda_j}, \quad \tilde{P}=I-P.
\]
Then the eigenvalues $\lambda_j$ and their generalized eigenspaces decompose $X_n$ as
\begin{equation}\label{eq:XnOplus}
    X_n=X_{\lambda_1}\oplus X_{\lambda_2}\oplus \dots \oplus X_{\lambda_J}\oplus \tilde{X}_n, 
\end{equation}
where each $X_{\lambda_j} = P_{\lambda_j}X_n$ is finite-dimensional.  As a result of point (2) we have that $X_{\lambda_j}\subset C[0,1]$ for all $j$.

We now focus on a specific $\lambda_j$ and construct the spectral projection associated to this eigenvalue.  We work first in the space $C[0,1]$ and then show that this expression remains valid on the larger space $X_n$.

To condense notation, we will set $L_K := I-T_K$. Since $\lambda_j \neq 1$ for all $ j = 1,\dots, J$ we have that the multiplication operator $v\to (1-\lambda_j) v$ is invertible and hence Fredholm with index zero. Then, since $T_K$ is compact, it further follows that $L_k-\lambda_jI$ is also Fredholm with index zero.  Identical arguments can be applied to show that the operator $(L_K-\lambda_jI)^{m_j}$ is also Fredholm with index zero.  Consequently, since $m_j$ is the algebraic multiplicity of the eigenvalue $\lambda_j$ there necessarily exists $m_j$, linearly independent functions spanning $\mathrm{ker}\left(( L_K-\lambda_jI)^{m_j}\right)$. These (generalized) eigenfunctions which we denote $\{\varphi_{j,k}(x)\}_{k = 1}^{m_j}$ span $X_{\lambda_j}$.  Recall again that  $\varphi_{j,k}\in C[0,1]$ from point (2).

Next, since $(L_K-\lambda_jI)$ is Fredholm index zero then $\mathrm{dim}\mathrm{ker}\left((L_K^*-\lambda_j I)^{m_j}\right)=m_j$, where the dual operator $L_K^*$ acts on the dual space $C[0,1]^*$.  The Riesz Representation Theorem \cite[Theorem~36.6]{kolmogorov75} guarantees that for each $\theta_\alpha\in C[0,1]^*$ there exists an $\alpha\in BV[0,1]$ such that $\theta_\alpha$ can be expressed as the integral 
\[ 
    \theta_\alpha(v)=\int_0^1 v(x)\mathrm{d}\alpha(x). 
\]
The above representation of the dual space of $C[0,1]$ together with the identity $\theta_\alpha [L_K(v)]=L^*_K[\theta_\alpha](v)$ then implies that 
\begin{equation} \label{eq:LKstar}  
    L_K^*[\theta_\alpha](v)=\int_0^1 v(x)\mathrm{d}\alpha(x)-\int_0^1 \int_0^1 K(x,y)v(y)\mathrm{d}y\mathrm{d}\alpha(x), 
\end{equation}
where $K(x,y) = W(x,y)D_2(u^*(x),u^*(y))$.  Since $W(x,y)$ is non-negative and $|D_2(u^*(x),u^*(y))|$ is continuous it follows from nearly identical arguments to those in Corollary~\ref{cor:Qcont} that $x \mapsto \int_0^1 |K(x,y)v(y)|\mathrm{d}y$ is continuous as  function of $x$.  Then $\int_0^1 \int_0^1 |K(x,y)v(y)|\mathrm{d}y\mathrm{d}\alpha(x)$ exists and Fubini's Theorem \cite[Theorem~35.4]{kolmogorov75} allows one to switch the order of integration in \eqref{eq:LKstar}. Thus, \eqref{eq:LKstar} becomes 
\[ 
    L_K^*[\theta_\alpha](v)=\int_0^1 v(x)\mathrm{d}\alpha(x)-\int_0^1 v(y)  \int_0^1 K(y,x)\mathrm{d}\alpha(y)\mathrm{d}x  
\]
for an arbitrary $v \in C[0,1]$ .

Next, for $\theta_\alpha$ to be an element of $\mathrm{ker}(L_K^*-\lambda_jI)$ it must hold that
\begin{equation}\label{eq:LKkerneldef} 
    0 = (L_K^*-\lambda_jI)\theta_\alpha (v) = (1-\lambda_j) \int_0^1v(x)\mathrm{d}\alpha(x) -\int_0^1 v(y)  \int_0^1 K(y,x)\mathrm{d}\alpha(y)\mathrm{d}x, 
\end{equation}
for all $v\in C[0,1]$.  It therefore holds that  there exists a function $\psi_\alpha(x)$ such that $\theta_\alpha$ admits the representation 
\[ 
    \theta_\alpha v= \int_0^1 v(x)\psi_\alpha(x) \mathrm{d}x .
\]
Putting this together with \eqref{eq:LKkerneldef} implies that this function $\psi_\alpha$ must also satisfy 
\begin{equation} \label{eq:toshowpsicont}
    (1-\lambda_j) \psi_\alpha(x)-\int_0^1 K(y,x)\psi_\alpha(y)\mathrm{d}y =0. 
\end{equation}
Since $W(x,y)=W(y,x)$, we use Hypothesis \ref{hyp:Graphon}-2 and repeat the compactness argument in Lemma~\ref{lem:fredholm} culminating in (\ref{eq:psicont}) to conclude that  $\int_0^1 K(y,x)\psi_\alpha(y)\mathrm{d}y$ is a continuous function of $x$.  From (\ref{eq:toshowpsicont}) we then obtain that
\[ \psi_\alpha(x)=\frac{1}{1-\lambda_j}\int_0^1 K(y,x)\psi_\alpha(y)\mathrm{d}y\]
and therefore $\psi_\alpha(x)\in C[0,1]$.

We now construct the spectral projection $P$ in \eqref{eq:P}.  If $m_j=1$, we will show that there exist $\varphi_j\in \mathrm{ker}(L_K-\lambda_jI)$ and a $\psi_j\in C[0,1]$ such that 
\[ 
    P_{\lambda_j} v=\int_0^1 \varphi_j(x) v(y)\psi_j(y) \mathrm{d} y. 
\]
We require $P_{\lambda_j} \varphi_j=\varphi_j$ while $P_{\lambda_j} \left[(L_K-\lambda_j) w\right]=0$ for all $w\in C[0,1]$.  The first condition requires 
\begin{equation} \label{eq:int1}
    \int_0^1 \varphi_j(y)\psi_j(y)\mathrm{d}y =1,  
\end{equation}
while the second requires 
\[ 
    \int_0^1 \left( (1-\lambda_j)w(y)-\int_0^1 K(y,z) w(z)\mathrm{d}z \right) \psi_j(y)  \mathrm{d}y =0,
\]
for any $w\in C[0,1]$.  For the second condition, since $w$ and $\psi_j$ are continuous we can change the order of integration so that this condition assumes the form
\[ 
    \int_0^1 w(y) \left( (1-\lambda_j)\psi_j(y)-\int_0^1 K(z,y) \psi_j(z)\mathrm{d}z \right) \mathrm{d}y=0,
\]
which is satisfied for all $w\in C[0,1]$ if $\psi_j$ is chosen to be the unique (since we are assuming momentarily that the $\mathrm{dimker}(L_K^*-\lambda_jI)=1$), up to scalar multiplication, function that  satisfies  \eqref{eq:toshowpsicont}. The integral on the left hand side of \eqref{eq:int1} is always non-zero as otherwise the operator $v \mapsto \int_0^1 v(y)\psi_j(y)\mathrm{d}y \in C[0,1]^*$ would be trivial since we can decompose $C[0,1]=\mathrm{span}\{\varphi_j\}\oplus \mathrm{Rng}\{ L_k-\lambda_jI\}$. The condition \eqref{eq:int1} effectively selects a unique scalar multiple of the  function $\psi_j(x)$.

We now consider the case where the algebraic multiplicity of $\lambda_j$ exceeds one, i.e. the case where $m_j>1$. Since $\mathrm{dimker}\left((L_K^*-\lambda_j)^{m_j}\right)=m_j$, one can verify by a similar argument as to the case of $m_j=1$ that there exists $m_j$ linearly independent functions $\{\psi_{j,k}\}_{k = 1}^{m_j}\subset C[0,1]$ such that the bounded linear functionals
\[ 
    v \mapsto \int_0^1 v(y) \psi_{j,k}(y)\mathrm{d}y 
\]
span the kernel of $(L_K^*-\lambda_jI)^{m_j}$.

Since $X_{\lambda_j}$ is finite-dimensional, there exist $g_{j,k}\in C[0,1]^*$, $k = 1,\dots,m_j$, so that the spectral projection onto $X_{\lambda_j}$ will therefore take the form
\[ 
    P_{\lambda_j}v= \sum_{k=1}^{m_j} g_{j,k}[v] \varphi_{j,k}(x). 
\]
Furthermore, each $g_{j,k}\in C[0,1]^*$ can be represented as a linear combination of the $m_j$ functionals spanning $\mathrm{ker}\left((L_K^*-\lambda_jI)^{m_j}\right)$:
\[ 
    g_{j,k}[v]=\int_0^1 v(y) \left(\sum_{l=1}^{m_j} b_{k,l}\psi_{j,l}(y)\right)\mathrm{d}y,
\]
for some constants $b_{k,l}$.  Each $g_{j,k}$ must satisfy 
\begin{equation}
     v\in \mathrm{Rng}\left(\left(L_K-\lambda_jI\right)^{m_j}\right) \implies g_k[v]=0
\end{equation}
and $g_{j,k}[\varphi_{j,l}]=\delta_{kl}$ for all $1\leq k,l\leq m_j$.

It therefore remains to uniquely solve, for each $k$, the $m_j$ conditions $g_k[\varphi_{j,l}]=\delta_{kl}$ for the $m_j$ coefficients $b_{k,l}$.  Solvability of this system of equations relies on the invertibility of a Gram-like matrix which follows from the linear independence of the $\varphi_{k,j}(x)$ and $\psi_{j,l}(y)$.  With this solution we can define $\tilde{\psi}_{k,l}(y)=\sum_{l=1}^{m_j} b_{k,l}\psi_{j,l}(y)$ and then drop the tildes so that the spectral projection assumes the form written in \eqref{eq:P}.

The spectral projection formula in \eqref{eq:P} has been derived for $v\in C[0,1]$, but the same operator describes the projection for $v\in X_n$, for any $n \geq 1$.  To see this, let $v\in X_n$ and consider $w=(L_K-\lambda_j)v$. Then we have 
\begin{equation} \begin{split}
    g_{j,k}[w]&=\int_0^1[(L_K-\lambda_j)v] \psi_{j,l}(y)\mathrm{d}y \\
    &= \int_0^1 (1-\lambda_j)v(y)\psi_{j,l}(y)\mathrm{d}y- \int_0^1 \int_0^1 K(y,z)v(z)\mathrm{d}z \psi_{j,l}(y)\mathrm{d}y\\
    &= \int_0^1 (1-\lambda_j)v(y)\psi_{j,l}(y)\mathrm{d}y- \int_0^1 v(y)  \int_0^1 K(z,y)\psi_{j,l}(z)\mathrm{d}z \mathrm{d}y \\
    &= \int_0^1 v(y)  \left((1-\lambda_j)\psi_{j,l}(y) -\int_0^1 K(z,y)\psi_{j,l}(z)\mathrm{d}z \right)\mathrm{d}y, 
\end{split} \end{equation}
where we have used Fubini's theorem to switch the order of integration in the second to last line, owing to the fact that $y \mapsto \int_0^1 K(y,z)v(z)\mathrm{d}z$ is continuous in $y$ for $v\in X_n$.  The final line implies that $g_{j,k}[(L_K-\lambda_j)v]=0$ for all $v\in X_n$.  A similar argument applies to generalized eigenfunctions. Since $X_{\lambda_j}\subset C[0,1]$ the range of $P_{\lambda_j}$ is again $X_{\lambda_j}$ yielding the spectral projection. This concludes the proof of point (3).

(4)\ We now return to the problem of inverting $DF(u^*)$ and obtaining the formula \eqref{eq:Minvgen}. Let $v \in X_n$ and set $w = DF(u^*)v$. Our goal is to obtain an expression for $v$ in terms of $w$. 

We begin by dividing $DF(u^*)v=w$ by the non-zero function $-Q(x)$ to reduce the problem of inverting $DF(u^*)$ to that of solving 
\[ 
    (I-T_K) v=-\frac{w}{Q(\cdot)}. 
\]
Recall the decomposition of the space $X_n$ into a finite sum of invariant subspaces.  Write
\begin{equation} 
\begin{split}
    v&=v_1+v_2+\dots v_J+\tilde{v}  \\
    -\frac{w}{Q(\cdot)}&=w_1+w_2+\dots +w_J+\tilde{w},
\end{split}
\end{equation}
where $v_j=P_{\lambda_j}v\in X_{\lambda_j}$, $w_j=P_{\lambda_j}(-\frac{w}{Q}) \in X_{\lambda_j}$, $\tilde{v}=\tilde{P}v\in \tilde{X}_n$ and $\tilde{w}\in \tilde{X_n}$.  Owing to invariance of these subspaces, we can invert $I-T_K$ on all of $X_n$ by inverting the operator restricted to each subspace.  

Consider first $X_{\lambda_j}$.  We must solve  $(I-T_K)v_j=w_j$.  But $X_{\lambda_j}$ is simply the span of the eigenfunctions and generalized eigenfunctions of $T_K$ associated to the eigenvalue $\lambda_j$.  Therefore $(I-T_K)v_j=(1-\lambda_j)v_j+N_{\lambda_j}[v_j]$ where the linear operator  $N_{\lambda_j}:X_{\lambda_j}\to X_{\lambda_j}$ is nilpotent.  This means that we can solve $(I-T_K)v_j=w_j$ for any $w_j\in X_{\lambda_j}$ by reduction to a finite-dimensional problem.  In particular we can write $v_j=\sum s_{j,k} \varphi_{j,k}(x)$ and $w_j=\sum r_{j,k}\varphi_{j,x}(x)$.  Note that $r_{j,k}=g_{j,k}\left[-\frac{w}{Q(\cdot)}\right]$.  Invertibility within this subspace implies that we can write each $s_{j,k}$ as a linear combination of the $r_{j,k}$. This implies that there exist coefficients $c_{j,k,l}$ such that 
\[ 
    v_j=\sum_{k,l=1}^{m_j} \varphi_{j,k}(x)\int_0^1 c_{j,k,l} \psi_{j,l}(y)\left(\frac{-w(y)}{Q(y)}\right)\mathrm{d}y. 
\]
Repeating this procedure over all $v_j$ we obtain the second expression in (\ref{eq:Minvgen}).

It remains to invert $(I-T_K)$ restricted to the invariant subspace $\tilde{X}_n$.  The spectrum of $T_K$ restricted to this subspace lies strictly to the left of the line $\mathrm{Re}(\lambda)=1$ by construction.  By point (2), the spectrum of $T_K$ is independent of the space $X_n$ and therefore the spectral radius of $T_K$, restricted to $\tilde{X}_n$ is independent of $n$.  Thus, there exists a constant $\xi>0$, independent of $n$ such that $\left.\sigma\left(T_K\right)\right|_{\tilde{X}_n}$ is contained inside a ball centered at $-\xi$ with radius less than $\xi+1$.  
 Then the spectral radius of the rescaled operator $\frac{\xi+T_K}{1+\xi}$, restricted to $\tilde{X}_n$ is strictly less than one and therefore by re-arranging 
\begin{equation}
    (I-T_K)=(I+\xi)-(\xi+T_K)=(1+\xi)\left(I-\frac{T_K+\xi}{1+\xi}\right), \label{eq:Nuemannrescaling}
\end{equation}
then we can solve $(1-T_K)\tilde{v}=\tilde{w}$ using Neumann series.  Writing 
\[ \tilde{w}=\tilde{P}\left(\frac{-w}{Q(\cdot)}\right) \]
the Neumann series provides the first term in the summation on the right-hand-side of \eqref{eq:Minvgen}. Thus, we have proved point (4) of the lemma. 

\section{Details for the Kuramoto Model}\label{sec:KuramotoApp}

In Section~\ref{sec:Kuramoto} we commented that to apply Theorem~\ref{thm:existence} to twisted states in the Kuramoto model one is required to quotient out the translational symmetry of the model. We proposed that the analysis undertaken in this work can be generalized to the Kuramoto model by replacing the function space $X_n$ with a subspace $Y_n$ of mean-zero functions. We now provide the necessary details that substantiate our statements.

First we show that $DF(u^*)$ is invertible on $Y_n$. Linearizing the Kuramoto model about a twisted state $u^*(x) = m (x - 1/2)$ with any $m \in \mathbb{Z}$ yields
\begin{equation}
    DF(u^*)v=-2\pi c_m v +2\pi \int_0^1 W(x,y) \cos(2\pi m(y-x)) v(y)\mathrm{d}y. 
\end{equation}
In Section~\ref{sec:Kuramoto} we have assumed that all elements of the spectrum of $DF(u^*)$ are negative aside from the isolated eigenvalue of algebraic multiplicity one at $\lambda=0$ whose eigenspace is spanned by constant functions. This zero eigenvalue comes exactly from the translational symmetry. To invert $DF(u^*)$ we have $DF(u^*)v=w$, requiring one to solve
\begin{equation} 
    (I-T_K)v=-\frac{w}{2\pi c_m},
\end{equation}
where in the case of the Kuramoto model we have
\begin{equation}
    T_K=\frac{1}{c_m} \int_0^1 W(x,y)\cos(2\pi m (y-x))v(y)\mathrm{d} y, 
\end{equation}
which is a compact as an operator on $Y_n$. For a spectrally stable twisted state solution of the Kuramoto model, the spectrum of $T_K$ is real and assumes values strictly less than one with the exception an isolated eigenvalue at $\lambda=1$ with constant eigenfunction. The spectral projection onto this eigenfunction is simply $Pv=\int_0^1 v(y)\mathrm{d}y$. Then, 
\begin{equation}
    DF(u^*)v=DF(u^*) P v+DF(u^*)(I-P)v=DF(u^*)(I-P)v.
\end{equation}  
Since the range of $I - P$ is exactly $Y_n$ and the spectrum of $DF(u^*)$ on $Y_n$ is bounded away from zero, it therefore holds that $DF(u^*)$ is invertible on $Y_n$ and can be expressed via Neumann series after perhaps shifting and rescaling the operator in a analogous manner as (\ref{eq:Nuemannrescaling}).  

With the invertibility of $DF(u^*)$ on $Y_n$, we now recall that $\mathcal{T}_n$ is defined by 
\[ 
    \mathcal{T}_n[u]=u-DF(u^*)^{-1} F_n(u). 
\]
We claim that $\mathcal{T}_n$ maps $Y_n$ back into itself. To verify this, we must show that $F_n(u)\in Y_n$ for any $u\in Y_n$. This fact follows from the following calculation 
\begin{equation}
\begin{split}
    \int_0^1 F_n(u)\mathrm{d}x &= \int_0^1\int_0^1 W_n(x,y) \sin (2\pi(u(y)-u(x))) \mathrm{d}y\mathrm{d}x \\
    &=  -\int_0^1\int_0^1 W_n(x,y) \sin (2\pi(u(x)-u(y))) \mathrm{d}y \mathrm{d}x \\
    &=  -\int_0^1\int_0^1 W_n(y,x) \sin (2\pi(u(y)-u(x))) \mathrm{d}y\mathrm{d}x \\
    &=  -\int_0^1\int_0^1 W_n(x,y) \sin (2\pi(u(y)-u(x))) \mathrm{d}y\mathrm{d}x \\ 
    &= -\int_0^1 F_n(u) \mathrm{d}x. 
\end{split}
\end{equation}
Therefore, $\int_0^1 F_n(u)\mathrm{d}x = 0$, meaning $F_n(u)$ is mean-zero and is in turn an element of $Y_n$ for any $u\in Y_n$. This makes the mapping $\mathcal{T}_n:Y_n\to Y_n $ well defined and the remainder of the proof proceeds as in the proof of Theorem~\ref{thm:existence}.

\bibliographystyle{abbrv}
\bibliography{references.bib}

\end{document}